\documentclass[11pt,reqno]{amsart}
\usepackage{amssymb, latexsym} 

\makeatletter
\DeclareOldFontCommand{\rm}{\normalfont\rmfamily}{\mathrm}
\DeclareOldFontCommand{\sf}{\normalfont\sffamily}{\mathsf}
\DeclareOldFontCommand{\tt}{\normalfont\ttfamily}{\mathtt}
\DeclareOldFontCommand{\bf}{\normalfont\bfseries}{\mathbf}
\DeclareOldFontCommand{\it}{\normalfont\itshape}{\mathit}
\DeclareOldFontCommand{\sl}{\normalfont\slshape}{\@nomath\sl}
\DeclareOldFontCommand{\sc}{\normalfont\scshape}{\@nomath\sc}
\makeatother

\def\D{{\mathcal D}}
\def\PP{{\mathcal P}}
\def\B{{\mathcal B}}
\def\PG{{\rm PG}}

\def\GammaL{\Gamma{\rm L}}
\def\PGammaL{{\rm P}\Gamma{\rm L}}

\def\ff{\overline{G}}
\def\GL{{\rm GL}}

\def\la{\langle}
\def\ra{\rangle}
\def\SL{{\rm SL}}
\def\PSL{{\rm PSL}}
\def\PGL{{\rm PGL}}
\def\PGaL{{\rm P}\Gamma {\rm L}}
\def\GaL{\Gamma {\rm L}}
\def\PSigmaL{{\rm P}\Sigma {\rm L}}
\def\SigmaL{\Sigma {\rm L}}
\def\PSU{{\rm PSU}}
\def\SU{{\rm SU}}
\def\PGU{{\rm PGU}}
\def\GU{{\rm GU}}
\def\GammaU{\Gamma{\rm U}}

\def\GL{{\rm GL}}

\def\Sym{{\rm Sym}}

\def\Alt{{\rm Alt}}
\def\Aut{{\rm Aut}}

\def\la{\langle}
\def\ra{\rangle}

\usepackage[left=2.5cm,right=2.5cm,
    top=2cm,bottom=2cm,bindingoffset=0cm]{geometry}

\usepackage{mathrsfs}

\usepackage{geometry}
\usepackage{tikz}
\usepackage{longtable}
\usepackage{mathtools} 
\usepackage{enumitem} 
\usepackage{listings} 
\usepackage[justification=centering]{caption} 
\usepackage[noadjust]{cite} 
\usepackage{graphicx,float,wrapfig,avant,colortbl,upgreek,array,tikz}
\usetikzlibrary{positioning}
\usepackage{subcaption}
\usepackage{changepage}

\usepackage{cellspace}
\definecolor{thistle}{rgb}{0.95,0.9,1}
 \definecolor{darkgreen}{rgb}{0,0.5,0}

\definecolor{RawSienna}{cmyk}{0,1,0,0}

\definecolor{col1}{HTML}{e60049}
\definecolor{col2}{HTML}{0bb4ff}
\definecolor{col3}{HTML}{50e991}
\definecolor{col4}{HTML}{e6d800}
\definecolor{col5}{HTML}{9b19f5}
\definecolor{col6}{HTML}{ffa300}
\definecolor{col7}{HTML}{dc0ab4}
\definecolor{col8}{HTML}{b3d4ff}
\definecolor{col9}{HTML}{00bfa0}

\usetikzlibrary{calc}
\usetikzlibrary{decorations.pathmorphing}

\numberwithin{equation}{section} 
\numberwithin{figure}{section} 
\numberwithin{table}{section} 

\title[Imprimitive rank $3$ partial linear spaces]{Proper partial linear spaces affording imprimitive rank 3 automorphism groups}

\author[A. A. Baykalov]{Anton A. Baykalov}
\address{School of Mathematical and Statistical Sciences \\
  University of Galway \\
  Galway,   H91 TK33,   Ireland} 
\email{a.a.baykalov@gmail.com}

\author[A. Devillers]{ Alice Devillers}
\author[C. E. Praeger]{Cheryl E. Praeger}
\address{Centre for the Mathematics of Symmetry and Computation\\
School of Physics, Mathematics and Computing\\
The University of Western Australia\\
Perth, WA 6009, Australia} 
\email{alice.devillers@uwa.edu.au, cheryl.praeger@uwa.edu.au}
\thanks{This work forms part of the Australian Research Council Discovery Grant project  
DP200100080. On final stages of preparation of this publication, Baykalov was supported by Taighde \'Eireann -- Research Ireland under Grant number 22/FFP-P/11449.}

\date{}

\newtheorem{Th}{Theorem}
\newtheorem{Lem}[Th]{Lemma}
\newtheorem{Pro}[Th]{Proposition}

\newtheorem{Cor}[Th]{Corollary}
\newtheorem{construction}[Th]{Construction}
\newtheorem{Hyp}[Th]{Hypothesis}

\theoremstyle{definition}
\newtheorem{Def}[Th]{Definition}
\newtheorem{Rem}[Th]{Remark}
\newtheorem{Remark}[Th]{Remark}
\newtheorem{Not}[Th]{Notation}
\newtheorem{Ex}[Th]{Example} 
\numberwithin{Th}{section}

\renewcommand{\leq}{\leqslant}
\renewcommand{\geq}{\geqslant}

 \newcommand{\omr}[1][r]{\langle \omega^{#1} \rangle}
 \newcommand{\diag}{\mathrm{diag}}

\begin{document}
\begin{abstract}
A partial linear space is a point--line incidence structure such that each line is incident with at least two points and  each pair of  points is incident with at most one line. It is said to be proper if there exists at least one non-collinear point pair, and at least one line incident with more than two points. The highest degree of symmetry for a proper partial linear space occurs when the automorphism group $G$ is transitive on ordered pairs of collinear points, and on ordered pairs of non-collinear points, that is to say, $G$ is a  transitive rank $3$ group on the points. While the primitive rank 3 partial linear spaces are essentially classified, we present the first substantial classification of a family of imprimitive rank $3$ examples. 
 We classify all imprimitive rank $3$ proper partial linear spaces such that the rank $3$ group is semiprimitive. In particular,  this includes all partial linear spaces with a rank 3 imprimitive automorphism group that is innately transitive or quasiprimitive.   We construct several infinite families of examples and ten individual examples. 
The examples  in the infinite families  admit a rank $3$ action of a linear or unitary group, and to our knowledge most of these examples have not appeared before in the literature.


\end{abstract}
\maketitle

\section{Introduction}

A \emph{partial linear space} $\D=(\PP,\B)$  is a structure with two types of objects, usually called points (elements of $\PP$)  and lines (elements of $\B$), and an incidence relation between  $\PP$ and  $\B$ such that each line is incident with at least two points and  each pair of  points is incident with at most one line. In this paper, for a partial linear space $\D=(\PP,\B)$, the sets $\PP$ and $\B$ are finite,  each line is incident with a constant number $\ell \geq 2$ of points, and each point is incident with a constant number of lines.  If two points are incident with a common line, we say the points are \emph{collinear}. If each pair of points is collinear, then $\D$ is a linear space; if $\ell=2$, then $\D$ is a graph.  We say that a partial linear space is \emph{proper} if it is neither a linear space nor a graph.
The maximum degree of transitivity that can be achieved by the automorphism group of a  proper partial linear space, acting on pairs of points, is  transitivity on (ordered) pairs of collinear points and on (ordered) pairs of non-collinear points. In other words a group achieving this degree of transitivity has a rank $3$ action on points.  Here the \emph{rank} of a transitive permutation group $G$ on a set $\Omega$ (or of its action) is the number of orbits in $\Omega$ of a point stabiliser of $G$.

The primitive rank 3 proper partial linear spaces are essentially classified: of the three possible types of primitive rank $3$ actions -- almost simple type, grid type, and affine type -- 
partial linear space examples for the  first two types were classified by Devillers \cite{D05,D08}, and those of affine type were recently treated in \cite{BDFP}, giving a satisfactory classification except for a few `hopeless' cases.
On the other hand, the problem of classifying the imprimitive rank $3$ partial linear spaces is largely untouched. 
 Note that, to get a {\it connected} partial linear space (such that its collinearity graph is connected) with an imprimitive rank 3 automorphism group, each pair of collinear points must lie in distinct blocks of imprimitivity, so the number of imprimitivity blocks must be at least the line-size $\ell$. In the simplest case, when $\ell=3$ and the group has three blocks of imprimitivity,  the partial linear space is called a transversal design with $3$ blocks, and this is the only case that is classified so far:  Devillers and Hall \cite{DH06} showed that these spaces arise from multiplication tables of elementary abelian groups. There is little hope of a complete classification  of rank $3$ partial linear spaces as imprimitive rank 3 groups are so `wild'. 

Nevertheless the classification of quasiprimitive and properly innately transitive rank $3$ groups, in \cite{DGLPP} and \cite{R3PIT} respectively, laid the groundwork for a classification of the innately transitive rank 3 partial linear spaces which is the main goal of this paper. 
Here a transitive permutation group $G$ is {\it quasiprimitive} if every normal subgroup of $G$ is transitive, and $G$ is {\it innately transitive} if $G$ has at least one transitive minimal normal subgroup  (called a {\it plinth}). Naturally, primitivity implies quasiprimitivity which implies innate transitivity. 
Recently, in \cite{semiprR3}, Huang, Li and Zhu made progress on classifying a larger family of rank 3 groups, namely the rank 3 semiprimitive groups (a transitive permutation group is {\it semiprimitive} if each of its normal subgroups is either transitive or semiregular). Each  imprimitive  rank $3$ group $G$ leaves invariant a unique nontrivial partition $\Sigma$ of the point set, and the $G$-action induced on $\Sigma$ is  $2$-transitive and hence of almost simple or affine type (see \cite[Corollary after Lemma 4]{Higman}).  In \cite[Theorem 1]{semiprR3}, the authors classify completely the rank 3 semiprimitive groups for which the induced action on $\Sigma$ is almost simple: there is one infinite family of examples and two exceptional examples which are not innately transitive. The infinite family arises by relaxing restrictions on the size of the blocks of imprimitivity in $\Sigma$ for the family of examples \cite[Construction 4.2]{R3PIT} of rank 3 innately transitive groups. Because of these links with innately transitive groups, it was  natural to extend our classification to the partial linear spaces affording such rank 3 semiprimitive groups. The rank 3 groups we focus on in this paper therefore satisfy the following hypothesis.  (During the refereeing process, we learned from an anonymous referee that the classification of rank 3 semiprimitive groups had been completed in the preprint \cite{Li2025}, identifying those for which the group $G^{\Sigma}$ is affine. We deal with these additional possibilities in Subsection~\ref{s:extra}, see Theorem~\ref{t:all}. We thank the referee for drawing our attention to \cite{Li2025}.)

\begin{Hyp}\label{hyp1}
Let $G\leq \Sym(\Omega)$ be a transitive imprimitive permutation group of rank $3$ that is either innately transitive or semiprimitive such that $G^{\Sigma}$ is almost simple, where $\Sigma$ is the unique nontrivial system of imprimitivity. 
\end{Hyp}

 We give a summary of permutation groups satisfying Hypothesis \ref{hyp1} in Table~\ref{t:qpinsprk3}; the information is gathered from \cite{DGLPP,R3PIT,semiprR3}.
It follows from \cite[Corollary after Lemma 4]{Higman} that an imprimitive rank 3  group preserves a unique system of imprimitivity $\Sigma,$ and we let $r=|\sigma|$ for $\sigma \in \Sigma.$ Since  $G^{\Sigma}$ is an almost simple group, we keep track of its socle which we denote by $M$. Notice that, if the imprimitive rank 3 group $G$ is innately transitive, then $G^{\Sigma}$ is always almost simple and the unique plinth of $G$ is isomorphic to $M$, see \cite[Lemma 2.5]{R3PIT}.  The actions on $\Omega$ for the first three lines of Table \ref{t:qpinsprk3} (the infinite series) are given by Constructions \ref{con:psl} and \ref{con:psu} for groups with linear and unitary $M$ respectively. For the remaining cases, where the faithful action of degree $|\Omega|=|\Sigma|\cdot r$ is not unique up to permutational isomorphism, we specify the stabiliser of a point $\alpha \in \Omega$ (up to conjugation). In the last column we specify the type of $G$: ``$\mathrm{qp}$'', ``$\mathrm{it}$'' or ``$\mathrm{sp}$''. Here ``$\mathrm{qp}$'' means quasiprimitive,  ``$\mathrm{it}$'' means innately transitive but not quasiprimitive, and ``$\mathrm{sp}$'' means semiprimitive but not innately transitive. If any of the types can occur in the series, then we use ``$\mathrm{qp/it/sp}$''; when there is a list of the groups in the column ``$G$'' we state the types in the respective order.  

Let us briefly introduce some notation that is used in Table \ref{t:qpinsprk3} and also later in the paper. We fix a primitive element $\omega$ in $\mathbb{F}_q$ or $\mathbb{F}_{q^2}$ for $M$ equal to $\PSL_n(q)$ or $\PSU_n(q)$, respectively.  For a vector space $V$ over $\mathbb{F}_{p^{a}}$ with a basis $\beta=\{v_1, \ldots, v_n\}$, we denote by $\phi_{\beta}$ (or just $\phi$ if the basis is understood) the $\mathbb{F}_{p^a}$-semilinear map on $V$ defined via
\begin{equation}\label{def:phi}
\left(\sum_{i=1}^n \alpha_i v_i\right) \phi_{\beta} = \left(\sum_{i=1}^n \alpha_i^p v_i\right).
\end{equation}
For a prime $p$ and an integer $a$, a prime divisor $r$ of $p^a -1$ is  a {\it primitive
prime divisor} if $r$ does not divide $p^m - 1$ for any positive integer $m < a$. This is equivalent to the statement that the multiplicative order $o_r(p)$ of $p$ in $\mathbb{Z}_r$ is $a.$ We denote the greatest common divisor and the least common multiple of integers $n$ and $m$ by $(n,m)$ and $\mathrm{lcm}\{n,m\}$, respectively. It follows from results in \cite{DGLPP,R3PIT,semiprR3} that each $G, \Sigma, r, M$ in Hypothesis~\ref{hyp1} satisfies one of the lines on Table~\ref{t:qpinsprk3}.  

\begin{table}[h]
\resizebox{\columnwidth}{!}{
\begin{tabular}{|c| p{1.2cm} | p{4.5cm}|p{7cm}|p{5.4cm}|c|}
\hline
 $M$                                  & $|\Sigma|$              & $r$    &  $G$    &Conditions on $G$ & type          \\ \hline

 $\PSL_n(q)$                                  & $\frac{q^n-1}{q-1}$ & \begin{minipage}[t]{5cm}prime  such that $r | {(q - 1})$,\\ $o_r(p)=r-1$ and\\  $(n,r) \ne (2,2)$
\end{minipage}  & $\langle\omega^r I\rangle \, \SL_n(q)/\langle\omega^r I\rangle\leq G\leq \GammaL_n(q)/\langle\omega^r I\rangle$ &  \begin{minipage}[t]{5.4cm} 
  $|G^\Sigma/(G^\Sigma\cap\PGL_n(q))|=a/j$ \\with $(j,r-1)=1$\end{minipage} & $\mathrm{qp/it/sp}$ \\ 
$\PSL_2(q)$                                  & $q+1$ & $2$   & $\langle\omega^2 I\rangle \, \SL_2(q)/\langle\omega^2 I\rangle\leq G\leq \GammaL_2(q)/\langle\omega^2 I\rangle$ &  \begin{minipage}[t]{5.4cm} $q\geq 5$, $q$ is odd,   $G^\Sigma\not\leq \PSigmaL_2(q)$\end{minipage} & $\mathrm{qp/it/sp}$\\ 
$\PSU_3(q)$                                  & $q^3+1$       &\begin{minipage}[t]{5cm}odd prime such that   \\ $r | q-1$,  $o_r(p)=r-1$ 
\end{minipage}& $\langle\omega I\rangle  \, \SU_3(q)/\langle\omega^r I\rangle \leq G\leq \GammaU_3(q)/\langle\omega^r I\rangle$   &   \begin{minipage}[t]{5.4cm} $|G^\Sigma/(G^\Sigma\cap\PGU_3(q))|=2a/j$\\ with $(j,r-1)=1$\end{minipage}& $\mathrm{it}$\\ 
$\PSL_3(2)$&$7$&$2$&$\PSL_3(2),$ $C_2\times \PSL_3(2)$ & & $\mathrm{qp,it}$\\

 $\rm{M}_{11}$                                    &$11$   &$2$    & $\rm{M}_{11}$, $C_2\times \rm{M}_{11}$ &       & $\mathrm{qp,it}$  \\     
  $\PSL_3(4)$                                    &$21$   &$6$    & $\PGL_3(4)$, $\PGammaL_3(4)$ &       & $\mathrm{qp,qp}$  \\ 
  $\PSL_3(5)$                                    &$31$   &$5$    & $\PSL_3(5)$ &       & $\mathrm{qp}$  \\ 
   $\PSL_5(2)$                                    &$31$   &$8$    & $\PSL_5(2)$ &       & $\mathrm{qp}$  \\ 
 $\PSL_3(8)$                                    &$73$   &$28$    & $\PGammaL_3(8)$ &       & $\mathrm{qp}$  \\ 
   $\PSL_3(3)$                                    &$13$   &$3$    & $\PSL_3(3)$ &       & $\mathrm{qp}$  \\ 
      $\Alt(6)$                                    &$6$   &$3$    & $3.\Sym(6)$ & $G_{\alpha}=\Sym(5)$      & $\mathrm{sp}$  \\
         $\rm{M}_{12}$                                    &$12$   &$2$    & $2.\rm{M}_{12}$ &      $G_{\alpha}=\rm{M}_{11}$ & $\mathrm{sp}$  \\ 
\hline

\end{tabular}}
\caption{$G$ satisfying Hypothesis \ref{hyp1},\\ $q=p^a$ with $p$ prime \label{t:qpinsprk3}} 
\end{table}

The following two results from \cite{D05} give us a way to construct partial linear spaces from rank 3 permutation groups which guarantees that we will find all the partial linear spaces admitting the rank 3 group as a subgroup of automorphisms. 

\begin{Th}[{\cite[Theorem 2.3]{D05}}] \label{th:AD1}
    Let $G$ be a rank $3$ permutation group on a set $\mathcal{P}$ and let $\alpha$ be a
point in $\mathcal{P}$. If $B$ is a non-trivial block of imprimitivity for the action of the point stabiliser $G_{\alpha}$ on one of its orbits, and if the setwise stabiliser of $B \cup \{\alpha\}$ in $G$ is transitive on $B \cup \{\alpha\}$, then the pair $(\mathcal{P},\mathcal{L})$, where $\mathcal{L} = (B \cup \{\alpha\})^G$, forms a proper partial linear space.
\end{Th}

\begin{Th}[{\cite[Theorem 2.2]{D05}}] \label{th:AD2}
   Let $G$ be a rank $3$ permutation group on a set $\mathcal{P}$. Suppose that $G$
is an automorphism group of a proper partial linear space $\mathcal{D} = (\mathcal{P},\mathcal{L})$. Then one of the following holds:
\begin{enumerate}[label=\normalfont (\arabic*)]
    \item  $G$ is imprimitive on $\mathcal{P}$, leaving invariant a nontrivial partition $\Sigma$ of $\mathcal{P}$, and $\mathcal{L}$ is a set of pairwise disjoint, equal-sized lines, namely the parts of the partition $\Sigma$; or
\item for any point $\alpha \in  \mathcal{P}$ and any line $L \in \mathcal{L}$ through $\alpha$, the set $L \backslash \{\alpha\}$  is a block of imprimitivity for $G_{\alpha}$ in one of its orbits, and the setwise stabiliser $G_L$ is transitive on $L$.
\end{enumerate}

\end{Th}

 If $G \leq \Sym(\Omega)$ of rank $3$ preserves a nontrivial imprimitivity system $\Sigma$, then the three orbits of $G_{\alpha}$ are $\{\alpha\},$ $\sigma \backslash \{\alpha\}$ and $\Omega \backslash \sigma$, where $\sigma \in \Sigma$ contains $\alpha$.  It is clear that if  either $\mathcal{D}$ is as in Theorem \ref{th:AD2}(1), or $\mathcal{D}$ is a partial linear space obtained using Theorem \ref{th:AD1} with $B \subseteq \sigma\backslash \{\alpha\},$ then $\mathcal{D}$ is not connected, since lines lying in distinct blocks of $\Sigma$ do not intersect. On the other hand, if $B \subseteq \Omega \backslash \sigma$, then the partial linear space $\mathcal{D}$ constructed via Theorem \ref{th:AD1} is connected since $\alpha$ is collinear with all points in $\Omega \backslash \sigma$ and, in turn, these points are connected to each point in $\sigma.$

Our main result is the following.

\begin{Th}\label{th:main}

    Let $\mathcal{D}$ be a connected proper partial linear space admitting a group $G$ of automorphisms satisfying Hypothesis $\ref{hyp1}$. 
    \begin{enumerate}[label=\normalfont (\arabic*)]
        \item If $G$ is as in the first three lines of Table $\ref{t:qpinsprk3}$, then $\mathcal{D}$ is isomorphic to one of the partial linear spaces in Table $\ref{tabLU}$. Here $Z$ is the subgroup of scalar matrices in $\GL_n(q).$
        \item If $G$ is as in the  lines $4$--$12$ of Table $\ref{t:qpinsprk3}$, then $\mathcal{D}$ is isomorphic to one of the partial linear spaces in Table $\ref{tabPLS}$.
    \end{enumerate}
\end{Th}

\begin{table}[h]
\begin{tabular}{|c|c|c|Sc|c|}
\hline
$\mathcal{D}$ &   Reference  & $r$  & Possible $G$    & Type of $G$   \\ \hline
 $\mathrm{AG}^*(n,4)$, $n \geq 2$ &  Definition \ref{AGdef} & 3   &\begin{minipage}[t]{5cm}
  $\langle \SL_n(4), \diag(1, \ldots, 1, \omega) \phi \rangle,$ \\
  $ \SL_n(4) \rtimes \langle \phi \rangle, \GammaL_n(4)$
\end{minipage}   & $\mathrm{qp/it/sp}$  \\ \hline
  $\Delta(n,3)$, $n\geq 3$  &  Definition \ref{Deltadef} & 2  & $\SL_n(3), \GL_n(3)$ & $\mathrm{qp/it/sp}$  \\ \hline
  $\Delta(n,4)$, $n\geq 2$  &  Definition \ref{Deltadef} &3  &  \begin{minipage}[t]{5cm}
  $\langle \SL_n(4), \diag(1, \ldots, 1, \omega) \phi \rangle,$ \\
  $ \SL_n(4) \rtimes \langle \phi \rangle, \GammaL_n(4)$
\end{minipage} & $\mathrm{qp/it/sp}$  \\ \hline
   $\mathrm{LSub}(n,16,4,5)$, $n\geq 2$  &  Definition \ref{con:subfield} & 5   & \begin{minipage}[t]{5cm} $G$ of rank 3 such that \\ $G \geq \langle \omega^5 I \rangle\,\SL_n(16)/\langle \omega^5 I \rangle$ and\\ 
   $ G \leq \GammaL_n(16)/\langle \omega^5 I \rangle$\end{minipage} & $\mathrm{qp/it/sp}$  \\ \hline
   $\mathrm{LSub}(2,81,9,5)$  &  Definition \ref{con:subfield} & 5   & $(Z \, \SL_2(3^4)\rtimes \langle \phi \rangle)/\langle \omega^5 I \rangle$ & $\mathrm{it}$  \\ \hline
   $\mathrm{LSub}(2,25,5,3)$  &  Definition \ref{con:subfield} & 3  & $(Z \, \SL_2(5^2)\rtimes \langle \phi \rangle)/\langle \omega^3 I \rangle$ & $\mathrm{it}$  \\ \hline
   $\mathrm{DLSub}(9,3,2,1)$  &  Definition \ref{DLSdef} & 2  & $\langle \langle \omega^2 I \rangle \, \SL_2(3^2),  \phi \, \diag(1,\omega)  \rangle/\langle \omega^2 I \rangle$ & $\mathrm{qp}$  \\ \hline
    $\mathrm{USub}(4,2,3)$  &  Definition \ref{con:subfieldSU} & 3   & $\GammaU_3(4)/\langle \omega^3 I \rangle$ & $\mathrm{it}$  \\ \hline
     $\mathrm{USub}(16,4,5)$  &  Definition \ref{con:subfieldSU} & 5   & $\GammaU_3(16)/\langle \omega^5 I \rangle$ & $\mathrm{it}$  \\ \hline
    $ \mathrm{AGU}^*(4)$  &  Definition \ref{con:AGSU} & 3  & $\GammaU_3(4)/\langle \omega^3 I \rangle$ & $\mathrm{it}$  \\ \hline

\end{tabular}
\caption{Partial linear spaces arising from linear and unitary groups}\label{tabLU}
\end{table}

\begin{table}[h]
\begin{tabular}{|c|c|c|c|ccc|c|c|}
\hline
$M$               & $|\Sigma|$ & $r$ &$\mathcal{D}$ &  \multicolumn{1}{c|}{n. of lines}    & \multicolumn{1}{c|}{size of lines}     & Possible $G$   & type                     \\ \hline
$\PSL_3(3)$        & 13  & 3 & $\mathcal{D}_1$  & \multicolumn{1}{c|}{234}            & \multicolumn{1}{c|}{3}                       & $\PSL_3(3)$        & $\mathrm{qp}$                    \\ \hline
                  &     &  &  $\mathcal{D}_2$  & \multicolumn{1}{c|}{117}            & \multicolumn{1}{c|}{4}                      &$\PSL_3(3)$ & $\mathrm{qp}$                      \\ \hline
       & 7   & 2 & $\mathcal{D}_3$  & \multicolumn{1}{c|}{14}             & \multicolumn{1}{c|}{4}                       & $\PSL_3(2)$, $C_2 \times \PSL_3(2)$ &    $\mathrm{qp}$, $\mathrm{it}$                      \\ \hline
                  &     &  & $\mathcal{D}_4$   & \multicolumn{1}{c|}{28}             & \multicolumn{1}{c|}{3}                      & $\PSL_3(2)$ & $\mathrm{qp}$                     \\ \hline
$\PSL_3(8)$  & 73  & 28 & $\mathcal{D}_5$ & \multicolumn{1}{c|}{686784}         & \multicolumn{1}{c|}{3}                       & $\PGammaL_3(8)$   & $\mathrm{qp}$                          \\ \hline
                  &     &  & $\mathcal{D}_6$   & \multicolumn{1}{c|}{98112}          & \multicolumn{1}{c|}{7}                       & $\PGammaL_3(8)$     & $\mathrm{qp}$                        \\ \hline
$\PSL_5(2)$        & 31  & 8 & $\mathcal{D}_7$  & \multicolumn{1}{c|}{248}            & \multicolumn{1}{c|}{16}                     &$\PSL_5(2)$ & $\mathrm{qp}$                    \\ \hline
$\PSL_3(5)$        & 31  & 5 & $\mathcal{D}_8$  & \multicolumn{1}{c|}{775}            & \multicolumn{1}{c|}{6}                       &$\PSL_3(5)$    &$\mathrm{qp}$                         \\ \hline
                  &     &   & $\mathcal{D}_9$  & \multicolumn{1}{c|}{3875}           & \multicolumn{1}{c|}{3}                     &$\PSL_3(5)$  &$\mathrm{qp}$                         \\ \hline
$\PGL_3(4)$        & 21  & 6  & $\mathcal{D}_{10}$ & \multicolumn{1}{c|}{2520}           & \multicolumn{1}{c|}{3}                      & $\PGL_3(4)$     & $\mathrm{qp}$                       \\ \hline
                
\end{tabular}
\caption{Partial linear spaces arising from sporadic cases}\label{tabPLS}
\end{table}

\begin{Remark}\label{qpitsp}
    The type of $G$ in the first four lines of Table \ref{tabLU} depends on $n$ and the choice of $G$ itself. For details see $(1)$, $(2)$ and $(3)$ of Lemma \ref{l:semi}. Note that $\langle \omega^r I\rangle$ is trivial in lines 1--3.  In lines 1--4, all rank 3 groups  $\langle \omega^r I \rangle\,\SL_n(q)/\langle \omega^r I \rangle \leq G \leq \GammaL_n(q)/\langle \omega^r I \rangle$ induce the corresponding space $\mathcal{D}$ via Theorem \ref{th:AD2} (2). We list such groups in column ``Possible $G$'' up to conjugation when the list is short. Notice that the groups in lines 5--6 and 8--10 are normal in $\GammaL_n(q)/\langle \omega^r I \rangle$  and $\GammaU_n(q)/\langle \omega^r I \rangle$ respectively. A conjugate (in  $\GammaL_n(q)/\langle \omega^r I \rangle$) of the group in line 7 induces the same partial linear space (see Remark \ref{rem:9iConj} for details).
\end{Remark}

\begin{Remark}\label{r:L23}
    We note that $\Delta(2,3)$, which is a proper  partial linear space by Lemma~\ref{l:AGDel}, admits the rank $3$ semiprimitive  group $\GL_2(3)$. This does not appear in Theorem~\ref{th:main}(1) since $\GL_2(3)$ is soluble and so does not satisfy Hypothesis~\ref{hyp1}. 
\end{Remark}

\begin{Remark}\label{r:t13}
    The results in Table \ref{tabPLS} are obtained computationally using  {\sc Magma}  \cite{magma}; see the Appendix for details on construction. We note that $\Aut(\mathcal{D}_i)$ is equal to the largest possible $G$ unless $i \in \{2,4\}$ where $\Aut(\mathcal{D}_2)=C_3^3:\PSL_3(3)$  and $\Aut(\mathcal{D}_4)=C_2^3:\PSL_3(2).$
\end{Remark}

 \begin{Remark}
In \cite{R3PIT}, the authors give two examples of proper partial spaces affording a properly innately transitive automorphism group of rank 3. We remark that the partial linear spaces from Examples 1.1 and  1.2 in \cite{R3PIT} are isomorphic to $\mathcal{D}_3$ and $\Delta(2,4)$ respectively. 
\end{Remark}

 While our main goal is to classify connected proper partial linear spaces of rank 3, it is possible to find all the disconnected ones too. It turns out that there is only one example apart from the partial linear spaces  in part $(1)$ of Theorem \ref{th:AD2}.

\begin{Th}\label{th:discon}
    Let $\mathcal{D}$ be a disconnected proper partial linear space admitting a group $G$ of automorphisms satisfying Hypothesis $\ref{hyp1}$. 
    Then one of the following holds:
    \begin{enumerate}[label=\normalfont (\arabic*)]
        \item $\mathcal{D}$ is  a disjoint union of equal-sized lines, namely the nontrivial blocks of imprimitivity of $G$;
        \item  $\mathcal{D}$ is a union of $73$ copies of the $\mathrm{Ree}(3)$-unital and  $G=\PGammaL_3(8)$ acting on the set of $73 \cdot 28 =2044$ points; each copy consists of the $28$ points of a block $\sigma \in \Sigma.$ In particular, $G_{\sigma}^{\sigma} \cong \PGammaL_2(8) \cong \mathrm{Ree}(3)$ and $\Aut (\mathcal{D}) \cong \mathrm{Ree}(3) \wr \Sym(73).$ 
    \end{enumerate}
\end{Th}

\subsection{Partial linear spaces admitting a semiprimitive automorphism group with $G^{\Sigma}$ affine}\label{s:extra}
It was brought to our attention during the refereeing process that imprimitive rank 3 semiprimitive groups with $G^{\Sigma}$ affine have been recently classified by  Li, Yi and Zhu in their preprint \cite{Li2025}. Their classification reports on exactly three such groups up to permutational isomorphism (even though one of them has two inequivalent representations in the sense that corresponding point stabilisers are not conjugate); see \cite[Theorem 1.1]{Li2025} and \cite[Examples 2.5, 2.6 and 2.8]{Li2025}.  All three groups are small enough to be easily handled by {\sc Magma} in the same way as the results of Table \ref{tabPLS} are obtained. The first of them is $\GL_2(3)$ that is mentioned in Remark \ref{r:L23}.   The second group is the stabiliser in $\mathrm{ P \Gamma U}_3(4)$ of a singular projective point $P$, acting on the remaining $2^6$ singular projective points. Such a group preserves a partial linear space $\mathcal{H}^*(2, 4^2)$ obtained from the Hermitian unital $\mathcal{H}(2, 4^2)$  (admitting $\mathrm{ P \Gamma U}_3(4)$) by removing the projective point $P$ and all singular lines through that point. 
The third group in the classification is $V \rtimes H_v \le \mathrm{AGL_4(2)}$, where $V=\mathbb{F}_2^4$, $H \cong A_7$ is 2-transitive on non-zero vectors of $V$ and $v \in V \backslash \{0\}.$ This group does not give rise to any rank 3 proper partial linear spaces.    We summarise these results as follows.

\begin{Lem}\label{gsigaffine}
    Let $\mathcal{D}$ be a proper partial linear space admitting an imprimitive rank $3$ semiprimitive group of automorphisms with $G^{\Sigma}$ affine. Then one of the following holds:
    \begin{enumerate}[label=\normalfont (\arabic*)]
        \item $\mathcal{D}$ is isomorphic to $\Delta(2,3)$ as in Definition $\ref{Deltadef}$ and $G = \GL_2(3)$;
        \item $\mathcal{D}$ is isomorphic to $\mathcal{H}^*(2, 4^2)$ as above and $G \cong 2^{2+4}:\GammaL_1(2^4). $
    \end{enumerate}
\end{Lem}

Now the following theorem follows from Theorems \ref{th:main} and \ref{th:discon} and Lemma \ref{gsigaffine}. 

\begin{Th}\label{t:all}
    The proper partial linear spaces affording an imprimitive semiprimitive rank $3$ automorphism group are classified.
\end{Th}

\bigskip

The paper is organised as follows. In Sections \ref{sec:spacesL} and \ref{sec:spacesU} we define the partial linear spaces that appear in Theorem \ref{th:main}(1) and describe some of their properties.  In Sections \ref{sec:linear} and \ref{sec:unitary} we find and describe the systems of imprimitivity of a point stabiliser in groups appearing in the first three lines of Table \ref{t:qpinsprk3}: the infinite families of groups with $M$ isomorphic to $\PSL_n(q)$ or $\PSU_3(q)$.  
Finally, in Sections \ref{mainproof} and \ref{sec:7} we prove Theorems \ref{th:main} and \ref{th:discon} respectively. The paper concludes with an appendix where our computations in {\sc Magma} are explained.

\section{Partial linear spaces arising from $\GammaL_n(q)$}\label{sec:spacesL}

In this and the following section, we introduce a number of families of partial linear spaces. Each of these families contains partial linear spaces that admit a rank 3 automorphism group which, as mentioned in the introduction, is transitive on the sets of ordered pairs of distinct collinear points, and non-collinear points. 

Let us first prove the following technical lemma.

\begin{Lem}\label{claim1}
Let $\mathcal{D}= (\mathcal{P}, \mathcal{L})$ be a point-line structure with the set of points $\mathcal{P}$ and the set of lines $\mathcal{L}$ such that each line has at least two points. Choose $L \in \mathcal{L}$ with distinct points $\alpha, \beta \in L.$  Assume that a subgroup $G$ of $\Aut(\mathcal{D})$ is such that
\begin{itemize}
    \item $G^{\mathcal{L}}$ is transitive;
    \item $G_L^L$ is $2$-transitive.
\end{itemize}
Then for each $L' \in \mathcal{L}$ and distinct $\alpha', \beta' \in L'$ there exists $g \in G$ mapping $L$ to $L'$, $\alpha$ to $\alpha'$ and $\beta$ to $\beta'$. 
\end{Lem}
\begin{proof}
    Since $G^{\mathcal{L}}$ is transitive, there exists $g_1 \in G$ such that $L^{g_1}=L'.$ Since $G_L^L$ is $2$-transitive, there exists $g_2 \in G_L$ such that $(\alpha)g_2= (\alpha')g_1^{-1}$ and $(\beta)g_2= (\beta')g_1^{-1}$. Hence, if $g=g_2g_1,$ then $L^g=L'$,  $(\alpha)g=\alpha'$ and $(\beta)g=\beta'$.
\end{proof}

\subsection{Definitions and properties of $\mathrm{AG}^*(n,q)$ and $\Delta(n,q)$}

Let $p$ be a prime and let $q=p^a$ for some $a \in \mathbb{N}.$

\begin{Def}\label{AGdef}
    Let $\mathrm{AG}^*(n,q)$ be the pair $(\Omega, \mathcal{L})$ where $\Omega$ is $\mathbb{F}_q^n \backslash \{0\}$ and $\mathcal{L}$ is the set of all affine lines in $\mathbb{F}_q^n$ not containing zero. In other words $\mathcal{L}=\{L_{u,v}\mid u,v \in \mathbb{F}_q^n, \dim(\langle u,v\rangle)=2\}$, where $L_{u,v}=\{\lambda u+(1-\lambda )v \mid \lambda\in \mathbb{F}_q\}.$
\end{Def}

 It is easy to see that lines of $\mathrm{AG}^*(n,q)$ have size $q$. Further, 
$|\mathcal{L}|=(q^n-1)(q^{n-1}-1)/(q-1)$ since $\mathrm{AG}(n,q)$ has $q^{n-1}(q^n-1)/(q-1)$ lines and each point of $\mathrm{AG}(n,q)$ lies on $(q^n-1)/(q-1)$ lines.

\begin{Def}\label{Deltadef}
    Let $\Delta(n,q)$ be the pair $(\Omega, \mathcal{L})$ where $\Omega$ is $\mathbb{F}_q^n \backslash \{0\}$ and $\mathcal{L}=\{L_{u,v} \mid u,v \in \mathbb{F}_q^n, \dim (\langle u,v \rangle)=2\}$ with $L_{u,v}=\{u,v, -(u+v)\}$.
\end{Def}

 Obviously,  lines of $\Delta(n,q)$ have size $3$. Further, $|\mathcal{L}|=(q^n-1)(q^{n}-q)/6$  since a line in $\Delta(n,q)$ is uniquely determined by a couple of linearly independent vectors in $\mathbb{F}_q^n$.

\begin{Lem}\label{l:AGDel}
  For all $n\geq 2$ and  $q\geq 3$,  $\mathrm{AG}^*(n,q)$  and  $\Delta(n,q)$ are  proper partial linear spaces.
  \end{Lem}

\begin{proof}
    In $\mathrm{AG}^*(n,q)$, a pair of collinear points lies on a unique line since lines are affine lines in $\mathbb{F}_q^n.$ In $\Delta(n,q)$ a line is determined uniquely by any two of its points by definition.  
    Since $q\geq 3$, the line-size is at least $3$ in both cases. Moreover, if $\omega$ is a generator of multiplicative group $\mathbb{F}_q^*$, then $\omega\neq 1$ so  $v,\omega v$, with $v\in \mathbb{F}_q^n\backslash \{0\}$, are distinct not collinear points in both cases. Thus $\mathrm{AG}^*(n,q)$ and $\Delta(n,q)$  are proper partial linear spaces.
\end{proof}

  Note that  $\mathrm{AG}^*(n,2)$ (where lines have size $2$) is a complete graph on $2^n-1$ vertices, and that $\Delta(n,2)$ is the linear space $\PG(n-1,2)$, so these are non-proper partial linear spaces. Also note that $\mathrm{AG}^*(n,3)=\Delta(n,3).$

\begin{figure}[h]
 \begin{minipage}{.47\textwidth}
 \begin{adjustwidth*}{}{-1em}
\begin{tikzpicture}[scale =0.6,line width=1pt]
\newcommand{\twoarcs}{
\node (v1) at (0:5cm) [ball color=darkgreen, circle, draw=black, inner sep=1pt] {};  
\node (v2) at (360/15:5cm) [ball color=darkgreen, circle, draw=black, inner sep=1pt] {};
\node (v4) at (360*3/15:5cm) [ball color=darkgreen, circle, draw=black, inner sep=1pt] {};
\node (v8) at (360*7/15:5cm) [ball color=darkgreen, circle, draw=black, inner sep=1pt] {};

\node (c11) at (10:60mm) {};
\node (c12) at (15:60mm) {};
\node (c21) at (45:60mm) {};
\node (c22) at (55:70mm) {};
\node (c31) at (270:30mm) {};
\node (c32) at (290:40mm) {};

\tikzset{
    electron/.style={draw=blue}, 
    gluon/.style={draw=purple}, 
}
\draw[color=\usecolor] (v1) --  (v2);
\draw[color=\usecolor] (v2) ..  controls (c21) and (c22) ..  (v4);
\draw[color=\usecolor] (v4) -- (v8);

}

\def\mycolours{{"col1","col2","col3","col4","col5","col6", "col7","col8","col9"}}

\foreach \x  [evaluate=\x as \usecolor using {\mycolours[Mod(\x,5)]}]in {0,1,2,3,4,5,6,7,8,9,10,11,12,13,14}{
\begin{scope}[rotate=360*\x/15]
\twoarcs;
\end{scope}
}

\node (v1) at (0:5cm) [ball color=thistle!10, circle, draw=thistle, inner sep=1pt,minimum width=15pt] {1};  
\node (v2) at (360/15:5cm) [ball color=thistle!10, circle, draw=thistle, inner sep=1pt,minimum width=15pt] {2};
\node (v3) at (2*360/15:5cm) [ball color=thistle!10, circle, draw=thistle, inner sep=1pt,minimum width=15pt] {3};
\node (v4) at (3*360/15:5cm) [ball color=thistle!10, circle, draw=thistle, inner sep=1pt,minimum width=15pt] {4};
\node (v5) at (4*360/15:5cm) [ball color=thistle!10, circle, draw=thistle, inner sep=1pt,minimum width=15pt] {5};
\node (v6) at (5*360/15:5cm) [ball color=thistle!10, circle, draw=thistle, inner sep=1pt,minimum width=15pt] {6};
\node (v7) at (6*360/15:5cm) [ball color=thistle!10, circle, draw=thistle, inner sep=1pt,minimum width=15pt] {7};  
\node (v8) at (7*360/15:5cm) [ball color=thistle!10, circle, draw=thistle, inner sep=1pt,minimum width=15pt] {8};  
\node (v9) at (8*360/15:5cm) [ball color=thistle!10, circle, draw=thistle, inner sep=1pt,minimum width=15pt] {9};  
\node (v10) at (9*360/15:5cm) [ball color=thistle!10, circle, draw=thistle, inner sep=1pt,minimum width=15pt] {10};  
\node (v11) at (10*360/15:5cm) [ball color=thistle!10, circle, draw=thistle, inner sep=1pt,minimum width=15pt] {11};  
\node (v12) at (11*360/15:5cm) [ball color=thistle!10, circle, draw=thistle, inner sep=1pt,minimum width=15pt] {12};  
\node (v13) at (12*360/15:5cm) [ball color=thistle!10, circle, draw=thistle, inner sep=1pt,minimum width=15pt] {13};  
\node (v14) at (13*360/15:5cm) [ball color=thistle!10, circle, draw=thistle, inner sep=1pt,minimum width=15pt] {14};  
\node (v15) at (14*360/15:5cm) [ball color=thistle!10, circle, draw=thistle, inner sep=1pt,minimum width=15pt] {15};  
\end{tikzpicture}
\end{adjustwidth*}
\caption{$\mathrm{AG}^*(2,4)$ }
\label{fig1}
\end{minipage}
 \begin{minipage}{.47\textwidth}
 \begin{adjustwidth*}{}{-1em}
\begin{tikzpicture}[scale =0.6,line width=1pt]
\newcommand{\arcsone}{
\node (v1) at (0:5cm) [ball color=darkgreen, circle, draw=black, inner sep=1pt] {};  
\node (v2) at (360/15:5cm) [ball color=darkgreen, circle, draw=black, inner sep=1pt] {};
\node (v5) at (360*4/15:5cm) [ball color=darkgreen, circle, draw=black, inner sep=1pt] {};

\node (c11) at (10:60mm) {};
\node (c12) at (15:60mm) {};
\node (c21) at (45:60mm) {};
\node (c22) at (55:70mm) {};
\node (c31) at (270:30mm) {};
\node (c32) at (290:40mm) {};

\tikzset{
    electron/.style={draw=blue}, 
    gluon/.style={draw=purple}, 
}
\draw[color=\usecolor] (v1) ..  controls (c11) and (c12) ..  (v2);
\draw[color=\usecolor] (v2) -- (v5);
}

\newcommand{\arcstwo}{
\node (v1) at (0:5cm) [ball color=darkgreen, circle, draw=black, inner sep=1pt] {};  
\node (v3) at (360*2/15:5cm) [ball color=darkgreen, circle, draw=black, inner sep=1pt] {};
\node (v9) at (360*8/15:5cm) [ball color=darkgreen, circle, draw=black, inner sep=1pt] {};

\node (c11) at (30:70mm) {};
\node (c12) at (15:70mm) {};
\node (c21) at (15:70mm) {};
\node (c22) at (35:70mm) {};
\node (c31) at (115:30mm) {};
\node (c32) at (125:40mm) {};

\tikzset{
    electron/.style={draw=blue}, 
    gluon/.style={draw=purple}, 
}
\draw[color=\usecolor] (v1)  ..  controls (c21) and (c22) ..   (v3);
\draw[color=\usecolor] (v3) -- (v9);

}

\def\mycolours{{"col1","col2","col3","col4","col5","col6", "col7","col8","col9"}}

\foreach \x  [evaluate=\x as \usecolor using {\mycolours[Mod(\x,5)]}]in {0,1,2,3,4,5,6,7,8,9,10,11,12,13,14}{
\begin{scope}[rotate=360*\x/15]
\arcsone;
\end{scope}
}

\foreach \x  [evaluate=\x as \usecolor using {\mycolours[Mod(\x+3,8)]}]in {0,1,2,3,4,5,6,7,8,9,10,11,12,13,14}{
\begin{scope}[rotate=360*\x/15]
\arcstwo;
\end{scope}
}

\node (v1) at (0:5cm) [ball color=thistle!10, circle, draw=thistle, inner sep=1pt,minimum width=15pt] {1};  
\node (v2) at (360/15:5cm) [ball color=thistle!10, circle, draw=thistle, inner sep=1pt,minimum width=15pt] {2};
\node (v3) at (2*360/15:5cm) [ball color=thistle!10, circle, draw=thistle, inner sep=1pt,minimum width=15pt] {3};
\node (v4) at (3*360/15:5cm) [ball color=thistle!10, circle, draw=thistle, inner sep=1pt,minimum width=15pt] {4};
\node (v5) at (4*360/15:5cm) [ball color=thistle!10, circle, draw=thistle, inner sep=1pt,minimum width=15pt] {5};
\node (v6) at (5*360/15:5cm) [ball color=thistle!10, circle, draw=thistle, inner sep=1pt,minimum width=15pt] {6};
\node (v7) at (6*360/15:5cm) [ball color=thistle!10, circle, draw=thistle, inner sep=1pt,minimum width=15pt] {7};  
\node (v8) at (7*360/15:5cm) [ball color=thistle!10, circle, draw=thistle, inner sep=1pt,minimum width=15pt] {8};  
\node (v9) at (8*360/15:5cm) [ball color=thistle!10, circle, draw=thistle, inner sep=1pt,minimum width=15pt] {9};  
\node (v10) at (9*360/15:5cm) [ball color=thistle!10, circle, draw=thistle, inner sep=1pt,minimum width=15pt] {10};  
\node (v11) at (10*360/15:5cm) [ball color=thistle!10, circle, draw=thistle, inner sep=1pt,minimum width=15pt] {11};  
\node (v12) at (11*360/15:5cm) [ball color=thistle!10, circle, draw=thistle, inner sep=1pt,minimum width=15pt] {12};  
\node (v13) at (12*360/15:5cm) [ball color=thistle!10, circle, draw=thistle, inner sep=1pt,minimum width=15pt] {13};  
\node (v14) at (13*360/15:5cm) [ball color=thistle!10, circle, draw=thistle, inner sep=1pt,minimum width=15pt] {14};  
\node (v15) at (14*360/15:5cm) [ball color=thistle!10, circle, draw=thistle, inner sep=1pt,minimum width=15pt] {15};  
\end{tikzpicture}
\end{adjustwidth*}
\caption{$\Delta(2,4)$}
\label{fig2}
\end{minipage}
\end{figure}

  \begin{Ex}
 We would like to illustrate some examples visually.      Notice that the natural action of $\GammaL_n(q)$ on $\mathbb{F}_q^n$ induces automorphisms on  $\mathrm{AG}^*(n,q)$ and $\Delta(n,q)$ (in particular, we may consider $\Omega$ to be same as defined in Construction \ref{con:psl} below with $r=q-1$). Let $(n,q)=(2,4),$ so $|\Omega|=15.$ In this case, a Singer cycle $\langle s \rangle$ of $\GL_2(4)$ (see Definition \ref{des:singer}) has order $q^2-1=15$ and acts on $\Omega$ regularly. Moreover, $\langle s \rangle$ acts transitively on the set of the 15 lines of $\mathrm{AG}^*(2,4)$ while it has two orbits on the set of the 30 lines of $\Delta(2,4).$     Hence it is possible to obtain all the lines of  $\mathrm{AG}^*(2,4)$ (resp. $\Delta(2,4)$) by repeatedly applying $s$ to a single line (resp. two lines from distinct orbits). Using   {\sc Magma}, it is easy to choose a Singer cycle, label the points in $\Omega$ by numbers $1, 2, \ldots, 15,$ so that $s=(1,2, \ldots, 15),$ and find the lines. One of the possible presentations gives us a line $\{1,2,4,8\}$ in $\mathrm{AG}^*(2,4)$ and a couple of lines $\{1,2,5\}$ and $\{1,3,9\}.$ Now, by placing the points of $\Omega$ in a circle and applying the rotation $s$ to these lines, one obtains all the lines of the corresponding partial linear space, see Figures \ref{fig1} and \ref{fig2}. 
  \end{Ex}

\subsection{Definition and properties of $\mathrm{LSub}(n,q,q_0,r)$}

To define the next two families of partial linear spaces we first need to define a  particular action of subgroups of $\GammaL_n(q)$. We follow \cite[Constructions 4.2]{R3PIT} stating the construction in a more general manner.

\begin{construction}\label{con:psl}
Let $q=p^a$ with $p$ a prime, $a\geq1$ and $q \geq 3$,  and let $\omega$ be a generator of the multiplicative group $\mathbb{F}_q^*.$ Let  $n\geq2$ with $(n,q)\ne (2,3)$ and let $r>1$ be an integer dividing $q-1$. Let $V=\mathbb{F}_q^n$ denote the space of $n$-dimensional row vectors, and   $\binom{V}{1}$ the set of $1$-subspaces of $V$.  Define
$$
\Omega := \{ \langle \omega^r\rangle u \mid  u\in V^*\}.
$$ 
Then the natural induced action of $\GammaL_n(q)$ on $\Omega$ given by
\begin{equation}\label{e:action}
    g:\omr v\mapsto \omr(v)g, \ \mbox{for $g\in\GammaL_n(q)$ and 
    $ v\in V^*$}
\end{equation}
 yields a permutation group $\ff=\GammaL_n(q)/Y$ where  $Y=\langle \omega^r I\rangle < \GammaL_n(q)$. The group $\ff$ stabilises the partition 
$$
\Sigma=\{\sigma(U) \mid U \in \tbinom{V}{1}\}\ 
\text{ where } \sigma(U)=\{\langle \omega^r\rangle \omega^i u\mid 0\leq i<r\} \text{ for } U=\langle u\rangle\in \tbinom{V}{1}.
$$
\end{construction}

The group actions in Construction~\ref{con:psl} are semiprimitive. Precisely, the following holds.

\begin{Lem}\label{l:semi}
   Let $q,n,r, Y, \Omega, \overline{G}$ be as in Construction~$\ref{con:psl}$,
   and suppose that $G$ is a group satisfying $Y\,\SL_n(q)/Y\leq G \leq \ff$. Then the following hold. 
   \begin{enumerate}[label=\normalfont (\arabic*)]
       \item $G$ is semiprimitive on $\Omega$;
       \item $G$ is innately transitive on $\Omega$ if and only if $r$ divides $(q-1)/(n,q-1)$; and in this case $Y\,\SL_n(q)/Y\cong \PSL_n(q)$ is the unique plinth of $G$; 
       \item $G$ is quasiprimitive on $\Omega$ if and only if $r$ divides $(q-1)/(n,q-1)$ and $G\cap (Z/Y)=1$;
       \item $G$ has rank $3$ on $\Omega$ if and only if one of the following holds:
       \begin{enumerate}
           \item[$(i)$] $(n,r)\ne (2,2)$,  $r$ is a primitive prime divisor of $p^{r-1}-1$  and $(r-1, \frac{a}{|G:G\cap (\GL_n(q)/Y)|})=1$;
           \item[$(ii)$] $(n,r)=(2,2)$ and $G\not\leq Y\,{\Sigma\mathrm{L}_2(q)}/Y$. 
       \end{enumerate}  Moreover, if $G$ has rank 3, then $r-1$ divides $a$.
   \end{enumerate}
 \end{Lem}

\begin{proof}
The subgroup $Y=\langle \omega^r I\rangle$ has index $r$ in the centre $Z:= \langle \omega I\rangle$ of $\GL_n(q)$, and $Z\cap\SL_n(q)$ has order $(n,q-1)$. Now $Y\,\SL_n(q)/Y\leq G \leq \ff$ and there is a unique subgroup $H$ of $\GammaL_n(q)$ such that $Y\leq H$ and $H/Y=G$. Further, it is well known and easy to check that every normal subgroup $N$ of $H$ satisfies either $N\leq Z$ or $N\geq \SL_n(q)$, and so every normal subgroup $K$ of $G$ satisfies either $K\leq Z/Y$ or $K\geq Y\,\SL_n(q)/Y$.  Since $Z/Y$ is semiregular on $\Omega$, it follows that each normal subgroup $K\leq Z/Y$ is semiregular on $\Omega$. Also since $\SL_n(q)$ is transitive on $V^*$ it follows that $Y\,\SL_n(q)/Y$ is transitive on $\Omega$ and hence each normal subgroup $K\geq Y\,\SL_n(q)/Y$ is transitive on $\Omega$. Thus $G$ is semiprimitive on $\Omega$, proving part $(1)$.

Now $G$ is innately transitive if and only if it has a transitive minimal normal subgroup $K$. Since $Z/Y$ is not transitive on $\Omega$, the only possibility for such a transitive minimal normal subgroup is $K=Y\,\SL_n(q)/Y\cong \SL_n(q)/(Y\cap\SL_n(q))$. Note that $\SL_n(q)/(Z\cap\SL_n(q))=\PSL_n(q)$ (which is a non-Abelian simple group since $(n,q)\ne(2,3)$ and $q\geq3$), and that $Y\cap\SL_n(q) \leq Z\cap\SL_n(q)$. Thus $K=Y\,\SL_n(q)/Y$ is a minimal normal subgroup if and only if $Y\cap\SL_n(q) = Z\cap\SL_n(q)$, and this holds if and only if $Y$ contains the unique subgroup $Z\cap\SL_n(q)$ of $Z$ of order $(n,q-1)$. In turn, this holds if and only if $|Y|=(q-1)/r$ is divisible by $(n,q-1)$, that is to say, $r$ divides $(q-1)/(n,q-1)$. This proves the first assertion of part $(2)$, and as $Z/Y$ is not transitive on $\Omega$ it follows in this case that $Y\,\SL_n(q)/Y$ is the unique plinth of $G$, and $(2)$ is proved.

Also $G$ is quasiprimitive if and only if it has no nontrivial intransitive normal subgroups, and this is the case if and only if $G$ is innately transitive, and the subgroup $Z/Y$ intersects $G$ trivially, proving part $(3)$. 

Finally we pull together various results to determine when $G$ can be rank $3$. Firstly, by  \cite[Proposition 4.9]{semiprR3}, if $n\geq 3$, then $G$ has rank $3$ if and only if the conditions on $r$ in part $(d)(i)$ hold.  Thus we may assume that $n=2$. If $r=2$, then by \cite[Proposition 4.10]{semiprR3}, $G$ has rank $3$ if and only if the conditions in part $(4)(ii$) hold. Thus we may assume that $n=2$ and $r>2$. Then it follows from \cite[Proposition 4.11]{semiprR3} that $G$ is innately transitive (so $r$ divides $(q-1)/(2,q-1)$ by part (2)). Also, by  \cite[Theorem 1.2 and Table 1]{DGLPP}, there are no examples if the imprimitive group $G$ is quasiprimitive, and hence $G$ is properly innately transitive. Finally, by  \cite[Theorem C and Table 3]{R3PIT}, $G$ has rank $3$ if and only if the conditions of $(4)(i$) hold. 

 Conditions in part $(4)(i)$ imply that $r-1$ divides $a$. Indeed,   since $r$ also divides $q-1=p^a-1$ we have that $r$ divides $(p^a-1,p^{r-1}-1)=p^{(a,r-1)}-1$, and by the definition of a primitive prime divisor it follows that $(a,r-1)=r-1$, that is, $r-1$ divides $a$. In part $(4)(ii)$, $r-1=1$ which obviously divides $a$. This proves the last statement of part $(4).$ 
\end{proof}

Notice that  $r$ must be prime for a group $G$ in Lemma~\ref{l:semi} to have rank $3$. This is the case in \cite[\S 4]{R3PIT} and \cite[\S 4]{semiprR3}, so Construction \ref{con:psl} is a direct generalisation of corresponding constructions in \cite[\S 4]{R3PIT} and \cite[\S 4]{semiprR3}.

\begin{Not}
   We often write $\omr \{v_1, \ldots, v_n\}$ for the set $\{\omr v_1, \ldots, \omr v_n\}.$ This convenient notation allows us to shorten some formulas significantly. See for example Tables \ref{t:blocks}, \ref{th:blocksn2} and \ref{t:blocksSU}.
\end{Not}

We are now ready to define the following series of partial linear spaces.

\begin{Def}\label{con:subfield}
With the notation and assumptions of Construction \ref{con:psl}, let $n \geq 2$ and fix a basis $\{e_1, e_2,  \ldots, e_n \}$ of $V$. Assume that
$q=q_0^f$ for a prime power $q_0$ and an integer $f > 1$, and also that $\frac{q-1}{q_0-1}=rk$ for some  integer $k$.  Recall that $r>1$.
Let $t$ be the least positive integer such that $\omr[t] \cap \omr =\omr[kr].$  Define the following sets:
\begin{itemize}
    \item $L_{u,v}= \omr\{\lambda_1 u + \lambda_2 v\mid \lambda_1,\lambda_2\in \mathbb{F}_{q_0}, (\lambda_1,\lambda_2) \ne (0,0) \}$ for linearly independent $u,v \in V$;
    \item  $\mathcal{L}=\{L_{u,v} \mid  (u,v)=(e_1,e_2)^g, g \in \GL_{n}(q), \det(g) \in \omr[t] \}$;
    \item the incidence structure $\mathrm{LSub}({n}, q, q_0,r)=(\Omega, \mathcal{L})$.
\end{itemize}  
\end{Def}

Note that $\mathbb{F}_{q_0}^* = \omr[rk]$, in particular, $\mathbb{F}_{q_0}^* \subseteq \omr$. Also $t$ is well defined since $\omr[kr] \cap \omr = \omr[kr].$
 If $n\geq3$, then $\SL_n(q)$ is transitive on all linearly independent pairs $u,v$, and  the condition on $\det(g)$ does not restrict $\mathcal{L}$, so  $\mathcal{L}=\{L_{u,v}\mid \dim(\langle u,v\rangle)=2\}.$ However this is not true in general for $n=2$.  It is worth mentioning that if $q$ is even, $q_0=2$ and $r=q-1$, then $\mathrm{LSub}({n}, q, q_0,r)=\Delta(n,q)$.
 
 Note that if $k=1$, then $t=1$ and $\mathcal{L}=\{L_{u,v}\mid \dim(\langle u,v\rangle)=2\}.$  In this case, $\mathbb{F}_{q_0}^* = \omr[kr]=\omr$ so $\Omega$ is the point-set of $\PG(fn-1,q_0)$, and the lines are all the lines of $\PG(fn-1,q_0)$ not contained in an $(f-1)-$spread obtained from field reduction. See \cite{VV2016} for details.

\begin{Lem}\label{lem:lineform}
Let  $u,v \in V$ be  linearly independent.
In terms of Definition \ref{con:subfield}, the following hold:
\begin{enumerate}[label=\normalfont (\arabic*)]
    \item $L_{u,v}= \{\omr u\}\cup  \omr\{ \lambda u + v \mid \lambda \in \mathbb{F}_{q_0}\}$ of cardinality $q_0+1$; 
     \item  if distinct points $\omr w_1,\omr w_2$ are collinear then  $w_1,w_2$ are linearly independent;  
    \item $(L_{u,v})^g =L_{(u)g, (v)g}$ for  $L_{u,v} \in \mathcal{L}$ and $g \in \GammaL_{n}(q);$ \label{lf3}
   \item  if $\mu\in\omr$, then  $L_{\mu u, v}=L_{u,\mu^{-1}v}$;
   \item  for $\eta\in\mathbb{F}_q^*$, we have $L_{u,\eta v}=L_{u,v}$  if and only if $\eta\in \mathbb{F}_{q_0}^*$;
   \item  $t$ is the least positive integer such that $k=t/(r,t)$; in particular, $t$ divides $kr$,  $k$ divides $t$, $kr={\rm lcm}\{r,t\}$,   and the subgroup product $\omr[t] \omr = \omr[t/k]$; 
   \item  if $\pi$ is the set of all primes dividing $k$, then $t=k \cdot r_{\pi},$ where $r_{\pi}$ is the $\pi$-part of $r$.
\end{enumerate}
\end{Lem}
\begin{proof} $(1)$
   Let $M =  \{\omr u\}\cup \omr\{  \lambda u + v \mid \lambda \in \mathbb{F}_{q_0}\}.$ By the definition of $L_{u,v}$ we have $M \subseteq L_{u,v}$. Now let $\omr(\lambda_1 u + \lambda_2 v) \in L_{u,v},$ so that $\lambda_1,\lambda_2 \in \mathbb{F}_{q_0}$.  If $\lambda_2 \ne 0$ then $\lambda_2\in \mathbb{F}_{q_0}^* \subseteq \omr,$ and hence
    $$
    \omr(\lambda_1 u + \lambda_2 v)=\omr\lambda_2(\lambda_1\lambda_2^{-1} u + v)=\omr(\lambda_1\lambda_2^{-1} u + v) \in M.
    $$ 
     On the other hand if $\lambda_2=0$ then $\lambda_1\ne 0$, so $\lambda_1\in \mathbb{F}_{q_0}^* \subseteq \omr,$ and we have $\omr(\lambda_1 u + \lambda_2 v) = \omr \lambda_1 u=\omr u \in M$. Hence $L_{u,v}\subseteq M$ and equality holds. The elements in $M$ are clearly pairwise non-equal and hence $|M|=q_0+1$,  proving part (1).

    $(2)$ Let $\omr w_1,\omr w_2$  be distinct points. Assume  that $\omr w_1,\omr w_2$ are collinear and $w_1,w_2$ are linearly dependent. Then $w_2=\mu w_1$ for some $\mu\in\mathbb{F}_q$. 
  Let $L_{u,v}$ be a line containing $\omr w_1$ and $\omr w_2$. Then $\omr w_1=\omr u$ or $\omr(\lambda u + v)$ for some $\lambda\in\mathbb{F}_{q_0}$, by part (1), and hence $\omr w_2=\omr \mu u$ or $\omr(\mu \lambda u + \mu v)$, respectively. Since $\omr w_2\in L_{u,v}$, it follows that $\mu\in\omr$, and so $\omr w_1=\omr w_2$, a contradiction.

    $(3)$ We may write $g= \phi^j h $, where $0\leq j<a$, $h \in \GL_{n}(q)$, and $\phi$ is as in \eqref{def:phi}, with respect to a basis of $V$ containing the vectors $u$ and $v$.
      Then the field automorphism $\phi$ normalises $\omr$ and $\mathbb{F}_{q_0}^*$, and  we have
    $$
    (\omr (\lambda_1 u + \lambda_2 v))g=\omr (\lambda_1^{\phi^j}(u)h +\lambda_2^{\phi^j}(v)h) \in L_{(u)h,(v)h} = L_{(u)g,(v)g}
    $$
    for all $\lambda_1, \lambda_2 \in \mathbb{F}_{q_0}$. Thus $(L_{u,v})^g \subseteq L_{(u)g, (v)g}$  and equality holds since $|L_{u,v}|=|L_{(u)g, (v)g}|=q_0+1$.

    (4) Since  $\mu\in\omr$, we have $\omr \mu u=\omr u$ and $\omr (\lambda \mu u+v)=\omr (\lambda  u+\mu^{-1} v)$, for all $\lambda\in\mathbb{F}_0$. The assertion now follows from part (1).
    
    (5) Assume that $\eta\in \mathbb{F}_{q_0}^*$.
    Both lines $L_{u,\eta v}$ and $L_{u,v}$ contain the point $\omr u$. By part (1),
each other point of $L_{u,\eta v}$ is equal to 
$\omr(\lambda u+\eta v)$ for some $\lambda\in \mathbb{F}_{q_0}$, which is a point of $L_{u,v}$. Thus  $L_{u,\eta v}\subseteq L_{u,v}$  and equality holds since both sets have the same size.

Now assume that  $L_{u,\eta v}=L_{u,v}$. 
The point $\omr (u+\eta v)\in L_{u,\eta v}$ must then be equal to $\omr(\lambda u+ v)$ for some $\lambda\in \mathbb{F}_{q_0}$ by part (1). Since this point is distinct from $\omr v$ it follows that $\lambda\ne 0$, and hence  $\omr (u+\eta v)=\omr(\lambda u+ v) =\omr(u+ \lambda^{-1}v)$. This implies that  $\eta=\lambda^{-1}\in \mathbb{F}_{q_0}^*$.

    (6)  First we show that $t$ divides $(q-1)$. Note that $t/(t,q-1)$ is coprime to $(q-1)/(t,q-1)$, and so $|\omega^t|=\frac{q-1}{(t,q-1)}=|\omega^{(t,q-1)}|$. This implies that $\omr[t]=\omr[(t,q-1)]$, and so by the minimality in the definition of $t$ we must have $t=(t,q-1)$, and hence $t$ divides $q-1$, as required. 
    The cyclic group $\langle \omega\rangle$ of order $q-1=kr(q_0-1)$ has a unique subgroup of order $d$ for each divisor $d$ of $q-1$. Since both $t$ and $r$ divide $q-1=|\omega|$, we obtain $\omr[t]\cap \omr = \omr[{\rm lcm}\{r,t\}]$ and hence, since also $kr$ divides $q-1$, we also have  $kr={\rm lcm}\{r,t\}=rt/(r,t)$. Hence $k=t/(r,t)$; in particular $t$ divides $kr$ and $k$ divides $t$. 
    
    To prove that $t$ is the least integer with this property, assume to the contrary that $k=t'/(r,t')$ for some $t'<t$. Then $t'=k(r,t')$ divides $kr$ and hence divides $q-1.$ Further, ${\rm lcm}\{r,t'\}=rt'/(r,t')=kr$, so $\omr[t'] \cap \omr =\omr[kr]$ which contradicts the minimality in the definition of $t$. Thus $t$ is  the least positive integer such that $k=t/(r,t)$.    
     
     Finally,  the product $\omr[t] \omr$ is a subgroup of $\langle \omega\rangle$, and hence is equal to $\omr[\ell]$ for some divisor $\ell$ of $q-1$. Since $\omr[\ell]$ contains both $\omr[t]$ and $\omr[r]$, it follows that $\ell$ divides both $t$ and $r$, that is, $\ell$ divides $(r,t)$. Hence $\omr[\ell]$ contains $\omr[(r,t)]$. On the other hand $\omr[(r,t)]$ contains both $\omr[t]$ and $\omr[r]$, and hence contains their product $\omr[\ell]$. It follows that these subgroups are equal and hence $\ell=(r,t)$. The final assertion follows from the equality $(r,t)=t/k$.

    (7) 
    Let $\pi'$ denote the complement of $\pi$ in the set of all primes, and let $x_\pi, x_{\pi'}$ denote the $\pi$-part and $\pi'$-part, respectively, of an integer $x$. Then in particular $t=t_\pi\cdot t_{\pi'}$ and  $k_{\pi'}=1$. Also by part (6) we have $t=k\cdot (r,t)$, and so $t_{\pi'}=k_{\pi'}\cdot (r,t)_{\pi'}=(r,t)_{\pi'}$, and $(r,t_\pi) = (r_\pi,t_\pi) = (r,t)_\pi$. Thus $(r,t)=(r,t)_\pi \cdot (r,t)_{\pi'}=(r,t_\pi)\cdot t_{\pi'}$, which implies that 
    \[
    k=\frac{t}{(r,t)} =\frac{t_\pi\cdot t_{\pi'}}{(r,t_\pi)\cdot t_{\pi'}} =\frac{t_\pi}{(r,t_\pi)}
    \]
    and it follows from the minimality of $t$ in (6) that  $t=t_\pi$. 
     If $k=1$, then $\pi$ is empty, $t=1$ and the statement holds. Further we assume $|\pi|\ge1$.  Let $s\in \pi$, so $t_s=k_s\cdot (r,t)_s$. Since $k_s\neq 1$, $(r_s,t_s)=(r,t)_s\neq t_s$, and it follows that $r_s<t_s$ and $(r,t)_s=r_s$. Thus $t_s=k_s\cdot r_s$.  
    
        Since this holds for all $s\in\pi$ we have $t=t_\pi=k\cdot r_\pi$, completing the proof of (7).
    \end{proof}

 Let $P_2$ be the stabiliser  in $\GL_n(q)$ of the subspace $\langle e_1, e_2 \rangle$, and let $M$ be the largest subgroup of $P_2$ which induces precisely $\GL_2(q_0)$ on $\langle e_1, e_2 \rangle$. Then $M$ consists of all matrices 
\begin{equation}\label{eq:X0mat}
    \begin{pmatrix}
    A & 0 \\
    c & B
\end{pmatrix}
\end{equation}
    with $A \in \GL_2(q_0),$ $B\in \GL_{n-2}(q)$ and  $c \in \mathbb{F}_q^{(n-2)\times 2}.$  Note that if $n=2$, then $M=\GL_2(q_0).$

 \begin{Lem} \label{lf7}
      The stabiliser in $\GL_n(q)$ of the line $L_{e_1,e_2}$ is equal to $X=MY$, where $Y=\langle \omega^rI\rangle$ is as in Construction~$\ref{con:psl}$. Moreover
   \begin{enumerate}
       \item[$(1)$] $X$  and its subgroup $M$ both act $2$-transitively on the $q_0+1$ points of $L_{e_1,e_2}$;
       \item[$(2)$] $\GL_n(q)$ acts transitively on $\widehat{\mathcal{L}}:=\{L_{u,v}\mid u,v\ \text{linearly independent}\}$; and  $G_1:=\{g\in\GL_n(q)\mid \det(g)\in\omr[t]\}$  is transitive on $\mathcal{L}$, and we have 
       \[
|\widehat{\mathcal{L}}| = \frac{rq(q^n-1)(q^{n-1}-1)}{q_0(q_0^2-1)(q-1)}\quad\text{and}\quad |\mathcal{L}|= \begin{cases}
    |\widehat{\mathcal{L}}|/(2r,t)  \text{ if } n=2; \\
    |\widehat{\mathcal{L}}| \text{ if } n\geq 3.    
\end{cases}
       \]
   \end{enumerate}
\end{Lem}
\begin{proof}
      Let $L=L_{e_1,e_2}$ and let $X$ be the stabiliser in $\GL_n(q)$ of $L$. First we show that $M\leq X$. Let $g\in M,$ so $g$ is as in \eqref{eq:X0mat} with $A \in \GL_2(q_0)$. Then by Lemma \ref{lem:lineform}(3), $L^g=L_{(e_1)g,(e_2)g}$, and by Lemma \ref{lem:lineform}(1), the elements of this line are $\omr (e_1)g$ and  $\omr (\lambda (e_1)g+(e_2)g)$ for $\lambda\in \mathbb{F}_{q_0}.$  Since $g$ stabilises $\langle e_1, e_2 \rangle$ and all entries of $A$ are in $\mathbb{F}_{q_0}$, all these points lie in $L$. Thus $L^g\subseteq L$ and we have equality since both sets have cardinality $q_0+1$. Hence $g\in X$, and therefore $M\leq X$.  Moreover $Y$ fixes every point of $\Omega$, so in particular $Y\leq X$, and hence   $MY\leq X$.
    
    Next we show that $M$, and hence also $X$, acts $2$-transitively on the points of $L$. Consider a pair of distinct points $\omr u_1$, $\omr u_2$ of $L_{e_1,e_2}$. By Lemma \ref{lem:lineform}(1),  for each $i$ we have $\omr u_i =  \omr (\nu_{i1}e_1 + \nu_{i2}e_2)$ where $(\nu_{i1},\nu_{i2})$ is either $(1,0)$ or $(\lambda,1)$ for some $\lambda\in\mathbb{F}_{q_0}$.    Let $h$ be the $2\times 2$ matrix with $ij$-entry $\nu_{ij}$. Then by Lemma \ref{lem:lineform}(2), $h$ is invertible and hence $h\in \GL_2(q_0)$ and $\diag(h, I_{n-2}) \in M$ maps the two points $\alpha_1:=\omr e_1, \alpha_2:=\omr e_2$ to the two points $\omr u_1,\omr u_2$ of $L_{e_1,e_2}$ (in that order). This proves part $(1)$.

    Since $ M$ is $2$-transitive on the points of $L$, it follows that $X=MX_{\alpha_1,\alpha_2}$. Let $g\in X_{\alpha_1,\alpha_2}$. Then for each $i\in\{1,2\}$,  $g:e_i\to\lambda_ie_i$ for some $\lambda_i\in \omr$. 
   Then $L^g=L_{(e_1)g,(e_2)g}=L_{\lambda_1 e_1,\lambda_2 e_2}$ by Lemma \ref{lem:lineform}(3), and $L^g=L$ as $g\in X$. Also by Lemma \ref{lem:lineform}(4), $L_{\lambda_1 e_1,\lambda_2 e_2}=L_{ e_1,\lambda_1^{-1}\lambda_2 e_2}$. 
    Thus $L=L_{ e_1,\lambda_1^{-1}\lambda_2 e_2}$ and it follows that $\lambda_1^{-1}\lambda_2\in\mathbb{F}_{q_0}^*$ by Lemma \ref{lem:lineform}(5).
    Since $\lambda_2\in \omr$, we have $\lambda_2I\in Y$. Further $g_1:=g(\lambda_2^{-1}I)$ leaves $\langle e_1,e_2\rangle$ invariant (and hence lies in $P_2$), and $g_1$ induces on $\langle e_1,e_2\rangle$  the diagonal matrix $\diag [\lambda_1\lambda_2^{-1},1]\in\GL_2(q_0)$, and therefore $g_1\in M$. Thus $g=g_1(\lambda_2I)\in MY$, and it follows that $X=MX_{\alpha_1,\alpha_2}\leq MY$, so equality  $X=MY$ holds, proving the main assertion of the lemma.

    Finally we prove part $(2)$. 
    By Lemma \ref{lem:lineform}(3) and since $\GL_n(q)$ is transitive on ordered bases, 
    $\GL_n(q)$ is transitive on $\widehat{\mathcal{L}}$ and the stabiliser of the line $L$ in this action is $X$.  Since $X<P_2<\GL_n(q)$, we have $|\widehat{\mathcal{L}}|=|\GL_n(q):P_2|\cdot |P_2:X|$, and the first factor is equal to the number of $2$-dimensional subspaces of $\mathbb{F}_q^n$, namely $\frac{(q^n-1)(q^{n-1}-1)}{(q^2-1)(q-1)}$. Since $X=MY$, the second factor is $|P_2:M|/|X:M|$. It follows from \eqref{eq:X0mat} that $|P_2:M|=|\GL_2(q):\GL_2(q_0)| = \frac{(q^2-1)(q^2-q)}{(q_0^2-1)(q_0^2-q_0)}$, and since $M\cap Y=\{\lambda I\mid \lambda\in\mathbb{F}_{q_0}^*\}$ we have 
    $|X:M|=|MY:M|=|Y:M\cap Y|=\frac{q-1}{r(q_0-1)}$. Therefore
    \[
    |\widehat{\mathcal{L}}| = \frac{(q^n-1)(q^{n-1}-1)}{(q^2-1)(q-1)}\cdot  \frac{(q^2-1)(q^2-q)}{(q_0^2-1)(q_0^2-q_0)}\cdot \frac{r(q_0-1)}{q-1} = \frac{rq(q^n-1)(q^{n-1}-1)}{q_0(q_0^2-1)(q-1)}.
    \]
    As noted after Definition~\ref{con:subfield},  if $n \geq 3$ then $\mathcal{L}=\widehat{\mathcal{L}}$, and the group $G_1$ is transitive on $\mathcal{L}$ since $G_1$ contains $\SL_n(q)$. Thus the lemma is proved if $n\geq3$.
    
    Assume that $n=2$. Here by definition, $G_1$ is transitive on $\mathcal{L}$, and the stabiliser of $L$ in this action is $X\cap G_1$. We proved above that $X=MY$ and, as noted above, $M=\GL_2(q_0)$ when $n=2$. 
    Note that, for $g\in M$, $\det(g)\in \mathbb{F}_{q_0}^*=\omr[rk]\subseteq \omr[t],$ and so $M\leq G_1$. Thus
    $X\cap G_1 = \GL_2(q_0)(Y\cap G_1) = M(Y\cap G_1)$, and we have  $Y\cap G_1=\{\lambda I \mid \lambda \in \omr, \lambda^2 \in \omr[t]\}$. 
    To complete the proof we require the following crucial fact:
    
    \medskip\noindent

  \medskip\noindent
    \emph{Claim.\ $|Y\cap G_1|= (2,t)(q-1)/rk =(2,t)(q_0-1)$, and moreover,  $(2,t)(r,t)=(2r,t)$.}\quad 
    
    Now  $t=k\cdot(r,t)$ by Lemma~\ref{lem:lineform}(6),   and hence $q-1=rk(q_0-1)$ is divisible by $t$ and it follows that
       $$
       \{\lambda \in \mathbb{F}_q^* \mid \lambda^2 \in \omr[t]\}=\omr[t/(2,t)]=
       \begin{cases}
           \omr[t/2] & \text{if $t$ is even} \\
           \omr[t]   & \text{if $t$ is odd}.
       \end{cases}
       $$
       Hence $|Y\cap G_1|= |\omr \cap \omr[t/(2,t)]|=|\omr[\mathrm{lcm}\{r,t/(2,t)\}]|$. If $t$ is odd, then $(2,t)=1$,  $\mathrm{lcm}\{r,t/(2,t)\} =\mathrm{lcm}\{r,t\}=rk$, and so in this case $|Y\cap G_1|=(q-1)/rk = q_0-1$ and $(2,t)(r,t)=(2r,t)$ as claimed. Thus we may assume that $t$ is even. By Lemma  \ref{lem:lineform}$(7)$, $t=k\cdot r_\pi$, where $\pi$ is the set of primes dividing $k$. It follows that $2\in\pi$, and hence $t_2=k_2 \cdot r_2\geq 2\cdot r_2$. Thus $r_2$ divides $t_2/2$,  so $\mathrm{lcm}\{r,t/2\}_2=t_2/2$ and hence $\mathrm{lcm}\{r,t/2\}=\mathrm{lcm}\{r,t\}/2=rk/2$. We conclude that $|Y\cap G_1|= (q-1)/(rk/2) = 2(q_0-1)$, and the first assertion of the Claim is proved. Note that if $r$ is odd then $(2,t)(r,t)=(2r,t)$ as claimed. Assume then  $r$ is even. Since $2\in\pi$, we have $t_2 = (k\cdot r_\pi)_2 = k_2\cdot r_2\geq 2\cdot r_2$, and it follows that  $(2,t)(r,t)=(2r,t)$ in this case too. Thus the Claim is proved.

       \medskip
        Note that $Y\cap M = Y\cap M\cap G_1$, since $M\leq G_1$. By the Claim, 
       \[
       \frac{|X|}{|X\cap G_1|}=\frac{|MY|}{|M(Y\cap G_1)|}= \frac{|M|\cdot |Y|}{|M\cap Y|}\cdot \frac{|M\cap (Y\cap G_1)|}{|M|\cdot |Y\cap G_1|} = \frac{|Y|}{|Y\cap G_1|}=\frac{(q-1)/r}{(2,t)(q-1)/rk}=\frac{k}{(2,t)}.
       \]
       This is equal to $\frac{t}{(2,t)(r,t)}$ (since $t=k\cdot(r,t)$) and hence, by the second assertion of the Claim, $\frac{|X|}{|X\cap G_1|}=\frac{t}{(2r,t)}$.

      Now  $|\widehat{\mathcal{L}}|=|\GL_2(q):X|$ and $|\mathcal{L}|=|G_1:X\cap G_1|$, and hence
      \[
      \frac{|\widehat{\mathcal{L}}|}{|\mathcal{L}|} = \frac{|\GL_2(q):G_1|}{|X:X\cap G_1|} = \frac{t}{t/(2r,t)}  
     =(2r,t). 
      \]
       This completes the proof of part (2).
       \end{proof}

\begin{Th}\label{th:LSub}
Let $n, q, q_0, k, r,$ and $\mathcal{D}=\mathrm{LSub}(n,q,q_0,r)$ be as in Definition~\ref{con:subfield}. Then the number  $\mathfrak{m}$ of lines containing a pair of collinear points is a constant, namely 
$$
\mathfrak{m}=\left\{ \begin{array}{ll}
        k/(2,k)& \text{ if } n=2; \\
        k &\text{ if } n\geq 3.
    \end{array}\right. 
    $$
In particular, $\mathcal{D}$ is a  partial linear space if and only if either  $n=2$ with $k\leq 2$, or $n\geq 3$ with $k=1$,  in which case $\mathcal{D}$ is a  proper partial linear space.  
\end{Th}

 \begin{proof} 
The following claim follows immediately from Definition \ref{con:subfield}, Lemma \ref{lf7}(1) and Lemma \ref{claim1}.

\medskip  
\noindent
{\it Claim 1:  If $w_1,w_2 \in V^*$ are such that $\omr w_1$ and $\omr w_2$ are distinct collinear points of $\Omega$ on a line $L$, then there is an element $g\in\GL_n(q)$ with $\det(g)\in\omr[t]$, mapping  $L_{e_1,e_2}$ to $L$ and  $\omr e_i$ to $\omr w_i$ for $i=1,2.$}

\medskip    
By Lemma \ref{lf7}(2),  $g$ (as in Claim 1) preserves $\mathcal{L}$ (so induces an automorphism of $\mathcal{D}$), and in particular  the number of lines containing $\omr w_1$ and $\omr w_2$ is equal to the number of lines containing $\omr e_1$ and $\omr e_2$. Let $\mathfrak{m}$ be this number.

 \medskip 
\noindent
{\it Claim 2: Let $\delta=1$ if $n\geq3$, and $\delta=2$ if $n=2$. Then a line $L\in\mathcal{L}$ contains $\omr e_1$ and $\omr e_2$ if and only if $L=L_{e_1,\lambda e_2}$ for some $\lambda\in\omr[\delta r]$. }

  \medskip 
\noindent   
     {\it Proof of Claim 2:}
Let $L\in\mathcal{L}$ be an arbitrary line containing $\omr e_1$ and $\omr e_2$.
By Claim 1  there is an element $g$ of $\GL_n(q)$  with $\det(g)\in\omr[t]$, mapping  $L_{e_1,e_2}$ to $L$ and  fixing $\omr e_i$ for $i=1,2.$
In particular $(e_i)g=\mu_i e_i$ for some $\mu_i\in \omr$ for $i=1,2.$ Then by Lemma~\ref{lem:lineform}(3,4), $L=L_{(e_1)g,(e_2)g}= L_{\mu_1 e_1,\mu_2 e_2}=L_{ e_1,\mu_1^{-1}\mu_2 e_2}$. Thus $L=L_{e_1,\lambda e_2}$ with $\lambda=\mu_1^{-1}\mu_2\in\omr$. This is of the  form in Claim 2 if $n\geq3$. So suppose that $n=2$. Then as $\det(g)=\mu_1\mu_2\in \omr[t]$, we also have  $\mu_1\mu_2\in\omr[t]\cap\omr=\omr[kr]=\mathbb{F}_{q_0}^*.$
 Thus $\lambda =\mu_1^{-1}\mu_2 = \eta \lambda'$ with $\eta = (\mu_1\mu_2)^{-1}\in\mathbb{F}_{q_0}^*$ and $\lambda'=\mu_2^2\in \omr[2r]$, and by Lemma~\ref{lem:lineform}(5),  $L= L_{ e_1,\lambda e_2}= L_{ e_1,\eta\lambda' e_2}=L_{ e_1,\lambda' e_2}$, as in Claim 2.

 Conversely, let $L=L_{e_1,\lambda e_2}\in\widehat{\mathcal{L}}$  with $\lambda \in \omr[\delta r]$. Then $L$ contains $\omr e_1$ and $\omr \lambda e_2=\omr e_2$, and it remains to prove that $L\in\mathcal{L}$. If $n\geq3$, then $\widehat{\mathcal{L}}=\mathcal{L}$ by Lemma~\ref{lf7}(2), and hence Claim 2 is proved in this case. So suppose that $n=2$. Here $\lambda=\mu^2$ for some $\mu\in\omr$, and the diagonal matrix $g=\diag(\mu^{-1}, \mu)\in \GL_2(q)$ with $\det(g)=1\in\omr[t]$. Moreover, by Lemma~\ref{lem:lineform}(3,4), $g$ maps $L_{e_1,e_2}$ to \[
 L_{(e_1)g,(e_2)g}=L_{\mu^{-1}e_1,\mu e_2} = L_{e_1, \mu^2 e_2} = L_{e_1,\lambda e_2}=L,
 \]
 and hence $L\in \mathcal{L}$ by Definition~\ref{con:subfield}. Thus Claim 2 is proved.

 \medskip

Consider two lines $L_1,L_2$ containing $\omr e_1$ and $\omr e_2$. Then by Claim 2, for $i=1,2$, $L_i=L_{e_1,\lambda_i e_2}$ for some $\lambda_i\in\omr[\delta r]$ with $\delta$ as in Claim 2.   By Lemma~\ref{lem:lineform}(5), $L_1=L_2$ if and only if $\lambda_1^{-1}\lambda_2\in \mathbb{F}_{q_0}^*$. 
Thus the number $\mathfrak{m}$ is equal to the number of choices of $\lambda$ in $\omr[\delta r]$   which are pairwise distinct modulo $\mathbb{F}_{q_0}^* = \omr[kr]$. If $n \geq 3$ then $\delta=1$, and the set $\{\omega^{ir}|0\leq i<k\}$ contains exactly one representative of each coset of $\omr[kr]$ in $\omr$. Hence in this case $\mathfrak{m}=k$ and the main statement of the theorem is proved in this case. We now consider the case $n=2$, so $\delta=2$.

If $k$ is even, then  $\omr[kr]\leq \omr[2r]=\omr[\delta r]<\omr$ and  there are $k/2$ cosets of $\mathbb{F}_{q_0}^* = \omr[kr]$ in $\omr[2r]$, giving exactly $\mathfrak{m}=k/2$ lines as stated in the theorem.

Finally assume that  $k$ is odd.
Here the choices are $\lambda=\omega^{2ir}$ for $0\leq i<k$. Note that these $k$ field elements are pairwise  distinct modulo $\omr[kr]$, for if $\omega^{2ir}=\omega^{2jr}$ modulo $\mathbb{F}_{q_0}^* = \omr[kr]$, with $0\leq i\leq j<k$ (so that $0\leq j-i<k$), then $\omega^{2(j-i)r}\in \omr[kr]$, and so $k$ divides $2(j-i)$. Since $k$ is odd, this implies that   $k$ divides $j-i$ and hence $i=j$.
Thus $\mathfrak{m}=k$ as stated in the theorem.

 This completes the proof of the main assertion of the theorem. It follows in particular that 
 $\mathcal{D}$ is a partial linear space if and only if $\mathfrak{m}=1$ which, in turn, holds if and only if either $n=2$ with $k\leq 2$, or $n\geq 3$ with $k=1$. 
 Note that the line-size is $q_0+1\geq 3$ by Lemma~\ref{lem:lineform}(1).
Moreover, $r>1$ so
 $\omr e_1$ and $\omr \omega e_1$ are distinct points of $\Omega$, which are  not contained in a common line by Lemma ~\ref{lem:lineform}(2).
Thus $\mathcal{D}$ is a proper  partial linear space.
\end{proof}

\subsection{Definitions and properties of  $\mathrm{DLSub}(q,q_0,r,j)$}

\begin{Def}\label{DLSdef}
    Using the notation of Definition \ref{con:subfield} for $\mathrm{LSub}({ 2},q,q_0,r)= (\Omega, \mathcal{L})$, we define, for each $j$, the set 
    $$
    \omega^j \mathcal{L}:=\{L_{u,v} \mid  (u,v)=(\omega^j e_1,e_2)^g, g \in \GL_2(q), \det(g) \in \omr[t] \}.
    $$
    Then, for $0<j< t$, define $\mathrm{DLSub}(q,q_0,r,j)$ to be the pair $(\Omega, \mathcal{L} \cup \omega^j \mathcal{ L})$.
\end{Def}

\begin{Remark}\label{rem:DLSub}
\phantom{1}
    \begin{enumerate}
        \item[(a)] Using the fact that ${\diag(\omega^j,1)}$ normalises the group $G_1=\{g\in\GL_n(q)\mid \det(g)\in\omr[t]\},$ we see that 
\begin{align*}
    \omega^j\mathcal{L} & = \{L_{u,v} \mid (u,v)=(\omega^j e_1,e_2)^g, g \in \GL_2(q), \det(g) \in \omr[t]\} \\ & = \{L_{u,v} \mid (u,v)=(e_1,e_2)^{{\diag(\omega^j,1)}g}, g \in \GL_2(q),  \det(g) \in \omr[t]\}\\
    &= \{L_{u,v} \mid (u,v)=(e_1,e_2)^{g\,{\diag(\omega^j,1)}}, g \in \GL_2(q),  \det(g) \in \omr[t]\}=\mathcal{L}^{\diag(\omega^j,1)}.
\end{align*} 
 Further, if $j\equiv j'\pmod{t}$, then ${\diag(\omega^{j-j'},1)}$ has determinant in $\omr[t]$, and as ${\diag(\omega^j,1)} = {\diag(\omega^{j-j'},1)}\cdot {\diag(\omega^{j'},1)}$, it follows that 
$$
 \omega^j\mathcal{L} = \mathcal{L}^{\diag(\omega^j,1)} = \mathcal{L}^{\diag(\omega^{j'},1)} =  \omega^{j'}\mathcal{L}.
$$
Thus in order to run through the distinct sets $ \omega^j\mathcal{L}$ we may assume that $0<j\leq t$. This also implies that $\omega^t\mathcal{L}=\omega^0\mathcal{L}=\mathcal{L}$,
and hence $\mathrm{DLSub}(q,q_0,r,t)=\mathrm{LSub}({2},q,q_0,r)$.
In the definition of $\mathrm{DLSub}(q,q_0,r,j)$, the condition $j<t$ is used to exclude this trivial case. 

\item[(b)] 
It is also possible that different values of $j$ in the interval $(0,t)$ give rise to isomorphic partial linear spaces. 
For example, from part (a) it follows that $\diag(\omega^{-j},1)$ maps the line set $\mathcal{L}\cup \omega^j\mathcal{L}$ of $\mathrm{DLSub}(q,q_0,r,j)$ to $\omega^{-j}\mathcal{L}\cup \mathcal{L}$, which is the line set of $\mathrm{DLSub}(q,q_0,r,t-j)$, and hence  $\mathrm{DLSub}(q,q_0,r,j)$ and $\mathrm{DLSub}(q,q_0,r,t-j)$ are isomorphic.

\item[(c)] Finally we note that, since $\widehat{\mathcal{L}}=\mathcal{L}$ when $n\geq 3$ (see Lemma~\ref{lf7}(2)), this `doubled' incidence structure is only defined in dimension $n=2$.
    \end{enumerate}   
\end{Remark}

\begin{Lem}\label{lem:DLS}
Using the notation of Definitions~$\ref{con:subfield}$ and~$\ref{DLSdef}$, 
    $\mathrm{DLSub}(q,q_0,r,j)$ is a partial linear space if and only if 
 $k=2$,  $r$ is even, and $j \ne r_2$. Moreover, if $\mathrm{DLSub}(q,q_0,r,j)$ is a partial linear space, then  $\mathcal{L} \cap \omega^j\mathcal{L}=\emptyset$, and in this case $\mathrm{DLSub}(q,q_0,r,j)$ is a proper partial linear space.  
\end{Lem}

\begin{proof}   
Since $0<j<t$, the integer $t\geq 2$ and hence $k\neq 1$ by Lemma~\ref{lem:lineform}(6).
  Also, for $\mathrm{DLSub}(q,q_0,r,j)$ to be a partial linear space it is necessary for $\mathrm{LSub}(2,q,q_0,r)$ to be a partial linear space, and by Theorem~\ref{th:LSub}, the latter holds if and only if the number 
  $k=2$ (since we cannot have $k=1$).
   Assume therefore that $k=2$. Then $q$ is odd, and by Lemma \ref{lem:lineform}(7), $t=2 \cdot r_2$, where $r_2$ is the $2$-part of $r$. 

 Suppose first that $r$ is odd, so $r_2=1$, $t=2$ and hence $j=1.$ Let $b=\diag(\omega^{r-1},1)$. Then $\det(b)=\omega^{r-1} \in \omr[2]=\omr[t]$, so that 
 $(L_{\omega e_1, e_2})^b=L_{\omega^r e_1, e_2}\in \omega \mathcal{L}$ (by Definition~\ref{DLSdef}). However, by Lemma \ref{lem:lineform}(4,5), $L_{\omega^r e_1, e_2}\ne L_{e_1, e_2}$ since $\omega^r\notin \mathbb{F}_{q_0}^*=\omr[2r]$, and the pair $\{\omr e_1, \omr e_2 \} \subseteq L_{e_1,e_2} \cap L_{\omega^r e_1, e_2}$. Thus  $\mathrm{DLSub}(q,q_0,r,j)$ is not a partial linear space if $r$ is odd. 
 
 Suppose next that $r$ is even with $j=r_2$. Then $j=t/2$ since $t=2\cdot r_2$. 
 In this case let $b= \diag(\omega^{r-j},1)$, so that $L_{\omega^r e_1 ,e_2}=(L_{\omega^{j}e_1,e_2})^b$.  We claim that $\det(b)=\omega^{r-j}\in \omr[t]$. Indeed,  $r-j=r-r_2=r_2(r'-1)$ where $r'$  is odd. If $r'=1$ then $r-j=0$ and $\det(b)=1\in \omr[t]$ as required. So assume that $r'>1$. Then $r-j=r_2(r'-1)=t\cdot (r'-1)/2$ and $(r'-1)/2$ is a positive integer; hence $r-j$ is a multiple of $t$, and again $\det(b)=\omega^{r-j}\in \omr[t]$, proving the claim.  Thus by Definition~\ref{DLSdef}, $L_{\omega^r e_1 ,e_2}=(L_{\omega^je_1,e_2})^b\in \omega^j\mathcal{L}$. Now the same argument as for $r$ odd shows that $L_{\omega^r e_1, e_2}\ne L_{e_1, e_2}$ and both lines contain $\{\omr e_1, \omr e_2 \}$, so $\mathrm{DLSub}(q,q_0,r,j)$ is not a partial linear space.

 Finally suppose that  $r$ is even and $j \ne r_2$. Let us show that $\mathrm{DLSub}(q,q_0,r,j)$ is a partial linear space.
  Let $\mathcal{D}_1=(\Omega, \mathcal{L})$ and  $\mathcal{D}_2=(\Omega, \omega^j \mathcal{L})$ so, by Remark~\ref{rem:DLSub}(a), the matrix  
     $\diag(\omega^j,1)$ induces an isomorphism from $\mathcal{D}_1$ to $\mathcal{D}_2$. 
It then  follows from Theorem \ref{th:LSub} that both  $\mathcal{D}_1$ and $\mathcal{D}_2$ are proper partial linear spaces. 

  Let $\omr w_1,\omr w_2$ be two collinear points of $\mathrm{DLSub}(q,q_0,r,j)$. 
  If no line of $\mathcal{L}$ contains both of these two points, then there is a line of $\omega^j\mathcal{L}$ containing them, and this line is unique since $\mathcal{D}_2$ is a partial linear space.
  Thus assume that $\omr w_1,\omr w_2\in L$ for some line $L\in\mathcal{L}$ (which is unique since $\mathcal{D}_1$ is a partial linear space). We need to show that $\omr w_1,\omr w_2$ do not lie in any line of $\omega^j\mathcal{L}$.  By Claim 1 of the proof of Theorem \ref{th:LSub}, we may assume that $L=L_{e_1,e_2}$ and $\omr w_1=\omr e_1,\omr w_2=\omr e_2$.
  
  Suppose that $\omr e_1,\omr e_2$ are contained in a line $L_{u,v}\in\omega^j\mathcal{L}$.
Then $(u,v)=(\omega^je_1,e_2)^g$ for some $g\in \GL_2(q)$ with $\det(g)\in\omr[t]$, and so $(\omega^je_1,e_2)^g=(e_1,e_2)^h$, where $h=\diag(\omega^j,1)g$, and $L_{u,v}=(L_{e_1,e_2})^h$.
  Thus $\omr (e_1)^{h^{-1}},\omr (e_2)^{h^{-1}}\in L_{e_1,e_2}=L$. By Lemma \ref{lf7}, there exists $g_0\in \GL_2(q_0)$ fixing $L$ and mapping $\omr e_i$ to $\omr e_i^{h^{-1}}$, for $i=1,2$.
  Therefore $L^{g_0h}=L^h=L_{u,v}$ and $\omr e_i^{g_0h}=(\omr e_i^{h^{-1}})^h=\omr e_i$, for $i=1,2$.
  This means that $e_i^{g_0h}=\mu_i e_i$ for some $\mu_i\in \omr$,for $i=1,2$.
  In particular $\mu_1\mu_2=\det(g_0h)=\det(g_0)\det(\diag(\omega^j,1)) \det(g) =\omega^j\det(g_0)\det(g)$. Recall that $\det(g_0)\in \mathbb{F}_{q_0}^*=\omr[2r]\leq \omr[t]$ while $\det(g)\in\omr[t]$. Thus $\mu_1\mu_2\in \omr\cap \omega^j\omr[t]$, so in particular $\omr\cap \omega^j\omr[t]\ne\emptyset$.  Since $r_2$ divides both $r$ and $t$, this condition implies that $r_2$ divides $j$. However as $0<j<t=2\cdot r_2$, we must have $j=r_2$, which is a contradiction. 
  
  Thus $\omr e_1,\omr e_2$ are not contained in any line of $\omega^j\mathcal{L}$, and hence $L$ is the unique line of $\mathcal{L}\cup\omega^j\mathcal{L}$ containing both of these points. In particular the line sets $ \mathcal{L}$ and $\omega^j \mathcal{L}$ are distinct orbits of the subgroup consisting of all $g\in \GL_2(q)$ with $\det(g)\in\omr[t]$, so $ \mathcal{L} \cap \omega^j \mathcal{L} = \emptyset.$
  Moreover we have shown that $\mathrm{DLSub}(q,q_0,r,j)$ is a partial linear space in this case, and therefore the first two assertions of the lemma are proved. 
  
  Finally we show that $\mathrm{DLSub}(q,q_0,r,j)$ is a proper partial linear space.   
  Note that  $|L|=q_0+1\geq 3$ by Lemma~\ref{lem:lineform}(1) (and by definition lines in $\omega^j\mathcal{L}$ have the same size as lines in $\mathcal{L}$). Since $r\geq r_2>1$,
 $\omr e_2$ and $\omr \omega e_2$ are distinct points of $\Omega$. They  are  not contained in a line of $\mathcal{L}$ by Lemma ~\ref{lem:lineform}(2). Suppose they are contained in a line of $\omega^j\mathcal{L}=\mathcal{L}^{\diag(\omega^j,1)}$. Then the two points 
 $$
 (\omr e_2)^{\diag(\omega^{-j},1)}=\omr e_2\quad \text{and}\quad (\omr \omega e_2)^{\diag(\omega^{-j},1)}=\omr \omega e_2
 $$ 
 are contained in a common line of $\mathcal{L}$, contradicting Lemma ~\ref{lem:lineform}(2).
Thus $\mathrm{DLSub}(q,q_0,r,j)$ is a proper  partial linear space. 
\end{proof}

\section{Partial linear spaces arising from $\GammaU_3(q)$.}\label{sec:spacesU}

Now we consider two families of partial linear spaces arising from an analogous action of $\GammaU_3(q)$.  We briefly recall the definitions and structures of these unitary groups and spaces; for more details see \cite[\S 2.3]{KL}.
Let $V={(\mathbb{F}_{q^2})}^3$, the space of $3$-dimensional row vectors and the natural module for $\GU_3(q)$, and let $(\cdot, \cdot)$ be the corresponding non-degenerate unitary form on $V$, so $\GU_3(q)=\{g \in \GL_3(q^2) \mid (ug,vg)=(u,v) \text{ for all } u,v \in V\}.$ Then $V$ has a basis  $\{e, x, f\}$ such that
\begin{equation} \label{unbasis}
(e,e)=(f,f)=(x,e)=(x,f)=0, \text{ } (x,x)=(e,f)=1,    
\end{equation}
and we write the elements of  $\GU_3(q)$ relative to this basis.
In particular, the Gram matrix of $(\cdot, \cdot)$ in this basis is 
\begin{equation}\label{e:P}
P=\begin{pmatrix}
0 & 0&1 \\0&1&0\\1&0&0
\end{pmatrix}.
\end{equation}    
 
Hence, by \cite[Lemma 2.1.8]{KL}, we may identify $\GU_3(q)$ with the group of matrices
$$
\GU_3(q)=\{ A \in \GL_3(q^2) \mid AP\bar{A}^T=P\} 
$$
where ${A}^T$ is the transpose matrix of $A$ and $\bar{A}= (a_{ij}^q)$ for $A=(a_{ij}).$ Then $\SU_3(q)$ is the group $\{g \in \GU_3(q) \mid \det(g)=1\}= \GU_3(q)\cap \SL_3(q^2).$

Let $Z$ be the subgroup of all scalar matrices of $\GL_3(q^2).$
We define $\GammaU_3(q)$ as the group of semisimilarities of the unitary space $V$, that is to say, $g \in \GammaU_3(q)$ if and only if $g \in \GammaL_3(q^2)$ and there exists $\lambda \in \mathbb{F}_{q^2}^*$ and $\alpha \in \Aut (\mathbb{F}_{q^2})$ such that 
$$
(v g, u g)= \lambda (v,u)^{\alpha} \text{ for all } v,u \in V.
$$
By \cite[\S 2.3]{KL}, $\GammaU_3(q) = (Z \, \GU_3(q)) \rtimes \langle \phi \rangle $ where   $\phi$ is $\phi_{\{e,x,f\}}$ as defined in \eqref{def:phi}.

\medskip

The following lemma on the trace of a finite field extension is useful, when dealing with unitary groups.  Here $\omega$ is  a primitive element  in  $\mathbb{F}_{q^2}$.

\begin{Lem}\label{l:trace}
Define the trace map $\mathrm{Tr}:\mathbb{F}_{q^2} \to \mathbb{F}_q$ by
$$
\mathrm{Tr}: \delta \mapsto \delta+\delta^q.
$$
\begin{enumerate}[label=\normalfont (\arabic*)]
    \item $\mathrm{Tr}$ is a linear transformation from $\mathbb{F}_{q^2}$ onto $\mathbb{F}_q$ where both $\mathbb{F}_{q^2}$ and $\mathbb{F}_q$ are viewed as vector spaces over $\mathbb{F}_q.$
    \item $|\mathrm{ker}(\mathrm{Tr})|=q$. In particular,
    $$\mathrm{ker}(\mathrm{Tr}) =\begin{cases}
        \omega^{(q+1)/2}\mathbb{F}_q & \text{ if $q$ is odd};\\
        \mathbb{F}_q & \text{ if $q$ is even}.
    \end{cases}$$
    \item For each $\delta \in \mathbb{F}_q$, the equation $\mathrm{Tr}(y)=\delta$ has exactly $q$ solutions.
\end{enumerate}
\end{Lem}
\begin{proof}
    Part $(1)$ is \cite[Theorem 2.23$(iii)$]{Ffields}. Parts $(2)$ and $(3)$ follow from part $(1)$ and direct calculation.
\end{proof}

\begin{construction}\label{con:psu}
Let $\mathcal{T}$ be the set of all $v \in V^*$ such that $(v,v)=0.$ Let $\binom{\mathcal{T}}{1}$ denote the set of totally isotropic 1-subspaces of $V$, so $\langle v \rangle \in \binom{\mathcal{T}}{1}$ if and only if $v \in \mathcal{T}$.
Let  $q =p^a$ for a prime $p$ and $a\geq 1$ such that $q \geq 3$, and let $\omega$ be a primitive element of $\mathbb{F}_{q^2}$. Let $r>1$ be an integer dividing $q^2-1$. Let 
$$\Omega=\{\omr v \mid v \in \mathcal{T} \}.$$  Then the natural induced action of $\GammaU_3(q)$ on $\Omega$ given by 
\begin{equation}\label{e:actionU}
    g: \omr v\mapsto \omr (v)g, \ \mbox{for $g\in\GammaU_3(q)$ and
    $ v\in \mathcal{T}$}
\end{equation}
induces a permutation group $\ff=\GammaU_3(q)/Y$ on $\Omega$, where $Y=\langle \omega^r I\rangle$. The group $\ff$ stabilises the partition $\Sigma$ of $\Omega$ defined by 
\[
\Sigma=\{\sigma(U) \mid U\in  \tbinom{\mathcal{T}}{1}\},\ \text{ where for }   U=\langle u \rangle, \  \sigma(U)=\{\langle \omega^r\rangle \omega^i u\mid 0\leq i<r\}.
\]
\end{construction}

We next show that the group actions in Construction~\ref{con:psu} are semiprimitive. 

 \begin{Lem}\label{l:semiU}
   Let $q,r, Y, \Omega, \overline{G}$ be as in Construction~$\ref{con:psu}$,
   and suppose that $G$ is a group satisfying $Y\,\SU_3(q)/Y\leq G \leq \ff$. Then the following hold. 
   \begin{enumerate}[label=\normalfont (\arabic*)]
       \item $G$ is semiprimitive on $\Omega$;
       \item $G$ is innately transitive on $\Omega$ if and only if $r$ divides $(q^2-1)/(3,q+1)$; and in this case $Y\,\SU_3(q)/Y\cong \PSU_3(q)$ is the unique plinth of $G$; 
       \item $G$ is quasiprimitive on $\Omega$ if and only if $r$ divides $(q^2-1)/(3,q+1)$ and $G\cap (Z/Y)=1$;
       \item $ G$ has rank $3$ on $\Omega$ if and only if $ G \geq Z \,\SU_3(q)/Y$,  $r$ divides $q-1$,  $r$ is an odd primitive prime divisor of $p^{r-1}-1$  and $(r-1, \frac{2a}{|G:G\cap\GL_n(q^2)|})=1$. 
         Moreover, if $G$ has rank 3, then $r-1$ divides $a$. 
   \end{enumerate}
 \end{Lem}

\begin{proof}
The subgroup $Y=\langle \omega^r I\rangle$ has index $r$ in the centre $Z:= \langle \omega I\rangle$ of $\GL_3(q^2)$, $|Z\cap\GU_3(q)|=q+1$, and so $Z\cap\SU_3(q)$ has order $(3,q+1)$. Now $Y\,\SU_3(q)/Y\leq G \leq \ff$ and there is a unique subgroup $H$ of $\GammaU_3(q)$ such that $Y\leq H$ and $H/Y=G$. Further, it is well known and easy to check that every normal subgroup $N$ of $H$ satisfies either $N\leq Z$ or $N\geq \SU_3(q)$, and so every normal subgroup $K$ of $G$ satisfies either $K\leq Z/Y$ or $K\geq Y\,\SU_3(q)/Y$.  Since $Z/Y$ is semiregular on $\Omega$, it follows that each normal subgroup $K\leq Z/Y$ is semiregular on $\Omega$. Also since $\SU_3(q)$ is transitive on $\mathcal{T}$ it follows that $Y\,\SU_3(q)/Y$ is transitive on $\Omega$ and hence each normal subgroup $K\geq Y\,\SU_3(q)/Y$ is transitive on $\Omega$. Thus $G$ is semiprimitive on $\Omega$, proving part $(1)$.

Now $G$ is innately transitive if and only if it has a transitive minimal normal subgroup $K$. Since $Z/Y$ is not transitive on $\Omega$, the only possibility for such a transitive minimal normal subgroup is $K=Y\,\SU_3(q)/Y\cong \SU_3(q)/(Y\cap\SU_3(q))$. Note that $\SU_3(q)/(Z\cap\SU_3(q))=\PSU_3(q)$ (which is a non-Abelian simple group since $q\geq3$), and that $Y\cap\SU_3(q) \leq Z\cap\SU_3(q)$. Thus $K=Y\,\SU_3(q)/Y$ is a minimal normal subgroup if and only if $Y\cap\SU_3(q) = Z\cap\SU_3(q)$, and this holds if and only if $Y$ contains the unique subgroup $Z\cap\SU_3(q)$ of $Z$ of order $(3,q+1)$. In turn, this holds if and only if $|Y|=(q^2-1)/r$ is divisible by $(3,q+1)$, that is to say, $r$ divides $(q^2-1)/(3,q+1)$. 
Also in this case $K=Y\,\SU_3(q)/Y\cong\PSU_3(q)$  is the unique transitive minimal normal subgroup, that is, the unique plinth  of $G$,
and part (2) is proved.

The group $G$ is quasiprimitive if and only if it has no nontrivial intransitive normal subgroups, and this is the case if and only if $G$ is innately transitive, and the subgroup $Z/Y$ intersects $G$ trivially, proving part $(3)$. 

Finally we pull together various results to determine when $G$ can be rank $3$. Firstly, by  \cite[Lemmas 5.2 and 5.3]{semiprR3}, if  $G$ is not innately transitive, then it does not have rank $3$. If $G$ is quasiprimitive, then it does not have rank $3$ by \cite[Proposition 4.9]{DGLPP}. Hence, to have rank $3$, $G$ must be properly innately transitive. Finally, by  \cite[Theorem C and Table 3]{R3PIT}, $G$ has rank $3$ if and only if the conditions of $(4)$ hold. 

 The conditions in part $(4)$ imply that $r-1$ divides $a$. Indeed,  since $r$ also divides $q-1=p^a-1$ the integer $r$ divides $(p^a-1,p^{r-1}-1)=p^{(a,r-1)}-1$, and by the definition of a primitive prime divisor it follows that $(a,r-1)=r-1$, that is, $r-1$ divides $a$.  This completes the proof of part $(4).$ 
\end{proof}

The following lemma which identifies the stabiliser of a nondegenerate $2$-subspace will be useful in our analyses.

\begin{Lem}\label{l:2substab}
    Let  $U$ be the subspace $\langle e,f\rangle$   and let $g\in\GL_3(q^2)$. Then $g\in \left(Z\,\SU_3(q)\right)_{U}$ if and only if $g$ has the form given in \eqref{eq:2substab} and conditions (a) and (b) both hold:
    \begin{enumerate}
    \item[(a)]  \begin{equation}\label{eq:2substab}
g= \begin{pmatrix}
    \nu_{11} & 0 & \nu_{12} \\
    0 & \delta & 0 \\
    \nu_{21} & 0 & \nu_{22}
\end{pmatrix} \textrm{ for some }  \delta \in \mathbb{F}_{q^2}^*, \nu_{ij}\in \mathbb{F}_{q^2} \textrm{ with }\det( \nu_{ij}) \ne 0, 
\end{equation}
such that 
    $
    \delta^{q+1} 
    = \nu_{11}\nu_{22}^q+ \nu_{12}\nu_{21}^q
    $ and $\mathrm{Tr}(\nu_{i1}\nu_{i2}^q)=0$ for $i=1,2;$
    \item[(b)] there exists $\lambda\in\mathbb{F}_{q^2}^*$ such that $\delta^{q+1} = \lambda^{q+1}$ and $ \delta \cdot \det( \nu_{ij})= \lambda^3$.
    \end{enumerate}
    
    Moreover in this case
       $g=(\lambda I)\cdot g_1$ with $g_1\in \SU_3(q)_{U}$.
 \end{Lem}

\begin{proof}
    Suppose first that $g\in Z\,\SU_3(q)$ and $g$ fixes the subspace $U=\langle e,f\rangle$. Then   $g$ leaves invariant the unitary form on $V$ modulo scalars, and hence $g$  also leaves $U^\perp =\langle x\rangle$ invariant. Thus  $g$ is as in \eqref{eq:2substab} with the entries $\delta, \nu_{ij}$ satisfying the conditions in \eqref{eq:2substab}. As $g\in Z\,\SU_3(q)$ we may write $g=(\lambda I)\cdot g_1$ with $\lambda\in\mathbb{F}_{q^2}^*$ and
    \[ g_1= \begin{pmatrix}
    \lambda^{-1} \nu_{11} & 0 &\lambda^{-1}  \nu_{12} \\
    0 & \lambda^{-1} \delta & 0 \\
    \lambda^{-1} \nu_{21} & 0 & \lambda^{-1} \nu_{22} 
\end{pmatrix}\in\SU_3(q)_U.
    \]
 Thus 
\[
1=(e,f)=((e)g_1, (f)g_1) = (\lambda^{-1} \nu_{11} e+\lambda^{-1} \nu_{12} f, \lambda^{-1} \nu_{21} e+\lambda^{-1} \nu_{22} f) = \lambda^{-1-q} (\nu_{11}\nu_{22}^q+ \nu_{12}\nu_{21}^q).
\]
Similarly 
\[1=(x,x)=((x)g_1,(x)g_1)=(\lambda^{-1} \delta x, \lambda^{-1} \delta x)=\lambda^{-1-q} \delta^{1+q}.\] Hence 
 $\delta^{q+1} = \lambda^{q+1} = \nu_{11}\nu_{22}^q+ \nu_{12}\nu_{21}^q$. 
 Also \[0=(e,e)=((e)g_1,(e)g_1)=
 \lambda^{-1-q} (\nu_{11}\nu_{12}^q+\nu_{11}^q\nu_{12})=\lambda^{-1-q} \mathrm{Tr}(\nu_{11}\nu_{12}^q).\] Thus $\mathrm{Tr}(\nu_{11}\nu_{12}^q)=0$ since $\lambda\neq 0$.
 The same argument using $0=(f,f)=((f)g_1,(f)g_1)$ yields that $\mathrm{Tr}(\nu_{21}\nu_{22}^q)=0.$
 Also, $1=\det(g_1)=\lambda^{-3}\delta \det( \nu_{ij})$,  so $\delta \cdot \det( \nu_{ij})=\lambda^3$, and hence conditions (a) and (b) hold. 
 
 Conversely suppose that  $g$ is as in \eqref{eq:2substab} and $\delta, \nu_{ij},\lambda$ satisfy all the conditions in parts (a) and (b).   A straightforward but tedious computation shows that $gP\bar{g}^T=\delta^{q+1}P$, with $P$ the Gram matrix of the unitary form in \eqref{e:P}, and since $\delta^{q+1}=\lambda^{q+1}$, we have $gP\bar{g}^T=\lambda^{q+1}P$. It follows, writing $g=(\lambda I)\cdot g_1$, that the element $g_1$ satisfies $g_1P\bar{g_1}^T=P$, that is to say, $g_1\in\GU_3(q)$. Further, $\det(g_1)=\lambda^{-3}\det(g)= \lambda^{-3}\delta \det( \nu_{ij})=1$ by part (b), and hence $g_1\in\SU_3(q)$. Thus $g\in Z\,\SU_3(q)$. Finally, it is clear from \eqref{eq:2substab} that $g$ leaves $U$ invariant.   

\end{proof}

\begin{Lem}\label{l:2stab}
    Let $\alpha_1=\omr e,$ $\alpha_2=\omr f\in\Omega$. Then $g\in \left(Z\,\SU_3(q)\right)_{\alpha_1,\alpha_2}$ if and only if the following two conditions both hold:
    \begin{enumerate}
    \item[$(a)$] $g=\diag(\mu_1, \delta, \mu_2)$ for some $ \mu_1, \mu_2 \in \omr,\ \delta\in\mathbb{F}_{q^2}^*$ such that 
    $$
    \delta^{q+1} 
    = \mu_1\mu_2^q \in\omr \cap \omr[q+1];
    $$
    \item[$(b)$] there exists $\lambda\in\mathbb{F}_{q^2}^*$ such that $\delta^{q+1} = \lambda^{q+1}$ and $ \mu_1\mu_2\delta= \lambda^3$.
    \end{enumerate}
    
    Moreover in this case
    $g=(\lambda I)\cdot g_1$ with $g_1=\diag(\lambda^{-1} \mu_1, \lambda^{-1} \delta, \lambda^{-1} \mu_2)\in\SU_3(q)_{\alpha_1,\alpha_2}$.
\end{Lem}

\begin{proof}
Suppose first that $g\in Z\,\SU_3(q)$, fixing $\omr e$ and $\omr f$. Then  $g$ leaves $U=\langle e,f\rangle$ invariant, and hence $g$ is as in \eqref{eq:2substab} and the conditions $(a)$ and $(b)$ of Lemma~\ref{l:2substab} hold. 
The facts that $g$ fixes $\omr e$, while $eg=\nu_{11}e + \nu_{12}f$, imply that $\nu_{12}=0$ and $\mu_1:=\nu_{11}\in\omr$. Similarly the facts that $g$ fixes $\omr f$, while $fg=\nu_{21}e + \nu_{22}f$, imply that $\mu_2:=\nu_{22}\in\omr$ and $\nu_{21}=0$. Thus $g=\diag(\mu_1, \delta, \mu_2)$. 
The condition $\delta^{q+1} = \nu_{11}\nu_{22}^q+ \nu_{12}\nu_{21}^q$  from  Lemma~\ref{l:2substab}$(a)$ gives 
$\delta^{q+1} = \mu_1\mu_2^q\in\omr$, and clearly $\delta^{q+1}\in\omr[q+1]$, so part $(a)$ holds. Also part $(b)$ follows from   Lemma~\ref{l:2substab}$(b)$. In this case the last assertion of this lemma follows from the last assertion of Lemma~\ref{l:2substab}.

Conversely, suppose that $g= \diag(\mu_1, \delta, \mu_2)$ and $\lambda$ satisfy all the conditions in parts (a) and (b). 
These conditions imply all the conditions in parts (a) and (b) of Lemma~\ref{l:2substab}, with $\nu_{ii}=\mu_i$ for $i=1,2$ and $\nu_{12}=\nu_{21}=0.$ Thus by Lemma~\ref{l:2substab}, $g\in \left(Z\,\SU_3(q)\right)_{U}$. Further,  $g$ fixes $\omr e$ and $\omr f$, and so $g\in \left(Z\,\SU_3(q)\right)_{\alpha_1,\alpha_2}$.
\end{proof}

\subsection{Definition and properties of $\mathrm{USub}(q,q_0,r)$}

The first family of incidence structures is defined as follows.

 \begin{Def}\label{con:subfieldSU}
In the notation of Construction \ref{con:psu}, assume that  $q=q_0^b$ where $q_0$ is a prime power, $b > 1$ is an integer, and $r:=(q-1)/(q_0-1)$ is odd; in particular, $(r,q+1)=1.$  
Define the following sets:
\begin{itemize}
    \item $W= \omega^{\frac{r(q+1)}{(q+1,2)}} \mathbb{F}_{q_0}
    =\begin{cases}
        \mathbb{F}_{q_0} &\text{if $q$ is even};\\
        \omega^{\frac{r(q+1)}{2}} \mathbb{F}_{q_0} &\text{if $q$ is odd};
    \end{cases}$
    \item $L_{u,v}=\{\omr u\} \cup  \omr\{\lambda u + v \mid \lambda \in W \}$ for linearly independent $u,v \in \mathcal{T}$; 
    \item  $\mathcal{L}=\{L_{u,v} \mid  (u,v)=(e,f)^g, g \in Z \,  \SU_3(q)\}$; 
    \item $\mathrm{USub}(q,q_0,r)=(\Omega, \mathcal{L})$.
\end{itemize}  
\end{Def}

\begin{Remark}\label{r:subfieldSU}
    \begin{enumerate}
        \item[(a)]  Notice that
$\omega^{\frac{r(q+1)}{(q+1,2)}} =  \omega^{\frac{q^2-1}{(q+1,2)(q_0-1)}}$ has multiplicative order $(q+1,2)(q_0-1)$, and hence $\widehat{W}:= \omr[\frac{r(q+1)}{(q+1,2)}]$ contains $\mathbb{F}_{q_0}^*= \omr[r(q+1)]$ as a (multiplicative) subgroup of index $(q+1,2)$. 
Also $\widehat{W}\subseteq \omr[r]$ and $\mathbb{F}_{q_0}^*= \omr[r(q+1)]\leq \omr[q+1]=\mathbb{F}_{q}^*$.   The set  $W^*$ is closed under taking inverses, and if $q$ is odd, then $W^*$ is the nontrivial coset of $\mathbb{F}_{q_0}^*$ in the multiplicative group $\widehat{W}$ of order $2(q_0-1)$. 
Moreover, the product of two elements of $W$ lies in  $\mathbb{F}_{q_0}$, and the product of an element in $W$ with an element in  $\mathbb{F}_{q_0}$ is in $W$.

\item[(b)] It follows from part (a) that $W \cup \mathbb{F}_{q_0}=\widehat{W} \cup \{0\}$. If $q$ is even then this set is equal to $\mathbb{F}_{q_0}$. In the case where $q$ is odd, the right-hand expression is a disjoint union of $\{0\}$ and the subgroup $\widehat{W}$, and is therefore closed under multiplication by elements of $\mathbb{F}_{q_0}$, while the left-hand expression is a disjoint union of $\mathbb{F}_{q_0}$ and a multiplicative translate of it, with both sets closed under addition. Thus, for all $q$, $W \cup \mathbb{F}_{q_0}$ is an $\mathbb{F}_{q_0}$-vector space of dimension $(q-1,2)$, which is in addition invariant under the field automorphism $\phi$. 

\item[(c)] Let $\delta\in W$, so by Lemma~\ref{l:trace}, $\mathrm{Tr}(\delta) = \delta(1+\delta^{q-1})$. If $\delta\ne 0$ then, noting that $r$ is odd and $\mathbb{F}_{q_0}$ is a subfield of $\mathbb{F}_q$, we have
$$
\delta^{q-1}\in \left(\omega^{\frac{r(q+1)}{(q+1,2)}}\,\mathbb{F}_{q_0}^*\right)^{q-1} = \{ -1\}.
$$
Hence $\mathrm{Tr}(\delta) = 0$, so $W \subseteq \mathrm{ker}(\mathrm{Tr})$ and $\delta^q=-\delta $ for all $\delta\in W$. 
    \end{enumerate}
\end{Remark}

\begin{Lem} \label{lem:lineformSU}
With the assumptions of Definition \ref{con:subfieldSU}, the following hold:
    \begin{enumerate}[label=\normalfont (\arabic*)]
    \item $ L_{u,v}  = \omr\{\lambda_1 u + \lambda_2 v \mid \lambda_1, \lambda_2 \in W \cup \mathbb{F}_{q_0}, (\lambda_1,\lambda_2)\neq (0,0), \lambda_1\lambda_2 \in W \} $ 
    of cardinality $q_0+1$; \label{lfSU1} 

    \item if distinct points $\omr w_1, \omr w_2$ are collinear, then $w_1$ and $w_2$ are linearly independent;  \label{lfSUd}
    
    \item if $L_{u,v} \in \mathcal{L}$ and $g \in \GammaU_3(q)$, then $(L_{u,v})^g =L_{(u)g, (v)g} \in \mathcal{L}$;  in particular, $\GammaU_3(q)$ leaves $\mathcal{L}$ invariant;
    \item if $\mu \in \omr,$ then $L_{\mu u ,v }=L_{u, \mu^{-1}v};$ 
    \label{lfSU4}

    \item if $L_{u,v} \in \mathcal{L}$ and $\eta \in \mathbb{F}_{q^2}^*$, then $L_{u, \eta v}= L_{u,v}$ if and only if $\eta \in \mathbb{F}_{q_0}^*.$ 
    \end{enumerate}
\end{Lem}
\begin{proof}
 The proofs of parts \ref{lfSUd}  and \ref{lfSU4} are analogous to the proofs of Lemma \ref{lem:lineform}(2) and (4). 
 
   (1) Let us first check that $\lambda u+ v\in \mathcal{T}$ when $\lambda\in W$, so  $L_{u,v}$ is well defined. Indeed,
 \[(\lambda u+ v, \lambda u+ v)=\lambda^{q+1}(u,u)+(v,v)+(\lambda+\lambda^q)(u,v)=0\]
since $u,v\in \mathcal{T}$ and $\lambda\in W \subseteq \mathrm{ker}(\mathrm{Tr})$ (by Remark~\ref{r:subfieldSU}(c)).
  It follows that $|L_{u,v}|=|W|+1=q_0+1$.
  
  Set $M=\omr \{\lambda_1 u + \lambda_2 v \mid \lambda_1, \lambda_2 \in W \cup \mathbb{F}_{q_0}, (\lambda_1,\lambda_2)\neq (0,0), \lambda_1\lambda_2 \in W \}.$ Note that $L_{u,v} \subseteq M$ since the pairs $(1,0)$ and $(\lambda,1)$, where $\lambda \in W$, satisfy the conditions on $(\lambda_1, \lambda_2)$ in the definition of $M.$ Let $\alpha=\omr (\lambda_1u + \lambda_2 v) \in M$.  If $\lambda_1=0$ then $\alpha = \omr \lambda_2 v$ and $\lambda_2\in (W\cup \mathbf{F}_0)^*=\widehat{W}\subseteq \omr$ (using Remark~\ref{r:subfieldSU}), so $\alpha= \omr v\in L_{u,v}$. Similarly if $\lambda_2$ is zero, then $\alpha= \omr u\in L_{u,v}$. 
   Assume now that neither of the $\lambda_i$ is zero. Then $\alpha= \omr(\lambda_1 \lambda_2^{-1} u +v)$ by part \ref{lfSU4}, since $\lambda_2 \in \widehat{W}\subseteq \omr[r]$. Moreover, $\lambda_1 \lambda_2^{-1} = \lambda_2^{-2}\cdot \lambda_1 \lambda_2$ lies in $W$ since $\lambda_2^{-2}\in \omr[\frac{2r(q+1)}{(q+1,2)}]\leq \omr[r(q+1)] 
    =\mathbb{F}_{q_0}^*$,
   and $\lambda_1\lambda_2 \in W$ (by the definition of $M$), and since $W$ is closed under multiplication  by an element of $\mathbb{F}_{q_0}$ (by Remark~\ref{r:subfieldSU}). Hence $\alpha \in L_{u,v}$, and we conclude that $M=L_{u,v}.$  Thus part \ref{lfSU1} is proved.

       $(3)$ Since $L_{u,v} \in \mathcal{L}$, we have $L_{u,v}=L_{(e)h,(f)h}$ for some $h \in Z \, \SU_3(q)$ with $(u,v)=(e,f)^h$ by Definition \ref{con:subfieldSU}. Hence, for any $\lambda_1, \lambda_2 \in W \cup \mathbb{F}_{q_0}$ such that $(\lambda_1,\lambda_2)\neq (0,0), \lambda_1\lambda_2 \in W $, 
      $$
    \omr (\lambda_1 u + \lambda_2 v)=\omr (\lambda_1(e)h +\lambda_2(f)h)=(\omr (\lambda_1e +\lambda_2f))h,
    $$
    so $L_{u,v}=L_{(e)h,(f)h}=(L_{e,f})^h$ (by part (1)). Therefore, for $g \in \GammaU_3(q)$, we have $(L_{u,v})^g=(L_{e,f})^{hg}$ with $hg \in \GammaU_3(q)$ and it suffices to show that, for all $g \in \GammaU_3(q)$,
    $$(L_{e,f})^{g}=L_{(e)g,(f)g} \text{ and } L_{(e)g,(f)g} \in \mathcal{L}.$$

      Let  $g\in \GammaU_3(q)=(Z \, \GU_3(q)) \rtimes \langle \phi \rangle $ where  $\phi$ is as in \eqref{def:phi}, with respect to the basis \eqref{unbasis} of $V$.
      Note (or see \cite[Proposition 2.3.4]{KL}) that $\GU_3(q)=  \SU_3(q) \rtimes \langle \diag(1, \omega^{q-1},1) \rangle$.
      Hence $g= \phi ^j \diag(1, \omega^{t(q-1)},1) z g_0$ with $0\leq j<2a$, $0\leq t<q+1$, $z \in Z$ and $g_0 \in \SU_{3}(q)$. Further, $(e,f)^g=(e,f)^{zg_0}$ since $\phi$ and $\diag(1,\omega^{q-1},1)$ stabilise $e$ and $f$. So $L_{(e)g,(f)g}=L_{(e)zg_0,(f)zg_0} \in \mathcal{L}$ since $zg_0 \in Z \, \SU_3(q)$. It remains to show that 
      $(L_{e,f})^g=L_{(e)g,(f)g}.$

     Consider an element $(\omr (\lambda_1 e + \lambda_2 f))g \in (L_{e,f})^g$, that is, $\lambda_1, \lambda_2 \in W \cup \mathbb{F}_{q_0},$ $(\lambda_1,\lambda_2)\neq (0,0),$ $ \lambda_1\lambda_2 \in W$ by part (1). Note that  we may write $g=\phi^j g_1$ with $g_1\in Z \, \GU_3(q)$.
 Then, since the field automorphism $\phi$ normalises $\omr$ and leaves $e$ and $f$ invariant, it follows that $g$ normalises $\omr$ and we have  
 $$
 (\omr (\lambda_1 e + \lambda_2 f))g = (\omr (\lambda_1^{\phi^j}e +\lambda_2^{\phi^j}f))g_1  
 = \omr (\lambda_1^{\phi^j}(e)g_1 +\lambda_2^{\phi^j}(f)g_1)
 =    \omr (\lambda_1^{\phi^j}(e)g +\lambda_2^{\phi^j}(f)g).
    $$
This point lies on $L_{(e)g,(f)g}$  by part (1), and since each of $W$ and $W\cup \mathbb{F}_{q_0}$ is  preserved by $\phi$.
    Thus $(L_{e,f})^g \subseteq L_{(e)g, (f)g}$  and  equality holds since $|L_{e,f}|=|L_{(e)g, (f)g}|=q_0+1$ by part (1). Thus part (3) is proved.

   (5) Assume that $\eta\in \mathbb{F}_{q_0}^*$.
    Both lines $L_{u,\eta v}$ and $L_{u,v}$ contain the point $\omr u$. Each other point of  $L_{u,\eta v}$ is equal to 
$\omr(\lambda u+ \eta v)$ for some $\lambda\in W$. Since $\mathbb{F}_{q_0}^* \subseteq \omr$, we have $\omr(\lambda u+\eta v)=\omr (\lambda\eta^{-1}u+ v)$, and as $\lambda$ runs through the elements of $W$, so does $\lambda\eta^{-1}$ (recall $W$ is closed under  multiplication by 
 elements of $\mathbb{F}_{q_0}$). It follows  that $L_{u,\eta v}$ and $L_{u,v}$ contain exactly the same set of points, and hence are equal. 
Conversely assume that  $L_{u,\eta v}=L_{u,v}$ where  $\eta \in \mathbb{F}_{q^2}^*$. 
The point $\omr (\omega^{\frac{r(q+1)}{(q+1,2)}} u+\eta v)\in L_{u,\eta v}$ must then be equal to $\omr(\lambda u+ v)$ for some $\lambda\in W$. Since this point is distinct from $\omr v$ it follows that $\lambda\ne 0$. Thus $\lambda^{-1}\in W^*\subseteq \omr$ and also  $\omega^{\frac{r(q+1)}{(q+1,2)}}\in\omr$, and hence
$$
\omr (\omega^{\frac{r(q+1)}{(q+1,2)}} u+\eta v)=\omr(\lambda u+ v) =\omr(\omega^{\frac{r(q+1)}{(q+1,2)}}u+ \omega^{\frac{r(q+1)}{(q+1,2)}}\lambda^{-1}v)
$$ 
This implies that  $\eta= \omega^{\frac{r(q+1)}{(q+1,2)}}\lambda^{-1}\in \omega^{\frac{r(q+1)}{(q+1,2)}}W=\mathbb{F}_{q_0}$, and as $\eta\ne 0$, we have $\eta\in \mathbb{F}_{q_0}^*$.
\end{proof}

To formulate the next lemma we need to define some notation. Let $M$ be the set of elements of  $Z \, \SU_3(q)$ of the form 
\begin{equation}\label{Melement}
g= \begin{pmatrix}
    \nu_{11} & 0 & \nu_{12} \\
    0 & \delta & 0 \\
    \nu_{21} & 0 & \nu_{22}
\end{pmatrix}, \textrm{ where } \delta \in \mathbb{F}_{q^2}^*, \det( \nu_{ij}) \ne 0,\textrm{ and }\nu_{11}, \nu_{22} \in \mathbb{F}_{q_0}, \nu_{12}, \nu_{21} \in W.   
\end{equation}
Using  Remark~\ref{r:subfieldSU}(a), it is routine to check that $M$ is a subgroup of $Z \, \SU_3(q)$,  and that $d=\det( \nu_{ij})\in \mathbb{F}_{q_0}^*\leq \omr[q+1].$  More precisely, $M\leq (Z \, \SU_3(q))_U.$

 
\begin{Pro}\label{prM}
    For each choice of the $\nu_{ij}$ as in \eqref{Melement}, there are exactly $(q+1)/(q+1,3)$ elements $\delta \in \mathbb{F}_{q^2}^*$ such that $g \in Z \, \SU_3(q).$
\end{Pro}
\begin{proof}
Suppose that $g$ is as in \eqref{Melement} for a fixed choice of the $\nu_{ij}$ such that $\nu_{11}, \nu_{22}\in\mathbb{F}_{q_0}$,  $\nu_{12}, \nu_{21}\in W$, and $d:=\det(\nu_{ij})\ne 0$.  We need to count the number of possible entries $\delta \in \mathbb{F}_{q^2}^*$ such that $g \in Z \, \SU_3(q)$, and for this we use the criteria established in Lemma~\ref{l:2substab}.

It follows from Remark~\ref{r:subfieldSU}(a) and (c) that $\eta^q=\eta$ for all $\eta\in\mathbb{F}_{q_0}$ (since $\mathbb{F}_{q_0}\leq \mathbb{F}_q$), and  $\eta^q=-\eta$ for all $\eta\in W$ (since $W\subseteq {\rm ker}({\rm Tr})$). Thus $\nu_{11}\nu_{12}^q = -\nu_{11}\nu_{12}$ and $\nu_{21}\nu_{22}^q=\nu_{21}\nu_{22}$, and each of these quantities is a product of an element of $\mathbb{F}_{q_0}$ and an element of $W$, and hence lies in $W$, by Remark~\ref{r:subfieldSU}(a). This then implies that  $\mathrm{Tr}(\nu_{i1}\nu_{i2}^q)=0$, for $i=1,2$ since $W\subseteq {\rm ker}({\rm Tr})$.
Also, $\nu_{11}\nu_{22}^q+\nu_{12}\nu_{21}^q = \nu_{11}\nu_{22}-\nu_{12}\nu_{21} = \det(\nu_{ij}) = d\ne 0$, and so $d\in\mathbb{F}_{q_0}^*$, again by Remark~\ref{r:subfieldSU}(a). Thus the entries of $g$ satisfy Lemma~\ref{l:2substab}(a) if and only if $\delta^{q+1}=d$. There are exactly $q+1$ elements  $\delta\in \mathbb{F}_{q^2}^*$ with this property since the map $\mathbb{F}_{q^2}^*\to \mathbb{F}_{q}^*$ given by $\delta\to\delta^{q+1}$ is a group epimorphism with kernel $\omr[q-1]$, and $d\in\mathbb{F}_{q_0}^*\subseteq \mathbb{F}_{q}^*$. 
Further, if $\omega^\ell$ is one of the elements satisfying $\omega^{\ell(q+1)}=d$, then the $q+1$ elements with $(q+1)^{th}$ power $d$ are precisely $\omega^{\ell+m(q-1)}$ for $0\leq m\leq q$.

So assume now that $\delta^{q+1}=d$. Then the element $g$ lies in $ Z\,\SU_3(q)$  if and only if the conditions of Lemma~\ref{l:2substab}(b) hold: namely, there exists an element $\lambda=\delta\,\omega^{s(q-1)}$, for some $s$, (to ensure that  $\lambda^{q+1}=\delta^{q+1}$), such that $\lambda^3=\delta d=\delta^{q+2}$. Substituting for  $\lambda$ in the second condition we obtain 
\[
\delta^3\omega^{3s(q-1)}= \delta^{q+2}, \ \text{that is,}\ \omega^{3s(q-1)}=\delta^{q-1}.
\]
We consider each of the elements $\delta=\omega^{\ell+m(q-1)}$ to determine when such an integer $s$, and hence such an element $\lambda$ exists. The above equation is equivalent to $\ell+m(q-1)\equiv 3s\pmod{q+1}.$
Suppose first that $3$ does not divide $q+1$. Then there exists $n$ such that $3n\equiv 1\pmod{q+1}$ and taking $s\equiv (\ell+m(q-1))n\pmod{q+1}$ we obtain  $3s \equiv \ell+m(q-1)\pmod{q+1}$, giving a suitable element $\lambda$. Thus in this case the conditions of Lemma~\ref{l:2substab}(b) hold for each of the $q+1$ possibilities for $\delta$.

Suppose now that $3$ divides $q+1$ (so $3$ does not divide $q-1$). 
Then the quantity $\ell+m(q-1)$  takes on each of the values $0,1,2$ modulo $3$ exactly $(q+1)/3$ times as $m$ runs over $0,\ldots,q$.  If $\ell+m(q-1)$ is not  a multiple of $3$ then there is no solution to $\omega^{3s(q-1)}=\delta^{q-1}$, and hence no suitable element $\lambda$ exists for $\delta=\omega^{\ell+m(q-1)}$.  On the other hand, if $\ell+m(q-1)=3n$, then taking $s\equiv n\pmod{q+1}$ we obtain a solution to $\omega^{3s(q-1)}=\delta^{q-1}$, and hence an element $\lambda=\delta\,\omega^{s(q-1)}$ with the required properties. Thus there exists an element $\lambda$ satisfying the conditions of Lemma~\ref{l:2substab}(b) for exactly $(q+1)/3$ of the possibilities for $\delta$. 
\end{proof}

\begin{Lem}\label{lf7U}
    Let $X$ be the stabiliser of $L_{e,f}$ in $Z \, \SU_3(q)$.  Then 
    \begin{enumerate}[label=\normalfont (\arabic*)]
    \item $X=MY$, where $Y =\langle \omega^r I\rangle$ is as in Construction $\ref{con:psu}$, and $M$ is the subgroup of matrices as in \eqref{Melement};
    \item $X$ and its subgroup $M$ both act $2$-transitively on the $q_0+1$ points of $L_{e,f};$
    \item $|\mathcal{L}|=\frac{q^3(q^3+1)(q-1)^2}{q_0(q_0^2-1)(q_0-1)}.$
    \end{enumerate}
\end{Lem}
\begin{proof}
      Let $L=L_{e,f}$, and let $\alpha_1=\omr e$ and $\alpha_2=\omr f$.
      
 \medskip\noindent
 \emph{Claim 1:} $MY\leq X$     

\smallskip
      First we show that $M$ leaves $L$ invariant. Let $g \in M$ be as in \eqref{Melement}, with $\nu_{11}, \nu_{22}\in\mathbb{F}_{q_0}$ and $\nu_{12}, \nu_{21}\in W$ such that $\det(\nu_{ij})\ne 0$. 
   Then by Lemma \ref{lem:lineformSU}(3), $L^g=L_{(e)g,(f)g}$, and so, 
   the elements of $L^g$ are the point 
   $$\omr (e)g=\omr (\nu_{11} e + \nu_{12}f)$$ 
   and, for $\lambda\in W$, the points  
   $$
   \omr (\lambda (e)g+(f)g)=\omr\left( \lambda(\nu_{11} e + \nu_{12}f)+(\nu_{21}e+\nu_{22}f)\right)= \omr\left((\lambda\nu_{11}+\nu_{21}) e + (\lambda\nu_{12}+\nu_{22})f\right).
   $$  
   It follows from the conditions on the $\nu_{ij}$ and $\lambda$ that, for each of the points $\omr(\mu e+\mu' f)$ in either of the two displays above, the coefficients $\mu,\mu'$ both lie in $W\cup\mathbb{F}_{q_0}$, their product $\mu\mu'\in W$ (see Remark~\ref{r:subfieldSU}), and  $(\mu,\mu')\ne(0,0)$. Thus, by 
   Lemma \ref{lem:lineformSU}(1) (with $u=e, v=f$), each of these points lies in $L$, and hence $L^g\subseteq L$. We have equality since both sets have cardinality $q_0+1$.
      It follows that the element $g$ leaves $L$ invariant, and hence $M\leq X$.  Moreover $Y$ fixes every point of $\Omega$ so in particular $Y\leq X$. Thus   $MY\leq X$, and Claim 1 is proved.

 \medskip\noindent
 \emph{Claim 2:} $M$, and hence also $X$, acts $2$-transitively on the $q_0+1$ points of $L$ (proving part (2)). 
   
   \smallskip
    Consider an ordered pair of distinct points $\omr u_1$, $\omr u_2$ of $L=L_{e,f}$, and recall that $|L|=q_0+1$ by Lemma~\ref{lem:lineformSU}(1).
        By Lemma \ref{lem:lineformSU}(1),  for each $i$ we have $\omr u_i =  \omr (\nu_{i1}e + \nu_{i2}f)$ where the  $\nu_{ij}\in W\cup \mathbb{F}_{q_0}$, $(\nu_{i1},\nu_{i2})\ne (0,0)$, and $\nu_{i1}\nu_{i2}\in W$. This implies that, for each $i$,  one of  $\nu_{i1},\nu_{i2}$ is in $W$ and the other is in $\mathbb{F}_{q_0}$. Since $\omega^{r\frac{q+1}{(q+1,2)}}\in\omr$, we may choose 
   $\nu_{11}, \nu_{22} \in \mathbb{F}_{q_0}$ and $\nu_{12}, \nu_{21} \in W$. Also $\det(\nu_{ij}) \ne 0$ by Lemma  \ref{lem:lineformSU}(2).
             By Proposition~\ref{prM}, we may choose $\delta\in\mathbb{F}_{q^2}^*$ such that the matrix $g$ formed as in \eqref{Melement} with these entries lies in $M$. This matrix $g$ maps  $\alpha_i$ to the point $\omr u_i$, for each $i$. Thus $M$ is $2$-transitive on the points of $L$, and hence also (by Claim 1) $X$ is $2$-transitive on the points of $L$ proving Claim 2.

 \medskip\noindent
 \emph{Claim 3:} $X=MY$ (proving part (1)).    

\smallskip
    Since $M$ is $2$-transitive on the points of $L$, it follows that $X=M X_{\alpha_1,\alpha_2}$.     Let $g\in X_{\alpha_1,\alpha_2}$. Then by Lemma~\ref{l:2stab}, $g=\diag(\mu_1, \delta, \mu_2)$, for some $\mu_1, \mu_2 \in \omr$ and $\delta\in\mathbb{F}_{q^2}^*$ such that  $\delta^{q+1} = \mu_1\mu_2^q \in\omr \cap \omr[q+1]$.
    Now $L=L^g=L_{(e)g,(f)g}=L_{\mu_1 e,\mu_2 f}$ by Lemma \ref{lem:lineformSU}(3), and by Lemma \ref{lem:lineformSU}(4), $L_{\mu_1 e,\mu_2 f}=L_{ e,\mu_1^{-1}\mu_2 f}$. Thus $L=L_{ e,\mu_1^{-1}\mu_2 f}$  and hence $\mu_1^{-1}\mu_2\in\mathbb{F}_{q_0}^*$ by Lemma \ref{lem:lineformSU}(5). 
    Now $g= h\cdot (\mu_1 I)$, where $h=\diag(1, \mu_1^{-1}\delta, \mu_1^{-1}\mu_2)$, and we have $\mu_1 I\in Y$ and also $h= (\mu_1^{-1} I)\cdot g\in Z\,\SU_3(q)$ (since $g\in X\leq  Z \, \SU_3(q)$). Thus $h\in M$ by the definition of the $M$-elements  in \eqref{Melement}. We conclude that $X=M X_{\alpha_1,\alpha_2}\leq MY$, so  $X= MY$, proving Claim 3.

 \medskip\noindent
 \emph{Claim 4:} $|\mathcal{L}|$ is as in part (3) (completing the proof of the lemma).    

\smallskip

    By Definition~\ref{con:subfieldSU}, $Z \, \SU_3(q)$ is transitive on $\mathcal{L},$ and hence 
    $$
    |\mathcal{L}|=\frac{|Z \, \SU_3(q)|}{|X|}.
    $$
    Also, 
    $$|Z \, \SU_3(q)|=\frac{|Z|\cdot |\SU_3(q)|}{|Z \cap \SU_3(q)|}= \frac{(q^2-1) \cdot q^3(q^3+1)(q^2-1)}{(3,q+1)}.$$ 
    Since $X=MY$, we obtain $|X|= |M|\cdot |Y|/|M \cap Y|.$ Using the definition of $M$ in \eqref{Melement}, the fact that $|W|=|\mathbb{F}_{q_0}|=q_0$, and Proposition \ref{prM}, it is routine to calculate that
    $$|M|=q_0(q_0^2-1)(q_0-1) \cdot \frac{(q+1)}{(3,q+1)}.$$
    Since $\mathbb{F}_{q_0}^* \leq \omr,$ we see that $M\cap Y$ consists of all matrices $\diag(\lambda, \lambda, \lambda)$ where $\lambda\in \mathbb{F}_{q_0}^*$, and so 
    $|M \cap Y|= (q_0-1)= (q-1)/r.$
    Elementary calculations now yield the value of $|\mathcal{L}|$ as in part (3) completing the proof. 
\end{proof}

\begin{Th}\label{t:USub}
   Let $\mathcal{D}=\mathrm{USub}(q,q_0,r)$  as in Definition \ref{con:subfieldSU}. Then $\mathcal{D}$ is a  proper partial linear space, and $\Aut(\mathcal{D})$ contains $\GammaU_3(q)/Y$.
\end{Th}

\begin{proof}
By Lemma \ref{lem:lineformSU}(1), lines have size $q_0+1$ where $q_0$ is a prime power, so the line-size is a constant at least $3$, and by Lemma \ref{lem:lineformSU}(3), $\Aut(\mathcal{D})$ contains $\GammaU_3(q)/Y$.
  We need to show that every pair of distinct points of $\Omega$ is contained in at most one line.

\medskip  
\noindent
{\it Claim 1:  If $w_1,w_2 \in \mathcal{T}$ are such that $\omr w_1$ and $\omr w_2$ are distinct collinear points of $\Omega$, and both lie on a line $L$, then there is an element $g\in Z \, \SU_3(q)$   mapping  $L_{e,f}$ to $L$, $\omr e$ to $\omr w_1$ and $\omr f$ to $\omr w_2$.} 

\medskip     
\noindent     {\it Proof of Claim 1:} The existence of $g\in Z\,\SU_3(q)$ mapping  $L_{e,f}$ to $L$ follows from  Definition \ref{con:subfieldSU}.  That $g$ may be modified so that $g:\omr e \to \omr w_1$ and $g:\omr f \to \omr w_2$ follows from Lemma \ref{lf7U}(2) and Lemma \ref{claim1}.

\medskip    
Note  that $g$ preserves $\mathcal{L}$ (so induces an automorphism of $\mathcal{D}$).
In particular  the number of lines containing $\omr w_1$ and $\omr w_2$ is equal to the number of lines containing $\omr e$ and $\omr f$. 

 \medskip 
\noindent
{\it Claim 2:  A line $L$ contains $\omr e$ and $\omr f$ if and only if $L=L_{e,f}$.} 

  \medskip 
\noindent    
     {\it Proof of Claim 2:}
Let $L$ be an arbitrary line containing $\omr e$ and $\omr f$.
By Claim 1,  there is an element $g$ of $Z \, \SU_3(q)$ mapping  $L_{e,f}$ to $L$ and  fixing $\omr e$ and $\omr f$.
Then by Lemma~\ref{l:2stab}, $g=\diag(\mu_1, \delta, \mu_2)$, for some $\mu_1, \mu_2 \in \omr$ and $\delta\in\mathbb{F}_{q^2}^*$ such that  $\delta^{q+1} = \mu_1\mu_2^q \in\omr \cap \omr[q+1]$. Note that $\omr \cap \omr[q+1]=\omr[r(q+1)]$ since, by Definition~\ref{con:subfieldSU}, $r$ is odd and divides $q-1$ so $(r,q+1)=1$; also $r(q+1)=(q^2-1)/(q_0-1)$ so $\omr[r(q+1)]=\mathbb{F}_{q_0}^*$. Now by Lemma \ref{lem:lineformSU}(3), $L=L_{(e)g,(f)g}= L_{\mu_1 e,\mu_2 f}$, and then by Lemma~\ref{lem:lineformSU}(4), we obtain 
$$
L=L_{ e,\mu_1^{-1}\mu_2 f}.
$$
 The conditions on the $\mu_i$ imply  that $\mu_1^{-1}\mu_2= (\mu_1 \mu_2^q)^{-1} \cdot \mu_2^{q+1} \in \omr[r(q+1)]=\mathbb{F}_{q_0}^*.$ Hence $L=L_{e,f}$ by Lemma~\ref{lem:lineformSU}(5), proving Claim 2.

\medskip
It follows directly from the Claims 1 and 2 that $\mathcal{D}$ is a partial linear space.
Since $r>1$ by Definition \ref{con:subfieldSU}, $\omr e$ and $\omr \omega e$ are distinct points of $\Omega$. They are  not contained in a common line, by Lemma \ref{lem:lineformSU}(2).
Since line-size is at least $3$, it follows that  $\mathcal{D}$ is a proper  partial linear space.
\end{proof}

\subsection{Definition and properties of $\mathrm{AGU}^*(q)$}

The second family of examples is analogous to the affine geometry $\mathrm{AG}^*(n,q)$ (Definition~\ref{AGdef}) in the case where the underlying space is equipped with a nondegenerate unitary form. 

 \begin{Def}\label{con:AGSU}
In the notation of Construction \ref{con:psu}, assume that $q>2$ and $q$  is even, and let $r=q-1.$  
We define the following sets:
\begin{itemize}
   \item $L_{u,v}=  \omr\{\lambda u + (1-\lambda) v \mid \lambda \in \mathbb{F}_q \}$   for linearly independent $u,v \in \mathcal{T}$;  
    \item  $\mathcal{L}=\{L_{u,v} \mid  (u,v)=(e,f)^g, g \in  Z \, \SU_3(q)\}$;
    \item $\mathrm{AGU}^*(q)=(\Omega, \mathcal{L})$.
\end{itemize}  
\end{Def}

The next lemma is an analogue  to Lemmas \ref{lem:lineform} and \ref{lem:lineformSU}.

\begin{Lem}\label{lem:lineformAGSU}
Using the notation of Definition \ref{con:AGSU}:
\begin{enumerate}[label=\normalfont (\arabic*)]
    \item $L_{u,v}$ has cardinality $q$;
        \item if distinct points $\omr w_1, \omr w_2$ are collinear, then $w_1$ and $w_2$ are linearly independent;      
    \item if $L_{u,v} \in \mathcal{L}$ and $g \in \GammaU_3(q),$ then $(L_{u,v})^g =L_{(u)g, (v)g} \in \mathcal{L};$  in particular, $\GammaU_3(q)$ leaves $\mathcal{L}$ invariant; 
    \item if $\mu \in \omr$ and $L_{u,v} \in \mathcal{L}$, then $L_{\mu u ,v }=L_{u, \mu^{-1}v};$
    \item if $\eta \in \mathbb{F}_{q^2}^*$ and $L_{u,v} \in \mathcal{L}$, then $L_{u, \eta v}= L_{u,v}$ if and only if  $\eta=  1.$  
    \end{enumerate}
\end{Lem}
     \begin{proof} 
 The proofs of parts \ref{lfSUd}  and \ref{lfSU4} are analogous to the proofs of Lemma \ref{lem:lineform}(2) and (4), respectively, so the proof details are omitted. 
     
(1) First let us show  that $\lambda u+(1-\lambda) v\in \mathcal{T}$ when $\lambda\in \mathbb{F}_q $. Indeed, using the fact that $\lambda^q=\lambda$ for $\lambda \in  \mathbb{F}_q $, we obtain
 \[(\lambda u+(1-\lambda) v,\lambda u+(1-\lambda) v)=\lambda^2(u,u)+(1-\lambda)^2(v,v)+2\lambda(1-\lambda)(u,v)=0\]
since $u,v\in \mathcal{T}$ and $q$ is even.
Assume $\omr (\lambda u+(1-\lambda) v)=\omr (\mu u+(1-\mu) v)$ where $\lambda,\mu\in  \mathbb{F}_q.$ Then $\lambda\mu^{-1}\in \omr$ and so 
$\omr (\mu u+(1-\mu) v)=\omr( \lambda u+(\lambda\mu^{-1}-\lambda)v)$, which we are assuming is equal to $\omr (\lambda u+(1-\lambda) v)$. This implies that $\lambda=\mu$. It follows that $|L_{u,v}|=|\mathbb{F}_q|$ and part (1) holds.

\medskip
  $(3)$ The proof here is similar to that for Lemma \ref{lem:lineformSU}(3), namely one first shows that it is sufficient to prove (3) in the case where $(u,v)=(e,f)$, and the same careful argument as that given in Lemma \ref{lem:lineformSU}(3) shows that, for $g \in \GammaU_3(q)$, the line $L_{(e)g,(f)g} \in \mathcal{L}$. Finally a similar computation shows that $(L_{e,f})^{g}=L_{(e)g,(f)g}$.

         (5) Assume that $L_{u,\eta v}=L_{u,v}$ and let $1 \ne \lambda \in \mathbb{F}_q^*$ (such an element $\lambda$ exists since $q>2$).  By Definition~\ref{con:AGSU},  $u,v$ are linearly independent.
                Now the point $\alpha:=\omr(\lambda u +(1-\lambda)\eta v) \in L_{u, \eta v}$ and so $\alpha$ must be equal to a point $\omr(\mu u + (1-\mu) v)$ of $L_{u,v}$, for some $\mu \in \mathbb{F}_q.$ Thus $\xi(\lambda u +(1-\lambda)\eta v)= \mu u + (1-\mu) v$, for some $\xi\in \omr = \omr[q-1]$. In particular $\mu=\xi\lambda$. Since $\lambda \ne 0$, we must have $\mu\ne0$, and so $\xi=\mu\lambda^{-1}\in \mathbb{F}_q^*\cap \omr[q-1] = \{1\}$. Thus $\xi=1$ and $\mu=\lambda$. We also require $\xi(1-\lambda)\eta=1-\mu$ (the coefficient of $v$), and so $(1-\lambda)\eta =1-\lambda $, which yields $\eta=1$ since $\lambda\neq 1$.
                Thus Part (5) is proved. 
\end{proof}

To formulate the next lemma we introduce a subgroup slightly different from the subgroup of elements \eqref{Melement} used in the previous subsection. Here we take $M$ to be the  set of elements  of $Z \, \SU_3(q)$ of the  form
\begin{equation}\label{MelementAG}
g= \begin{pmatrix}
    \nu_{1} & 0 & 1-\nu_{1} \\
    0 & \delta & 0 \\
    \nu_{2} & 0 & 1-\nu_{2}
\end{pmatrix},  \textrm{ where } \delta \in \mathbb{F}_{q^2}^*,\nu_{i} \in \mathbb{F}_q \textrm{ for } i=1,2, \textrm{ and }\nu_1\neq \nu_2. 
\end{equation}
Each matrix $g$ of this form leaves $U=\langle e,f\rangle$ invariant, and hence $g\in  (Z \, \SU_3(q))_U$, so Lemma~\ref{l:2substab} applies to elements of $M$. Also
\begin{equation} \label{e:detg}
  d:=\delta^{-1}\det(g) =   \nu_1(1-\nu_2)-\nu_2(1-\nu_1) = \nu_1-\nu_2\in\mathbb{F}_q^*.   
\end{equation}
Moreover it is routine to check that the set $M$ forms a subgroup, and so $M\leq (Z \, \SU_3(q))_U$.

\begin{Pro}\label{prMAG}
    For each choice of the $\nu_{1},\nu_2$ as in \eqref{MelementAG}, there are exactly $(q+1)/(3,q+1)$ elements $\delta \in \mathbb{F}_{q^2}^*$ such that $g \in (Z \, \SU_3(q))_U.$
\end{Pro}
\begin{proof}
Suppose that $\nu_{1},\nu_2$ satisfy the conditions in \eqref{MelementAG}, that is, each $\nu_i\in\mathbb{F}_q$ and $\nu_1\ne \nu_2$, and let $g$ be a matrix as in \eqref{MelementAG} for some $\delta\in\mathbb{F}_{q^2}^*$. Each such matrix leaves the subspace $U$ invariant.  
 Thus we need to count the number of possible entries $\delta \in \mathbb{F}_{q^2}^*$ such that $g \in Z \, \SU_3(q)$, and for this we use the criteria established in Lemma~\ref{l:2substab}, setting $\nu_{i1}=\nu_i$ and $\nu_{i2}=1-\nu_i$, for $i=1,2$. 

Note that $\nu_i^q=\nu_i$ for each $i$, and so $\nu_{11}\nu_{22}^q+\nu_{12}\nu_{21}^q=\nu_{1}(1-\nu_{2})+(1-\nu_{1})\nu_{2}$ is equal to the element $d$ in \eqref{e:detg} since $q$ is even. 
Moreover, for $i=1, 2$, $\mathrm{Tr}(\nu_{i1}\nu_{i2}^q)=0$ since $\nu_{i1}\nu_{i2}^q=\nu_i(1-\nu_i)\in \mathbb{F}_{q}$ and $\mathbb{F}_{q}=\mathrm{ker}(\mathrm{Tr})$ by Lemma~\ref{l:trace}(2). Thus $g$ satisfies Lemma~\ref{l:2substab}(a) if and only if $\delta^{q+1}=d$. 

By \eqref{e:detg}, $d=\nu_1-\nu_2\in \mathbb{F}_{q}^*=\omr[q+1]$, so arguing as in the proof of Lemma~\ref{prM}, there are exactly $q+1$ elements $\delta\in\mathbb{F}_{q^2}^*$ such that $\delta^{q+1}=d$, and for each such element $\delta$, the matrix $g$ lies in $ Z\,\SU_3(q)$  if and only if the conditions of Lemma~\ref{l:2substab}(b) hold: namely, there exists an element $\lambda=\delta\,\omega^{s(q-1)}$, for some $s$, (to ensure that  $\lambda^{q+1}=\delta^{q+1}$), such that $\lambda^3=\delta d=\delta^{q+2}$.  These are exactly the same conditions that were considered in the proof of Lemma~\ref{prM}. Using exactly the same argument we conclude that $g \in (Z \, \SU_3(q))_U$ for exactly $(q+1)/(3,q+1)$ values of $\delta$, as asserted.
\end{proof}

\begin{Lem}\label{lf7UAG}
    Let $X$ be the stabiliser of $L_{e,f}$ in $Z \, \SU_3(q)$. Then 
    \begin{enumerate}[label=\normalfont (\arabic*)]
    \item  $X=MY$, where $M$ is the subgroup of $Z \, \SU_3(q)$ consisting of matrices as in \eqref{MelementAG}, and $Y =\langle \omega^r I\rangle$  as in Construction $\ref{con:psu}$;
    \item $X$ and its subgroup $M$ both act $2$-transitively on the points of $L_{e,f};$
    \item $|\mathcal{L}|=q^2(q^3+1)(q-1).$
    \end{enumerate}
\end{Lem}

\begin{proof}
      Let $L=L_{e,f}$ and let $X$ be the stabiliser in $Z \, \SU_3(q)$ of $L$.
      
\medskip\noindent
 \emph{Claim 1:} $MY\leq X$.  

      \smallskip
      First we show that $M$ leaves $L$ invariant. Let $g \in M$ be as in \eqref{MelementAG}. By Lemma \ref{lem:lineformAGSU}(3), $L^g=L_{(e)g,(f)g}$, so the points of $L^g$ are the $q$ points  $\omr (\lambda (e)g+(1-\lambda)(f)g)$ for $\lambda\in \mathbb{F}_q.$ 
     Now 
\begin{align*}
    \omr (\lambda (e)g+(1-\lambda)(f)g)&=\omr (\lambda (\nu_{1}e + (1-\nu_1)f))+(1-\lambda)(\nu_{2}e + (1-\nu_{2})f)\\
       &=\omr (\lambda (\nu_{1}-\nu_2)+\nu_{2})e + (1-\nu_2-\lambda(\nu_1-\nu_2))f).
\end{align*}
Note that each coefficient is in $\mathbb{F}_q$ (since $\lambda,\nu_1,\nu_2 \in \mathbb{F}_q$ and the sum of the $e$- and $f$-coefficients is equal to $1$),  so the point $\omr (\lambda (e)g+(1-\lambda)(f)g)$ lies on $L$ for each $\lambda\in\mathbb{F}_q$.  
      This proves that $L^g\subseteq L$ and we have equality since both sets have cardinality $q$. It follows that the subgroup $M$ leaves $L$ invariant, and hence $M\leq X$.  Also $Y$ fixes every point so $Y\leq X$. Thus   $MY\leq X$, and Claim 1 is proved.

\medskip\noindent
 \emph{Claim 2:} $M$, and hence also $X$, acts $2$-transitively on the $q$ points of $L$ (proving part (2)). 
   
   \smallskip

 Consider a pair of distinct points $\omr u_1$, $\omr u_2$ of $L_{e,f}$. By Definition \ref{con:AGSU}, for each $i$, we have $\omr u_i =  \omr (\nu_{i}e +(1-\nu_{i})f)$ for some $\nu_{i} \in \mathbb{F}_q$ and $\nu_1\neq \nu_2$. Thus the $\nu_{i}$ satisfy the conditions of \eqref{MelementAG}, and so by  Proposition~\ref{prMAG}, there are precisely  $(q+1)/(3,q+1)$ elements $\delta\in\mathbb{F}_q^*$ such that the matrix  $g$ in  $\eqref{MelementAG}$ with these entries $\nu_i, \delta$, lies  in $Z\,\SU_3(q)$, and hence lies  in $M$. Each of these elements $g\in M$ maps the two points $\alpha_1:=\omr e, \alpha_2:=\omr f$ to the two points $\omr u_1,\omr u_2$ of $L_{e,f}$ (in that order). 
Thus $M$ is $2$-transitive on the points of $L$, and hence (by Claim 1) also $X$ is $2$-transitive on the points of $L$ proving Claim 2.

 \medskip\noindent
 \emph{Claim 3:} $X=MY$ (proving part (1)).    

\smallskip

    Since $M$ is $2$-transitive on the points of $L$, it follows that $X=X_{\alpha_1,\alpha_2}  M$. Now let $g\in X_{\alpha_1,\alpha_2}$. Then, by Lemma~\ref{l:2stab},  
    $g=\diag(\mu_1, \delta, \mu_2)$ for some $\mu_i\in\omr[q-1]$ (since $r=q-1$ in Definition~\ref{con:AGSU}) and $\delta \in \mathbb{F}_{q^2}^*$ such that $\delta^{q+1}=\mu_1\mu_2^q\in \omr[q-1]\cap\omr[q+1]$.  Now $\omr[q-1] \cap \omr[q+1] =\omr[(q^2-1)/(2,q+1)]$, which is $\{1\}$ since $q$ is even. This implies firstly that  $\delta^{q+1}=1$, which means that $\delta\in\omr[q-1]$. It also implies that   $\mu_1=\mu_2^{-q}$. Since $\mu_2\in\omr[q-1]$ we have $\mu_2=\nu^{q-1}$ for some $\nu\in\mathbb{F}_{q^2}^*$ and hence $\mu_2^{-q}=\nu^{-q^2+q}=\nu^{-1+q}=\mu_2$.  Therefore $\mu_1=\mu_2=\mu$, say, and we have   $g=\diag(\mu, \delta, \mu) = (\mu\, I)\cdot \diag(1,\mu^{-1}\delta, 1)$. Now $\mu\,I\in Y$ since $\mu\in\omr[q-1]=\omr$, and  also
    $\diag(1,\mu^{-1}\delta, 1) = (\mu^{-1} I)\cdot g \in Z\,\SU_3(q)$. Thus 
    $\diag(1,\mu^{-1}\delta, 1)\in M$, by the definition of $M$-elements in \eqref{MelementAG}, and hence $g\in MY$. Thus $X\leq MY$, and equality holds by Claim 1, proving Claim 3. 

 \medskip\noindent
 \emph{Claim 4:} $|\mathcal{L}|$ is as in part (3) (completing the proof of the lemma).    

\smallskip

   Since $Z \, \SU_3(q)$ is transitive on $\mathcal{L},$ we have 
    $$|\mathcal{L}|=\frac{|Z \, \SU_3(q)|}{|X|}.$$
    Also 
    $$
    |Z \, \SU_3(q)|=\frac{|Z|\cdot |\SU_3(q)|}{|Z \cap \SU_3(q)|}= \frac{(q^2-1) \cdot q^3(q^3+1)(q^2-1)}{(3,q+1)}.
    $$ 
    Since $X=MY$, we have $|X|= |M|\cdot |Y|/|M \cap Y|.$ To determine $|M|$ we note that there are $q(q-1)$ choices for the entries $\nu_{1}, \nu_2\in \mathbb{F}_q$ with $\nu_1\ne \nu_2$, and given these choices for the $\nu_{i}$, the number of choices for $\delta$ to form an element $g$ of $M$ is given by Proposition \ref{prMAG}, yielding
    $$
    |M|=q(q-1) \cdot \frac{(q+1)}{(3,q+1)}.
    $$
    Since $\omr \cap \mathbb{F}_{q}^*=\omr[q-1] \cap \omr[q+1] =1$ (as $q$ is even), we have $|M \cap Y|= 1$, and also $|Y|=q+1$ (since $r=q-1$). Thus
    $$
    |X|= q(q-1) \cdot \frac{(q+1)^2}{(3,q+1)}.
    $$
    The claimed expression for $|\mathcal{L}|$ in part (3) now follows.
\end{proof}

\begin{Th}\label{t:AGUstar}
    Let  $\mathcal{D}=\mathrm{AGU}^*(q)$ as in Definition~$\ref{con:AGSU}$. Then 
    $\mathcal{D}$ is a proper partial linear space, and $\Aut(\mathcal{D})$ contains $\GammaU_3(q)/Y$.
\end{Th}

\begin{proof}
By Lemma \ref{lem:lineformAGSU}(1), lines have size $q$ where $q$ is a power of $2$ and $q\neq 2$, so the line-size is a constant at least $4$,   and by Lemma \ref{lem:lineformAGSU}(3), $\Aut(\mathcal{D})$ contains $\GammaU_3(q)/Y$.
Also $r=q-1\geq 3$, and so $\omr e$ and $\omr \omega e$ are distinct points of $\Omega$; and they are  not contained in a common line by Lemma \ref{lem:lineformAGSU}(2). Thus to complete the proof that $\mathcal{D}$ is a proper partial linear space, it remains to show that each pair of distinct points of $\Omega$ lies on at most one line of $\mathcal{L}.$ 

\medskip   
\noindent
{\it Claim 1:  If $w_1,w_2 \in \mathcal{T}$ are such that $\omr w_1$ and $\omr w_2$ are distinct collinear points of $\Omega$ on a line $L$ of $\mathcal{D}$, then there is an element $g\in Z \, \SU_3(q)$   such that  $(L_{e,f}, \omr e, \omr f)^g=(L, \omr w_1, \omr w_2)$.}

\medskip    \noindent
This claim follows immediately from Definition \ref{con:AGSU}, Lemma \ref{lf7UAG}(2) and Lemma \ref{claim1}.
Note  that $g$ preserves $\mathcal{L}$ (so induces an automorphism of $\mathcal{D}$).
Hence  the number of lines containing $\omr w_1$ and $\omr w_2$ is equal to the number of lines containing $\omr e$ and $\omr f$. 

 \medskip 
\noindent
{\it Claim 2:  A line $L \in \mathcal{L}$ contains $\omr e$ and $\omr f$ if and only if $L=L_{e,f}$.} 

  \medskip 
\noindent   
     {\it Proof of Claim 2:}
Let $L$ be an arbitrary line containing $\omr e$ and $\omr f$.
By Claim 1  there is an element $g\in Z \, \SU_3(q)$  mapping  $L_{e,f}$ to $L$ and  fixing $\omr e$ and $\omr f$.
It therefore follows from Lemma~\ref{l:2stab} that $g=\diag(\mu_1,\delta,\mu_2)$ for some $\mu_i\in\omr = \omr[q-1]$, and $\delta\in\mathbb{F}_{q^2}^*$, such that $\delta^{q+1}=\mu_1\mu_2^{q}\in\omr[q-1]\cap \omr[q+1]=\omr[q^2-1]=1$ (as $q$ is even). Thus, noting that $\mu_2=\nu^{q-1}$ for some $\nu\in\mathbb{F}_{q^2}^*$, we also have $\mu_1=\mu_2^{-q}=
\nu^{-q^2+q}=\nu^{-1+q}=\mu_2$.
 Therefore, by Lemma~\ref{lem:lineformAGSU}(4,5), we obtain 
 $$
 L=L_{(e)g,(f)g}=L_{\mu_1 e, \mu_2 f}=  L_{ e,\mu_1^{-1}\mu_2 f}=L_.
 $$ 

\medskip
By Claims 1 and 2, each pair of distinct points of $\Omega$ lies on at most one line in $\mathcal{L}$, and this completes the proof.
\end{proof}

\section{Systems of imprimitivity for point stabilisers of linear groups}\label{sec:linear}

In this section we investigate systems of imprimitivity of $G_{\alpha}$ on $\Omega \backslash \sigma$ for $G \leq \ff$ of rank $3$ with $\ff$ as in Constructions \ref{con:psl}.
 
\medskip

Recall that, with the notation of Construction \ref{con:psl}, a group $G \leq \Sym(\Omega)$ such that 
$$
Y \, \SL_n(q)/Y \leq G \leq \ff
$$ 
has rank 3 if and only if the conditions in the first line (or the second line if $n=r=2$) of Table \ref{t:qpinsprk3} hold. In particular, $r$ is always a prime dividing $q-1$ in this section.

\medskip
Properly innately transitive and semiprimitive groups of rank 3 appearing in the first two lines of Table \ref{t:qpinsprk3} are classified in \cite{R3PIT} and \cite{semiprR3} in terms of  Construction \ref{con:psl}.
The imprimitive quasiprimitive  groups of rank $3$ were classified in \cite{DGLPP} using different notation. So, to be able to treat these groups simultaneously,  we show explicitly (see Lemma \ref{QPtoIT}) that all imprimitive quasiprimitive  rank $3$ groups with plinth $\PSL_n(q)$ as in the first two lines of Table \ref{t:qpinsprk3} arise in Construction \ref{con:psl}. First we describe briefly the rank $3$ imprimitive quasiprimitive groups with plinth $\PSL_n(q)$. (This is necessary because of the $56$ errors in numbering of the results in \cite{DGLPP} introduced by the publisher, see \cite{DGLPP2}.)

\begin{Rem}\label{alloccur}
Let $G_0\leq \Sym(\Omega_0)$ be an imprimitive quasiprimitive  group of rank $3$ with plinth $M_0 \cong \PSL_n(q)$ as in lines $2$ or $3$ of \cite[Table 1]{DGLPP}. Let $\Sigma_0$ be the unique nontrivial system of imprimitivity of $G_0$, and let $r=|\sigma|$ for $\sigma \in \Sigma_0$. We briefly restate the conditions on $G_0$ from \cite{DGLPP}. In particular, $G_0$ is almost simple and is isomorphic to $G_0^{\Sigma_0}$ which is 2-transitive of degree  $|\Sigma_0|=(q^n-1)/(q-1)$, and we can assume that 
$$
\PSL_n(q) \leq G_0 \leq \PGammaL_n(q)=\PGL_n(q) \rtimes \langle \phi \rangle.
$$  
Let $H_0=G_0 \cap \PGL_n(q)$ so,  for some $d$ dividing $(n,q-1)$, $H_0$ is the following set of cosets of $Z$ (the subgroup of scalars in $\GL_n(q)$):  $H_0 = \{A\,Z \mid A \in \GL_n(q), \det(A) \in \langle \omega^d \rangle\}$.  \ Let $j$ be the largest divisor of $a$ such that $G_0 \leq \PGL_n(q) \rtimes \langle \phi^j \rangle$, so $|G_0/(G_0 \cap \PGL_n(q))|=a/j$. Let  $\tau= \diag(1, \ldots, 1, \omega)$.    \begin{enumerate}
    \item If $n \geq 3,$ then $G_0= \langle H_0, \phi^j \tau^m \,Z\rangle$   with $0\leq m\leq d-1$, and the following conditions hold:  $r$ is a prime, $(rd,n)=d$, $o_r(p^j)=r-1$, $dr$ divides $q-1$, and $dr$ divides $(m + \lambda d)(q-1)/(p^j-1)$ for some $\lambda \in \{0, 1, \ldots, r-1\}.$  
    
    \item If $n=2,$ then $r=2,$ $q \equiv 1 \mod 4,$ $G_0=\langle \PSL_2(q), \phi^j \tau\,Z  \rangle$ where $a/j$ is even and either $p^j \equiv 3 \mod 4$ or $p^j \equiv 1 \pmod 4$ and $a/j \equiv 0 \mod 4.$ Moreover, up to conjugation in $G_0$, if $\alpha \in \Omega$, then $G_{\alpha}=P \rtimes \langle \tau^4 \, Z, \phi^j \tau^k \, Z \rangle/Z $ with $|P|=q$ and $k$ is equal to either 1 or 3 (the two actions are not permutationally equivalent), \cite[Proposition 4.1, p.663]{DGLPP}. 
\end{enumerate} 
If $n\geq 3$, these conditions imply that $r$ divides $(q-1)/(n,q-1)$, since $dr$ divides $(q-1)/(n,q-1)$ and $(dr,n)=d$;  also, as shown in \cite[Proof of Theorem A, p. 152]{R3PIT}, $o_r(p^j)=r-1$ if and only if  $o_r(p)=r-1$ and $(j,r-1)=1$. If $n=2,$ then, again  $r$ divides $(q-1)/(n,q-1)$ since $r=2$ and $q \equiv 1 \pmod 4.$ 
\end{Rem}


\medskip\noindent
\emph{Note:}  For $n\geq 3$, 
\cite[Remark 7 on p.665]{DGLPP} appears to allow an additional possibility for the group $G_0$ than the one given in Remark~\ref{alloccur}(1), namely $G_0=H_0$. To avoid confusion we note that in this case  $G_0= \langle H_0, \phi^j \tau^m \,Z\rangle$ with $j=a$ and $m=0$ and the proof of \cite[Case (2), Proposition 4.12]{DGLPP} shows that the conditions in Remark \ref{alloccur} (1) hold. Moreover,  since the stabiliser $(G_0)_\sigma$ acts regularly on $\sigma$, and since the action must also be $2$-transitive (for $G_0$ to have rank $3$), it follows that $r=2$. 

\medskip
    
 It is convenient to work with $\Omega_0$ being the set $\Omega$ in Construction \ref{con:psl}. We show that the group $G_0$ is permutationally equivalent to a subgroup of  $\ff = \GammaL_n(q)/Y$ acting on $\Omega$.  By Lemma~\ref{l:semi}(2), $\ff$ is innately transitive with plinth $M = Y\,\SL_n(q)/Y\cong M_0$ since $r$ divides $(q-1)/(n,q-1)$, Further, $\ff$  is properly innately transitive since the normal subgroup $Z/Y\cong C_r$ is intransitive, and by   Lemma~\ref{l:semi}(4), $\ff$ has rank $3$.

    \begin{Lem}\label{QPtoIT}
    Let $G_0\leq \Sym(\Omega_0)$ be a quasiprimitive imprimitive group of rank $3$ with plinth $M_0= \PSL_n(q)$ as in Remark~{\rm\ref{alloccur}}, and let $\ff = \GammaL_n(q)/Y\leq \Sym(\Omega)$ acting on $\Omega$ as in   Construction {\rm\ref{con:psl}}.  Then the action of $G_0$ on $\Omega_0$ is permutationally isomorphic to the action on $\Omega$ of a subgroup $G$ of $\ff$ such that $G$ contains $Y\,\SL_n(q)/ Y$ but does not contain $Z\, \SL_n(q)/ Y$ (and $G$ satisfies the conditions in Lemma {\rm\ref{l:semi}(4)} to have rank $3$).
    \end{Lem}

     \begin{proof}
     First we show that the action of $M_0$ on $\Omega_0$ is permutationally isomorphic to the action of $M=Y\,\SL_n(q)/Y$ on $\Omega$.  Let $\Sigma_0$ be the unique nontrivial system of imprimitivity of $G_0$ in $\Omega_0$,
 Let $R_0$ be the stabiliser in $M_0$ of a point $\alpha_0 \in \Omega_0.$ Let $V:=(\mathbb{F}_q)^n=\langle e_1, \ldots, e_n \rangle$, and note that $G_0^{\Sigma_0}\leq \PGaL_n(q)$ 
 and $G_0^{\Sigma_0}$ is permutationally isomorphic to its action on the set $\binom{V}{1}$ of $1$-spaces. Thus  we may identify $\Sigma_0$ with $\binom{V}{1}$, and without loss of generality we assume that the $1$-space $\sigma_0 \in \Sigma_0$ fixed by $R_0^{\Sigma_0}$ is $\sigma_0 =\langle e_1 \rangle.$  By the proof of \cite[Propositions 4.10 and 4.12: Case (2)]{DGLPP}, $R_0$ is the unique normal subgroup of $(M_0)_{\sigma_0}$ of index $r$. By Lemma~\ref{l:semi}(2), $M=Y\,\SL_n(q)/Y \cong \PSL_n(q)=M_0$ since $r$ divides $(q-1)/(n,q-1).$  Writing $M_0=Z\, \SL_n(q)/ Z$ we define a 
 monomorphism $\rho:M_0\to \ff$ by $\rho:xZ\to xY$ (for $x\in\SL_n(q)$), and note that $\rho(M_0)=M$. Then, $\rho(x):\omr v\to \omr v^x$\  (for $\omr v\in\Omega$ and $x\in\SL_n(q)$) defines an action of $M_0$ on $\Omega$. If $\alpha =\omr e_1 \in \Omega,$ then $R:=M_{\alpha}$ is the unique normal subgroup of $M_{\sigma_{e_1}}=(M_0)_{\sigma_0}.$ So $R_0=R$, and $M_0^{\Omega_0}$ and  $M^{\Omega}$ are permutationally isomorphic.

Since $M_0\unlhd G_0\leq N_{\Sym(\Omega_0)}(M_0)$, it follows that $G_0$ is  permutationally isomorphic to a subgroup $G$ of $N_{\Sym(\Omega)}(M)=\ff$ containing $M$. Further, since $G_0$ is quasiprimitive on $\Omega_0$ but $Z\, \SL_n(q)/ Y$ is not quasiprimitive on $\Omega$ (as  its nontrivial normal subgroup $Z/Y$ is intransitive), it follows that $G$ does not contain $Z\, \SL_n(q)/ Y$. Finally $G$ is rank 3, and hence the conditions of  Lemma {\rm\ref{l:semi}(4)} hold.  
     \end{proof}

     Before classifying the systems of imprimitivity of $G_{\alpha}$ on $\Omega \backslash \sigma$, we state some preliminary information and prove several statements that are used in the classification.

\medskip

Let $V=\mathbb{F}_q^n$ with $q=p^a$ and $n \geq 2$. Let $\omega$ be a generator of $\mathbb{F}_q^*$. Fix a basis $\{e_1, \ldots, e_n\}$ of $V$. Let $Z=Z(\GL_n(q))$, let $r$, $\Omega$, $\Sigma$, $Y$ be as in  Construction \ref{con:psl}. It is more convenient to work with matrix groups than with their quotients, so in this section we let  $$
Y \, \SL_n(q) \leq G \leq \GammaL_n(q)$$
where $\GammaL_n(q)$ acts on $\Omega$ as in Construction \ref{con:psl} via \eqref{e:action}.

We assume that $G/Y$ is (imprimitive) semiprimitive of rank 3, so all conditions on $q,r,n$ in  Table \ref{t:qpinsprk3} hold. In particular, $r$ is prime dividing $q-1$, 
and $|G/G \cap \GL_n(q)|=a/j$ where $j$ divides $a$. Also, if $G_1 :=G \cap \GL_n(q)$, then $G=\langle G_1, h \phi^j \rangle $ for some $h \in \GL_n(q)$ and $\phi$ as in  \eqref{def:phi} with respect to the basis $\{e_1, \ldots, e_n\}$.

Let $\alpha=\langle \omega^r \rangle e_1\in\Omega $ so that $\alpha\in \sigma=\{\omr\omega^i e_1 \mid 0 \leq i \leq r-1\}.$ If $g \in G_{\alpha}$, then $g=(\phi^j)^k\cdot\hat{g}$ where $k \in \{0, 1, \ldots, a/j-1\}$ and $\hat{g}\in \GL_n(q)$ is of the form
\begin{equation*}
    \begin{pmatrix}
    \delta & 0 \\
    b & A
\end{pmatrix}
\end{equation*}
where $\delta \in \langle \omega^r \rangle$, $b^{\top} \in \mathbb{F}_q^{n-1}$, $A \in \GL_{n-1}(q).$ When working with the action of $G$ on $\Omega$,   since the kernel $Y=\omr$ of  this action is contained in $G_\alpha\cap\GL_n(q)$ and $Y$ is transitive on  $\omr e_1$ (viewed as a set of vectors), we may assume that $\delta=1$ (replacing $\hat{g}$ by another element of  the coset $\hat{g}Y$ if necessary), so we may assume that $\hat{g}$ has the form  
\begin{equation}\label{eq:delta1}
    \begin{pmatrix}
    1 & 0 \\
    b & A
\end{pmatrix}.
\end{equation}

\begin{Lem}\label{lem:blockS} Let $Y \, \SL_n(q) \leq G \leq \GammaL_n(q)$.
     Let $B$ be a block  containing $\langle \omega^r \rangle e_2$ in a  system of imprimitivity of $G_{\alpha}$ on $\Omega \backslash  \sigma$. If $\langle \omega^r \rangle v \in B,$ then $v\in \langle e_1,e_2\rangle$.
\end{Lem}
\begin{proof}
  The statement is trivially true when $n=2$, so assume that $n\geq 3.$
Assume for a contradiction that $B$ contains a point $\langle \omega^r \rangle v$, where $v\notin \langle e_1,e_2\rangle$.  By relabeling vectors in $\langle e_3, \ldots, e_n \rangle$ we may assume without loss of generality that  $v= \lambda_1 e_1 + \lambda_2 e_2 +e_3.$

    Since $B$ is a block of a   nontrivial system of imprimitivity, we have $1<|B|\leq\frac{|\Omega \backslash \sigma|}2$. Therefore, there exists $\langle \omega^r \rangle u \in (\Omega \backslash \sigma) \backslash B$, so let $\langle \omega^r \rangle u  \notin B$ where $u= \mu_1 e_1 + \mu_2 e_2 + y$ with $y \in \langle e_3, \ldots, e_n \rangle$. We can assume that $y \ne 0$. Indeed, if all such $\omr u$ with $y \ne 0$ lie in $B$, then 
    $$|B| \geq \frac{(q^n-q^2)}{(q-1)}r > \frac{1}{2}  \left(\frac{(q^n-1)}{(q-1)}-1 \right)r =|\Omega \backslash \sigma|/2, $$
  which is a contradiction.
    
    Suppose that $n \geq 4$ and let 
    $$ g=
    \begin{pmatrix}
        1 & 0 & 0    \\
        0 & 1 & 0    \\
        \mu_1-\lambda_1 & \mu_2-\lambda_2 & y\\
        b & c & A
    \end{pmatrix} 
    $$
    where $b^{\top}, c^{\top} \in \mathbb{F}_q^{n-3}$ and $A \in \mathbb{F}_q^{(n-3) \times (n-2)}$
    are chosen so that $\det (g)=1.$ Then $g \in \SL_n(q)\leqslant G$. Also $g \in G_{\alpha}$ and $g$ stabilises $\omr e_2$,  which implies that $(B)g=B$, while $(\omr v)g= \omr u \notin B.$ Hence $g$ does not preserve the imprimitivity system $\{(B)g \mid g \in G_{\alpha}\}$ which is a contradiction.

Hence $n=3$, so $y=\mu_3 e_3$ and  $u=\mu_1 e_1 +\mu_2 e_2 + \mu_3 e_3$ with $\omr u \notin B$. Suppose next that $\mu_3 \in \omr$. Then, replacing $u$ by $\mu_3^{-1}u$,  we may assume that $\mu_3=1$, and then 
$$g =\begin{pmatrix}
        1 & 0 & 0    \\
        0 & 1 & 0    \\
        \mu_1 - \lambda_1 & \mu_2-\lambda_2 &  1
    \end{pmatrix} \in \SL_3(q)\leqslant G,  
    $$
$g \in G_{\alpha}$, $g$ stabilises $\omr e_2$ and $(\omr v)g = \omr u\notin B$, which is a contradiction. Hence $\mu_3 \not \in \omr$, and this implies that all points of the form $\omr (b e_1 + c e_2 + e_3)$ with $b,c \in \mathbb{F}_q$ lie in $B$. 
Consider the matrix
$$g =\begin{pmatrix}
        1 & 0 & 0    \\
        0 & (\mu_2+1)\mu_3^{-1} & 1    \\
        \mu_1 - \lambda_1 & \mu_2-\lambda_2(\mu_2+1)\mu_3^{-1} &  \mu_3-\lambda_2
    \end{pmatrix}.
    $$
A straightforward computation shows that $\det(g) = 1$ and hence $g \in \SL_3(q)\leqslant G$. Also $g\in G_{\alpha}$ and  
    $(\omr e_2)g= \omr( (\mu_2+1)\mu_3^{-1}e_2 + e_3)$ and we have just shown that this point lies in $B$. Hence $g$ leaves $B$ invariant. However we also have $(\omr v)g =\omr u \notin B$ which is a contradiction. Thus the lemma is proved.
\end{proof}

Next we identify a particular imprimitivity system for $G_\alpha$ acting on $\Omega \backslash \sigma$.

\begin{Lem}\label{lem:Ssystem}  Let $Y \, \SL_n(q) \leq G \leq \GammaL_n(q)$, and $S=  \omr\{\lambda_1 e_1 + \lambda_2 e_2 \mid \lambda_1 \in \mathbb{F}_q, \lambda_2 \in \mathbb{F}_q^*\}$. Then $S \subset \Omega \backslash \sigma$ with $|S|=q r$, and the set $\{ (S)g \mid g \in G_{\alpha}\}$ is a system of imprimitivity for $G_{\alpha}$ on $\Omega \backslash \sigma$.  \end{Lem}
\begin{proof}
   It is straightforward to see that   $|S|=qr$, and by definition $S \subset \Omega \backslash \sigma$. 
    Let $g \in G_{\alpha},$ so we can assume that $g=\phi^k\cdot \hat{g}$ for some $\hat{g} \in  \GL_n(q)_{\alpha}$ as in \eqref{eq:delta1} and $k \in \{0,1, \ldots, a-1\}.$  Assume that $(S)g \cap S\neq  \emptyset,$ that is there exists $\beta=\omr(\nu_1 e_1 +\nu_2 e_2) \in S$ (with $\nu_2\neq 0$) such that $(\beta)g\in S$. 
    Now 
    $$ 
    (\beta)g=(\omr(\nu_1 e_1 +\nu_2 e_2))g
    =\omr(\nu_1^{p^k} (e_1)\hat{g} +\nu_2^{p^k}(e_2)\hat{g})
    =\omr(\nu_1^{p^k} e_1 +\nu_2^{p^k} (e_2)\hat{g}),
    $$ 
    and since $(\beta)g\in S$, it follows that  $(e_2)\hat{g}=\mu_1e_1+\mu_2 e_2$ for some $\mu_i\in\mathbb{F}_q$ with $\mu_2\neq 0$.

    Consider an arbitrary point $\omr(\lambda_1 e_1 + \lambda_2 e_2)\in S$ (so $\lambda_2\neq 0$).
Then 
$$
(\omr(\lambda_1 e_1 + \lambda_2 e_2))g
=\omr(\lambda_1^{p^k} e_1+\lambda_2^{p^k} (e_2)\hat{g})= \omr( (\lambda_1^{p^k}+\lambda_2^{p^k} \mu_1 )e_1+\lambda_2^{p^k} \mu_2 e_2)
$$
and this point lies in $S$ since $\lambda_2,\mu_2\neq 0$.   Thus $(S)g\subseteq S$, and we have equality since $|(S)g|=|S|$.
    Hence for $g \in G_{\alpha},$ either $(S)g=S$ or $(S)g \cap S= \emptyset.$ So $\{ (S)g \mid g \in G_{\alpha}\}$ is a system of imprimitivity.
\end{proof}

 We say that a system of imprimitivity $\{A_i \mid 1 \leq i \leq k\}$ of a transitive group $H \leq {\rm Sym}(d)$ is a {\it refinement} of a system of imprimitivity $\{B_i \mid 1 \leq i \leq m\}$ of $H$ if for each $i \in \{1, \ldots, k\}$ there exists $j \in \{1, \ldots, m\}$ such that $A_i \subseteq B_j$. Since $H$ is transitive, each $B_j$ is a union of the same number of $A_i$, so $|A_i|=d/k$ divides $|B_j|=d/m$. 

\begin{Cor}\label{cor:block}
   Let $B$ be a block containing $\langle \omega^r \rangle e_2$ in a  system of imprimitivity of $G_{\alpha}$ on $\Omega \backslash \sigma$. Then the system $\{(B)g \mid g \in G_{\alpha}\}$
 is a refinement of $\{ (S)g \mid g \in G_{\alpha}\}$. In particular,   $|B|$ divides $|S|=qr$.
\end{Cor}
\begin{proof}
 By Lemma \ref{lem:blockS}, $B \subseteq S$  and $\{ (S)g \mid g \in G_{\alpha}\}$ is a system of imprimitivity by Lemma \ref{lem:Ssystem}. It follows from \cite[\S3, Lemma 2]{sup} that $\{(B)g \mid g \in G_{\alpha}\}$
 is a refinement of $\{ (S)g \mid g \in G_{\alpha}\}$, so   $|B|$ divides $|S|=qr$.  
\end{proof}

 Next we record information about the $G_\alpha$-orbits in $\Omega\setminus\sigma$.

 \begin{Lem}\label{lem:gltrans}
     Let $Y \, \SL_n(q) \leq G \leq\GammaL_n(q)$ and let  ${\Omega}$ and $r$ be as in Construction \ref{con:psl}.
     \begin{enumerate}[label=\normalfont (\arabic*)]
     \item If $n \geq 3$, then
     $\SL_n(q)_{\alpha}$, and hence also $G_\alpha$, acts transitively on $\Omega \backslash \sigma.$
     \item If $n=2$ and $Z \, \SL_2(q) \leq G$, then one of the following holds:
     \begin{enumerate}[label=$(\alph*)$]\label{it:gltransB}
         \item $G_{\alpha} \cap \GL_2(q)$ acts transitively on $\Omega \backslash \sigma.$ \label{it:gltrans2}
         \item $r=2$, $q$ is odd, $G \cap \GL_2(q)=Z \, \SL_2(q)$, and in this case $(Z \, \SL_2(q))_{\alpha}$ has two orbits of size $q$ on $\Omega \backslash \sigma.$\label{it:gltrans3}
     \end{enumerate} 
     \end{enumerate}
 \end{Lem}
 \begin{proof}
 Let $\omr v \in \Omega \backslash \sigma$, so $v=\lambda_1e_1+ \ldots+ \lambda_ne_n$ is such that $v \notin \langle e_1 \rangle $.
 Assume first that  $n\geq 3$ and let 
 $$g= \begin{pmatrix}
    1 & 0 & \ldots & 0 \\
    \lambda_1 & \lambda_2 & \ldots & \lambda_n \\
      \mathbf{0} & \mathbf{0} & \multicolumn{2}{c}{A}
\end{pmatrix},$$
where $\mathbf{0}$ is an $(n-2)$-column of zeroes and $A \in \GL_{n-2}(q)$ is such that $\det (g)=1.$
Then $g \in \SL_n(q)_{\alpha}$ and $(\omr e_2)g=\omr v$. Thus all points of $\Omega \backslash \sigma$ lie in the orbit of $\omr e_2$ under $\SL_n(q)_{\alpha}$, proving (1).

Now let $n=2$.
Then  $v= \lambda_1e_1+  \lambda_2e_2$ with $\lambda_2\ne 0$ since $v \notin \langle e_1 \rangle $.
Let  
\begin{equation*}
 g= \begin{pmatrix}
    1 & 0  \\
    \lambda_1 & \lambda_2
\end{pmatrix}.
\end{equation*}
Assume that $Z \, \SL_2(q) \leq G$.  Note that $Z \, \SL_2(q)$ is the set of elements of $\GL_2(q)$ with determinant  a square, and has index $1$ or $2$ in $\GL_2(q)$.

 If $G \cap \GL_2(q) >Z \, \SL_2(q),$ then $ G \cap \GL_2(q) =\GL_2(q)$ and so $g$ lies in $G_{\alpha} \cap \GL_2(q)$. In this case $\omr v$ lies in the orbit of $\omr e_2$ under $G_{\alpha} \cap \GL_2(q)$ and \ref{it:gltrans2} holds.
Assume  now that $G \cap \GL_2(q) =Z \, \SL_2(q)$.
If $\lambda_2$ is a square then 
$g$ lies in $G_{\alpha} \cap \GL_2(q)$ and again $\omr v$ lies in the orbit of $\omr e_2$ under $G_{\alpha} \cap \GL_2(q)$ and \ref{it:gltrans2} holds. In particular, \ref{it:gltrans2} holds if $q$ is even since all non-zero elements are squares in this case. So assume that $q$ is odd and that $\lambda_2$ is a non-square.   

Now suppose that $r$ is an odd prime. Then $\omega^r$ is also a non-square. Let $z=\diag(1,\omega^r)$. Then $\det(zg)=\lambda_2\omega^r$ is a square, so  $zg \in G_{\alpha} \cap \GL_2(q)$. Also $(\omr e_2)zg=\omr(\omega^r v)=\omr v$ so $\omr v$ lies in the orbit of $\omr e_2$ under $G_{\alpha} \cap \GL_2(q)$. Thus \ref{it:gltrans2} also holds when $q$ is odd and $r>2$. 
Finally assume that $r=2$.
Then $G_\alpha \cap \GL_2(q)$ consists of all matrices of the form
\begin{equation*}  
 \begin{pmatrix}
    \mu_1 & 0  \\
    \eta_1 & \eta_2
\end{pmatrix}
\end{equation*}
where $\mu_1\in\omr=\omr[2]$ and so $\eta_2$ must be a square. Thus the $ G_{\alpha} \cap \GL_2(q)$-orbit containing $\omr v$ is different from the orbit containing $\omr e_2$, and we have shown that 
$G_{\alpha} \cap \GL_2(q)$ has two orbits on $\Omega \backslash \sigma$, namely:
\begin{equation}\label{eq:orbits}
O_i =\langle \omega^2 \rangle\{\lambda_1 e_1 + \omega^{i-1}e_2 \mid \lambda_1 \in \mathbb{F}_q\} \text{ for } i \in \{1,2\}
\end{equation}
each of size $q$, and \ref{it:gltrans3} of the lemma holds.
 \end{proof}

 \begin{Def}\label{des:singer} A  {\it Singer cycle} \index{Singer cycle} of $\GL_n(q)$ is a cyclic subgroup of order $q^n - 1$.
\end{Def}

 A Singer cycle always exists.
Indeed, a field $\mathbb{F}_{q^n}$ can be considered as an $n$-dimensional vector space $V=\mathbb{F}_q^n$ over $\mathbb{F}_q.$  Right multiplication by a generator of  $\mathbb{F}_{q^n}^*$ determines a bijective linear map from $V$ to itself of order $q^n-1$. So  $\mathbb{F}_{q^n}^*$  is isomorphic to a cyclic subgroup of $\GL_n(q)$ of order $q^n-1$. It is easy to see that   the action of $\mathbb{F}_{q^n}^*$ on the set $V \backslash \{0\}$ is  regular. It is well known that all Singer cycles are conjugate in $\GL_n(q).$

\begin{Lem}\label{irrsub}
A proper subgroup $C$ of a Singer cycle $T \leq \GL_n(q)$ is irreducible if and only if $|C|$ does not divide $q^s-1$ for every proper divisor $s$ of $n$. 
\end{Lem}
\begin{proof}
Since all subgroups of a cyclic group are cyclic, $C=\langle \sigma \rangle$ for some $\sigma \in T.$

Assume that $C$ is reducible, so there is a non-zero proper  subspace $V_1$ of $V=\mathbb{F}_q^n$ such that $$(V_1)\sigma=V_1.$$ Therefore, $V_1 \backslash \{0\}$ is a collection of $m$ orbits of $C$ acting on all non-zero vectors. Since $T$ is regular on  $V \backslash \{0\}$, $C$ is semiregular and all orbits have size $|C|.$ Thus, $$q^s-1 =|V_1 \backslash \{0\}|=m|C|$$ for some  $s<n$. Thus, $|C|$ divides $q^s-1$. So $|C|$ divides 
$(q^s-1, q^n-1)=q^{(s,n)}-1$ where $(s,n)$ is a proper divisor of $n$ since $s<n$.  The converse follows from the proof of \cite[Satz II.7.3(a)]{hupp}.
\end{proof}

 Now we are ready to classify the systems of imprimitivity of $G_{\alpha}$ on $\Omega \backslash \sigma$. The classification splits into two cases $n \geq 3$ and $n=2$ which are handled in Theorems \ref{th:blocks} and \ref{th:blocksn2} respectively.

\begin{table}[h]
\begin{tabular}{cccclc}
\hline
Line   & $q$ & $r$  & $B$  & $|B|$  \\ 
\hline
1 & 3  & 2    &       $B_4= \{e_2, 2e_1 + 2 e_2\}$\   & 2                 \\ 
  &     &         &           $B_5=  \{e_2, e_1 + 2 e_2\}$        &2 \\ \hline
2  & 4  & 3    &       $B_4= \{e_2, e_1 +  e_2\}$\  &2                    \\ 
  &     &         &           $B_5= \{(1+\lambda)e_1 +\lambda e_2 \mid \lambda \in \mathbb{F}_4^*\}$&3\\ \hline
  3  & $2^4$  & 5    &     $B_{6}= \omr \{\lambda e_1 +e_2 \mid \lambda \in \mathbb{F}_4\}$\  &4    \\ \hline

\end{tabular}
\caption{Additional blocks for Theorem~\ref{th:blocks}, $n \geq 3$}\label{t:blocks}
\end{table}

\begin{Th} \label{th:blocks}
        Let  $n \geq 3$ and assume $Y \, \SL_n(q) \leq G \leq\GammaL_n(q)$ with $G^{\Omega}$ of rank 3 acting on $\Omega$ as in {\rm Construction \ref{con:psl}}  and satisfying Hypothesis \ref{hyp1}. Let $B$ be a nontrivial block containing $\beta=\langle \omega^r \rangle e_2$ in a system of imprimitivity for $G_{\alpha}$ on $\Omega \backslash \sigma$. Then   one of the following holds.
        \begin{enumerate}[label=\normalfont (\arabic*)]
                \item $B$ is $B_1= \omr\{e_2, \omega e_2, \ldots, \omega^{r-1} e_2 \}$ of size $r$;
                \item $B$ is  $B_2=  \omr\{\lambda e_1+e_2\mid \lambda\in \mathbb{F}_q\}$ of size $q$;
             \item $B$ is $B_3=S=\omr \{\lambda_1 e_1 + \lambda_2 e_2 \mid \lambda_1 \in \mathbb{F}_q, \lambda_2 \in \mathbb{F}_q^*\}$ of size $qr$;
             \item  $q, r, n$ and $B$ are as in one of the lines of Table~$\ref{t:blocks}$.
                          \end{enumerate}
\end{Th}

\begin{Remark}\label{r:sl}
    {\rm 
  Let $G, \Omega, r$ be as in Construction \ref{con:psl}, so $G^\Omega$ is semiprimitive on $\Omega$ by Lemma~\ref{l:semi}, and assume that $G^{\Omega}$ has  rank $3$. Let $q=p^a$ with $p$  prime and $a\geq1$, so $|G:G\cap\GL_n(q)|=a/j$ for some $j\mid a$. Let $Y=\langle \omega^rI\rangle\leq \GammaL_n(q)$.   Then, by Lemma~\ref{l:semi}(4), one of the following holds.
  \begin{enumerate}
      \item[$(a)$] $(n,r)\ne(2,2)$, $r$ is a primitive prime divisor of $p^{r-1}-1$, $(r-1,j)=1$ and $(r-1)\mid a$ (so in fact $(r-1)\mid a/j$). 
      \item[$(b)$] $(n,r)=(2,2)$ and $G\not\leq Y\,\SigmaL_2(q)/Y$.
  \end{enumerate} 
  In particular, in case $(a)$, if $r=q-1$ then $p^a- 2$ divides $a$, and this is only possible if  $(q,r)=(3,2)$ or $(4,3).$ These remarks will be useful in the proof of Theorems ~\ref{th:blocks} and \ref{th:blocksn2}.
    }
\end{Remark}

\begin{proof} (Proof of Theorem~\ref{th:blocks}.)\quad
By Lemma~\ref{lem:Ssystem}, $B_3=S$ is a block of size $qr$, and by Corollary~\ref{cor:block}, $B\subseteq S$ and $|B|$ divides $qr$. If $B=B_3$ then (3) holds, so from now on we assume that $B$ is a proper subset of $S$. 
Let $\Delta=\{(B)g\mid g\in G_\alpha\}$ be the system of imprimitivity  determined by $B$, so $\Delta$ is a proper refinement of $\{(S)g \mid g \in G_{\alpha}\}$. 
By the definition of a block, $G_{\alpha, B}$  leaves $S$ invariant. Then, since $G_{\alpha,S}$ is transitive on $S$ and $G_{\alpha, B} < G_{\alpha, S} < G_\alpha$, it follows that $B$ is a block for $G_\alpha$ in $\Omega \backslash \sigma$ if and only if $B$ is a block for $G_{\alpha,S}$ in $S$.  Moreover,  since $S=\omr\{ (\lambda_1 e_1 + \lambda_2 e_2) \mid \lambda_1 \in \mathbb{F}_q, \lambda_2 \in \mathbb{F}_q^*\}$ (see Lemma~\ref{lem:Ssystem}), it follows that $G_{\alpha, S}$ stabilises $\langle e_1, e_2\rangle$  in its action on $V$. 
Hence the restriction homomorphism $\tau: G_{\alpha,S} \to \GammaL_2(q)=\GaL(\langle e_1, e_2\rangle)$ is well-defined and,  for $g \in G_{\alpha,S}$ we have  
\begin{equation}\label{eq:ghat}
g^{\tau} = \delta I \cdot \begin{pmatrix}
    1 & 0 \\
    \lambda & \mu
\end{pmatrix} \cdot \phi^k    
\end{equation}
for some $\delta \in \omr$, $\lambda \in \mathbb{F}_q,$ $\mu \in \mathbb{F}_q^*$, and $0\leq k<a$. In view of \eqref{eq:delta1}, for each $g \in G_{\alpha, S}$, the elements $g$ and $\delta^{-1} g^{\tau}$ induce the same action on $S$, and we denote the map $g \mapsto \delta^{-1}g^{\tau}$ by $\theta.$   It is straightforward to check that $\theta$ is  a homomorphism from $G_{\alpha,S}$ to $\GaL_2(q)=\GaL(\langle e_1, e_2\rangle)$. 

To determine the possibilities for $B$ it is therefore sufficient to consider the action of $G_{\alpha,S}$ on $\langle e_1, e_2\rangle$ (and hence on $S$).  Our strategy is to consider the following subgroup of $(G_{\alpha,S})^\theta$: 
\begin{equation} \label{eq:Hmat}
 H=\left\{ \begin{pmatrix}
    1 & 0 \\
    \lambda & \mu
\end{pmatrix} \mid \lambda \in \mathbb{F}_q, \mu \in \mathbb{F}_q^* \right\}. 
\end{equation}
Notice that $(G_{\alpha,S})^\theta$ contains $H$ since  $G\geq \SL_n(q)$ and $n \geq 3.$

Since $ \begin{pmatrix}
    1 & 0 \\
    \lambda & \mu\end{pmatrix} $ maps $\omr e_2 $ to $\omr (\lambda e_1+\mu e_2)$, we see that $H$ is transitive on $S$.
 Hence each block of $G_{\alpha,S}$ acting on $S$ is a block of $H$  acting on $S$. 
Now we identify all the nontrivial blocks for $H$ in $S$,  and then decide whether each of these is a block for $G_{\alpha,S}$. We first consider the block $B_2$.

\medskip
\noindent\emph{Claim 1.\quad $B_2=   \omr\{\lambda e_1+e_2\mid \lambda\in \mathbb{F}_q\}$ is a block for $G_{\alpha,S}$ in $S$, as in part $(2)$. }

 \medskip
 
Since the Sylow $p$-subgroup $P$ of $H$,  consisting of matrices in \eqref{eq:Hmat} with $\mu=1$, is normal in $(G_{\alpha,S})^\theta$, its orbits form a system of imprimitivity for $G_{\alpha,S}$ in $S$ and the orbit containing $\beta$ is $B_2$.  Thus $B_2$ is a block  of size $q$, as in part (2), proving Claim 1.

 \medskip
In our analysis it turns out that the most delicate case is the one where $|B|=r$ (which leads to some exceptional blocks in Table~\ref{t:blocks}), so we deal with this case next.  The next step, Claim~2, identifies the cases $(q,r)=(3,2)$ and $(q,r)=(4,3)$ as rather special so we finish treating them in Claims 3 and 4, and then proceed to the remaining cases, where $|B|\ne r$.

\medskip

\noindent\emph{Claim 2.\quad $B_1=\omr\{e_2, \omega e_2, \ldots, \omega^{r-1} e_2 \}$ is a block for $G_{\alpha,S}$ in $S$ as in part $(1)$.  If $B \neq B_1$ is a block for $H$  in $S$, and $|B|=r$, then either  $(q,r)=(3,2)$ or $(q,r)=(4,3)$.}

 \medskip
 
 First we note that, as $B_1$ consists of the $r$ points of $\Omega$ contained in the $1$-space $\langle e_2\rangle$, so $B_1= \sigma(\langle e_2\rangle)$ in the notation of Construction \ref{con:psl}. In particular, it is a block for $G$ in $\Omega$ and, hence, it is a block for $G_{\alpha,S}$ in $S$.

We now assume that $B$ is a block for $H$ in $S$, and that $B\ne B_1$, $|B|=r$, and $\beta\in B$. Since $H$ is transitive on $S$, it follows that $H_B$ is transitive on $B$, and as $r=|B|$ is prime, there exists
\begin{equation}\label{eq:h}
h =\begin{pmatrix}
    1 & 0 \\
    \lambda & \mu 
\end{pmatrix}  \in H_B 
\end{equation}
 that induces a cycle of length $r$ on $B$.  Here $\lambda \in \mathbb{F}_q$ and $\mu \in \mathbb{F}_q^*.$ 
 If $\lambda=0$ then $h = \diag (1,\mu)$ and, as $h$ induces a permutation of  order 
 $r$ on $B$, we have $\mu\not\in\omr$ and $B=B_1$, which is not the case. Hence $\lambda\ne 0$. 
Now 
$$
h^r = \begin{pmatrix}
    1 & 0 \\
    \lambda(1 + \mu + \ldots +\mu^{r-1}) & \mu^r
\end{pmatrix}
$$
and $h^r$ fixes $B$ pointwise.
In particular $h^r$ fixes $\omr e_2,$ and hence $1 + \mu + \ldots +\mu^{r-1}=0$ since $\lambda\ne 0$. If $\mu=1$ then $0=1 + \mu + \ldots +\mu^{r-1}=r\cdot 1_{\mathbb{F}_q}$, which is not the case since $r$ divides $q-1$. Hence $\mu \ne 1$ and 
$$
0=(1-\mu)(1 + \mu + \ldots +\mu^{r-1})=1-\mu^r,$$
so $\mu^r=1$ and 
\begin{equation} \label{eq:Br}
B=\omr \{e_2, \lambda e_1 + \mu e_2, \lambda(1+\mu)e_1 +\mu^2 e_2, \ldots, \lambda(1 + \mu + \ldots +\mu^{r-2})e_1 + \mu^{r-1} e_2\}.    
\end{equation}

     Now $\diag(1, \omega^r)$ lies in $H$ and stabilises $\omr e_2$. Therefore  $\diag(1, \omega^r)\in H_B$ and we have 
\begin{align*}
\left(\omr(\lambda e_1 + \mu e_2)\right)\diag(1, \omega^r)^i & =\omr(\lambda e_1 + \omega^{ri} \mu e_2) \\ &=\omr( \omega^{-ri} \lambda e_1 + \mu e_2) \in B \text{ for } i \in \{1, \ldots, (q-1)/r\}.
\end{align*}
Since the coefficients of $e_2$ in these $(q-1)/r$ points are the same and the coefficients of $e_1$ are pairwise distinct and non-zero, we conclude that $(q-1)/r \le|B|-1= r-1$, so $q-1 \leq r^2-r<r^2.$ In particular $r^2$ does not divide $q-1$, and hence the order $|\omr|=(q-1)/r$ is not divisible by $r$. Since $\mu \ne 1$ and $\mu^r=1,$ this implies that $\mu \not\in \omr$. Hence it follows from \eqref{eq:Br} that, for distinct points of $B$, say $\omr v$ and $\omr v'$, the coefficients of $e_2$ in $v, v'$ must lie in different cosets of $\omr$ in $\mathbb{F}_q^*$ (independent of the choices of $v, v'$ in  $\omr v$ and $\omr v'$).  Further, since $\left(\omr(\lambda e_1 + \mu e_2)\right)\diag(1, \omega^r) = \omr (\omega^{-r}\lambda e_1 + \mu e_2)=:\omr v\in B$, the coefficient of $e_2$ in $v$ must lie in $\omr \mu$. Hence this point of $B$ must be equal to $\omr (\lambda e_1 + \mu e_2)$, and we conclude that $\omega^{-r} \lambda = \lambda$. Since $\lambda \ne 0$, this implies that $\omega^r=1$, so $r=q-1.$ Since $n \geq 3,$ it follows from Remark~\ref{r:sl} that  $(q,r)=(3,2)$ or $(q,r)=(4,3)$. This completes the proof of Claim 2.

\medskip

From now on we assume that   $B$ is a block for $H$ in $S$, and that $B\ne B_1, B_2, S$.  First, we fully consider the cases $(q,r)$ in $ \{(3,2), (4,3)\}.$

\medskip
\noindent\emph{Claim 3.\quad If $(q,r)=(3,2)$  then $B$ is either   $B_4=\omr[2]\{e_2, 2e_1 + 2 e_2\}$ or $B_5=\omr[2]\{e_2, e_1 + 2 e_2\}$, as in Table~\ref{t:blocks}.  In both cases, $B$ is a block for $G_{\alpha,S}$  in $S$. (Note that $\omega^2=1$ here.)}
\medskip

In this case $Y=1$ and $\GaL_n(q)=\GL_n(3)$, so $H=(G_{\alpha,S})^\theta$ and each block for $H$ on $S$ is a block for $G_{\alpha,S}$ on $S$.   Now $|H|=qr=6=|S|$, and $H_\beta=1$, so by \cite[Lemma 2.4]{PS}, the nontrivial blocks of imprimitivity for $H$ in $S$  containing $\beta$ are in one-to-one correspondence with the nontrivial proper subgroups of $H\cong S_3$, namely three subgroups of order $2$, and one subgroup of order $3$, corresponding to three blocks of size $2$ and one block of size $3$. By Claims 1 and 2, the block of size $3$ is $B_2$ and one of the blocks of size $2$ is $B_1$, which we are excluding, so  we need to identify exactly two further blocks of size $2$, and these are $B_4$ and $B_5$ as in Claim 3. To see this we note that, for $i=4,5$, $H_{B_i}=\langle h_i\rangle \cong C_2$
where 
$$
h_4=\begin{pmatrix}
    1 & 0\\
    2 & 2
\end{pmatrix}, \text{ }
h_5=\begin{pmatrix}
    1 & 0\\
    1 & 2
\end{pmatrix} .
$$

\medskip
\noindent\emph{Claim 4.\quad If $(q,r)=(4,3)$  then  $B$ is either $B_4=\omr[3]\{e_2, e_1 +  e_2\}$ or $B_5=\omr[3]\{(1+\lambda)e_1 + \lambda  e_2\mid \lambda\in \mathbb{F}_4^*\}$, as in Table~\ref{t:blocks}.  In both cases, $B$ is a block for $G_{\alpha,S}$  in $S$. (Note that $\omega^3=1$ here.)}
\medskip

Again, $Y=1$. Here $H$ is of order $12$. In fact, $H_\beta=1$ and as in the previous case,  by \cite[Lemma 2.4]{PS}, the nontrivial blocks of imprimitivity for $H$ in $S$ containing $\beta$ are in one-to-one correspondence with the nontrivial proper subgroups of $H\cong A_4$. Thus for the $H$-action, there are three blocks of size $2$, four blocks of size $3$, and one block of size $4$. 

Further by Remark~\ref{r:sl}(a), the parameter $j=1$ and $G\not\leq \GL_n(4)$, so $G$ contains an element of the form $g \phi$ with $g \in \GL_n(4)$. Since $\SL_n(4)\leq G$ and $\SL_n(4)$ is transitive on $2$-spaces, we may assume that $g$ stabilises ${\langle e_1,e_2 \rangle}$ setwise. Also,  as $n\geq 3$, $(G\cap\GL_n(4))_{\langle e_1, e_2\rangle}$ induces $\GL_2(4)$ on $\langle e_1, e_2\rangle$, so we may assume further that  $g|_{\langle e_1,e_2 \rangle}=I_2.$ 
 Thus both $g$ and the element $\phi$ fix $\alpha$ and $\beta$, and hence their product $g\phi\in G_{\alpha,\beta} \leq G_{\alpha,S}$. Therefore $g\phi$ leaves invariant each nontrivial block $B$ for $G_{\alpha,S}$  in $S$ containing $\beta$. Recall that $g$ fixes $S$ pointwise since $g|_{\langle e_1,e_2 \rangle}=I_2.$
Thus $g\phi$ and $\phi$ induce the same action on $S$.
By Claims 1 and 2, the block of size $4$ is $B_2$ and one of the blocks of size $3$ is $B_1$, which we are excluding, so  we need to identify blocks of size $2$ and  $3$.

If $|B|=3=r$, then $H_B$ contains an element $h$ of order $r$, as in \eqref{eq:h}, permuting the $r$ points of $B$ transitively. In particular $h$ does not fix $\beta$, so $\mu\ne 1$ and  up to swapping $h$ with its inverse we may assume that $\mu=\omega$. Hence 
$B= \{e_2, \lambda e_1 + \omega e_2, \lambda(1 +\omega) e_1 + \omega^2 e_2 \}$. Since   ${g\phi}$, as in the paragraph above, leaves $B$ invariant it follows that
$$
(\lambda e_1 +\omega e_2) {g\phi}=(\lambda e_1 +\omega e_2) {\phi} =\lambda^2 e_1 +\omega^2 e_2 \in B.
$$
Considering the coefficients of $e_2$ in the points of $B$, we conclude that  $\lambda^2=\lambda(1+\omega)$. If $\lambda=0$ then $B$ is the block $B_1$ and we are assuming that this is not the case. Hence $\lambda\ne 0$ and so $\lambda=1+\omega$, and $B=\{e_2, (1+\omega)e_1+ \omega e_2, (1+ \omega^2)e_1 + \omega^2 e_2\}$, which is the block $B_5$ in Claim 4 (noting that $-1=1$ in $\mathbb{F}_4$).

Suppose now that $|B|=2$. Then $H_B\cong C_2$, and $H_B$ contains an involution $h_2$ (interchanging the two points of $B$) of the form 
$$h_2 = \begin{pmatrix}
    1 & 0\\
    \lambda & 1
\end{pmatrix}.
$$ 
Since $h_2$ does not fix $\beta$, we have $0\ne \lambda\in\mathbb{F}_4$, and since $|B|=2$ we have $B=\{e_2, \lambda e_1+e_2\}$. Also, since $g\phi$ fixes $B$, we must have $\lambda^\phi=\lambda$, and hence $\lambda=1$, and $B$ is the block $B_4$ of Claim 4.  It is easy to check that $B_4$ and $B_5$ are indeed blocks for $G_{\alpha,S}$ in $S$.  Thus Claim 4 is proved.

\medskip
From now on we may assume that  $B$ is a block for $H$ in $S$, that $B\ne B_1, B_2, S$, that $|B|\ne r$, and   that $(q,r)\notin\{ (3,2),(4,3)\}$. By Remark~\ref{r:sl} it follows that $r$ is a proper divisor of $q-1$. Since $|B|$ divides $|S|=qr$ and  $r$ is prime, either $|B|=p^k$ or $|B|=p^kr$ for some positive integer $k$. 
Next we identify all possibilities  where $|B|=p^k$.

\medskip
\noindent\emph{Claim 5.\quad If $|B|=p^k$, for some positive integer $k$, then $(q,r)=(2^4,5)$  
    and $B$ is $B_{6}= \omr \{\lambda e_1 +e_2 \mid \lambda \in \mathbb{F}_4\} \subset B_2$
as in Table~\ref{t:blocks}. Moreover, $B_6$ is a block for $G_{\alpha,S}$ on $S$ and there are no further nontrivial blocks for $(G_{\alpha,S})^S$  when  $(q,r)=(2^4,5)$ (including blocks of size $p^kr$).}
\medskip

Suppose that $|B|=p^k$.  Recall from Claim 1 that $H$ has a normal Sylow $p$-subgroup $P$, and note that $P$ is semiregular on $S$. This means that $H_B= P_1 \rtimes H_\beta$ for some subgroup $P_1\leq P$ of order $p^k$, and $B$ is the $P_1$-orbit containing $\beta$. Thus $B\subseteq B_2$ and as we are assuming that $B\ne B_2$, $B$ is a proper subset of $B_2$. 

 It is easy to verify that 
$H_{B_2}$ consists of all matrices of the form 
\begin{equation}\label{e:hb2}
\begin{pmatrix}
    1 & 0 \\
    \lambda & \omega^{ri}
\end{pmatrix}
\end{equation}
with $\lambda \in \mathbb{F}_q$ and $i \geq0$ so $|H_{B_2}|=q(q-1)/r$, and also that, for $\beta=\omr e_2$,  $H_{\beta}$ consists of all such matrices with $\lambda=0$, so $H_\beta = \langle \diag(1, \omega^r) \rangle\cong C_{(q-1)/r}$. Since $\beta\in B\subset B_2$, we have $H_\beta<H_B<H_{B_2}$, and so by the previous paragraph, $H_B= P_1 \rtimes \langle \diag(1, \omega^r) \rangle$ with $1<P_1 < P$. Notice that $P$ is isomorphic to the additive group of $\mathbb{F}_q$ which is an $a$-dimensional vector space over $\mathbb{F}_{p}$, where $q=p^a$, and $P_1$ corresponds to a proper non-zero subspace of $\mathbb{F}_{p}^a.$
Since 
\[
\begin{pmatrix}
    1 &  0\\
    \lambda & 1
\end{pmatrix}^{\diag(1, \omega^r)}=\begin{pmatrix}
    1 &  0\\
   \omega^{-r} \lambda & 1
\end{pmatrix}, 
\]
$\diag(1, \omega^r)$ acts on $P$ as a linear transformation of $\mathbb{F}_{p}^a$ leaving the subspace $P_1$ invariant, so $H_\beta$ is reducible on $P$. Indeed, $\langle \diag(1, \omega^r) \rangle$ induces on $\mathbb{F}_{p}^a$ a subgroup of a Singer cycle (the cyclic subgroup generated by $\omega$ acting on $\mathbb{F}_q= \mathbb{F}_{p}^a$ via multiplication in $\mathbb{F}_q$, see the discussion after Definition \ref{des:singer}). 
\begin{itemize}
    \item By Lemma~\ref{irrsub}, there exists a proper divisor $c$ of $a$ such that $|H_\beta|=(q-1)/r$ divides $p^c-1$; in particular $a\geq c+1\geq 2$ and $c\leqslant a/2$. 
    \item Also, by Remark~\ref{r:sl}, $r$ is a primitive prime divisor of $p^{r-1}-1$ and $r-1$ divides $a/j$ where $|G:G\cap\GL_n(q)|=a/j$.
\end{itemize}
We claim that $r$ is odd, and $a=r-1 > 2$. If $r=2$ then $(p^a-1)/2$ divides $p^c-1\leq p^{a/2}-1$, which is impossible. Thus $r$ is an odd prime. 
Suppose (for a contradiction) that $a=2$. Then $r=3$ (as $r-1$ divides $a$ by the second condition above), and by the first condition, $c=1$ and $(q-1)/r=(p^2-1)/3$ divides $p-1$. This implies that $p=2$ so $(q,r)=(4,3)$,  which we assume is not the case.  Hence $a>2$. Next suppose that $(p,a)=(2,6)$. Then   $r=3$ (as $r-1$ divides $a$ and $7$ is not a primitive prime divisor of $2^{7-1}-1$), and $(p^a-1)/r = 21$ should divide $2^c-1$ for some $c\in\{1, 2,3\}$ (as $c$ is a proper divisor of $a=6$), which is impossible. Thus  $(p,a)\ne (2,6)$, and $a>2$, and it now follows from Zsigmondy's Theorem \cite{ZSIG} that $p^a-1$ has a primitive prime divisor, say $r'$.  If $r'\ne r$, then $r'$ divides $(p^a-1)/r$, which divides $p^c-1$, which is a contradiction. Thus $r'=r$ and so we must have $a=r-1$ (by the second condition). This proves the claim.

Therefore we have
\begin{equation}\label{e:2a}
    2a> a+1 = r \geq \frac{p^a-1}{p^c-1} > p^{a-c}\geq p^{a/2}\geq 2^{a/2}.
\end{equation}
 The condition $2a>2^{a/2}$ implies that $a< 8$, and since we have $a=r-1>2$ from the claim, it follows that $(a,r)=(4,5)$ or $(6,7)$. If $(a,r)=(6,7)$, then the proper divisor $c$ of $a$ satisfies $c\leq 3$, and so by \eqref{e:2a} we have  $r=7>p^{6-c}\geq p^3$, which is a contradiction. 
Hence $(a,r)=(4,5)$ and $c\in\{1,2\}$. Thus $r=5>p^{4-c}$ by \eqref{e:2a}, and so $c=2$ and $p=2$.

In this case the $H_\beta$-orbits in $\mathbb{F}_{16}^*$ are the cosets $\omega^i\mathbb{F}_4^*$ with $0\leq i\leq 4$. We also have the condition that $r-1=4$ divides $a/j=4/j$ so $j=1$  and $G$ contains an element of the form $g \phi$ with $g \in \GL_n(16)$. Since $\SL_n(16)\leq G$ and $\SL_n(16)$ is transitive on $2$-spaces, we may assume that $g$ stabilises ${\langle e_1,e_2 \rangle}$ setwise, and since $(G\cap\GL_n(16))_{\langle e_1, e_2\rangle}$ induces $\GL_2(16)$ on $\langle e_1, e_2\rangle$ (because $n \geq 3$), we may assume further that  $g|_{\langle e_1,e_2 \rangle}=I_2.$  Thus $G_{\langle e_1, e_2\rangle}$ contains an element $g \phi$ which acts on $\langle e_1, e_2\rangle$ as $\phi$. %
In fact, $g\phi\in G_{\alpha,\beta}$ and the $G_{\alpha,\beta}$-orbits in $\mathbb{F}_{16}$ are $\{0\}, \mathbb{F}_4^*, \mathbb{F}_{16}\setminus\mathbb{F}_4$. It follows that there is a unique $G_{\alpha,S}$-block of size $4$ contained in $B_2$  and containing $\beta$, corresponding to the subfield $\mathbb{F}_4$ of $\mathbb{F}_{16}$ and the stabiliser subgroup $H_B$ consists of the matrices in \eqref{e:hb2} with $\lambda\in\mathbb{F}_4$. This is the block $B=B_{6}$ in Claim $5$.  Moreover, a computation using {\sc Magma} confirms that there are no further nontrivial blocks  for $G_{\alpha,S}$ in $S$ when $(q,r)=(2^4,5)$. Recall that $(G_{\alpha,S})^\theta \leq \GaL_2(q)$ so  it is enough to verify this computationally only for $n=2$. 
This completes the proof of Claim 5.

\medskip

 Thus we may assume that $|B|=p^kr$ with  $k\geqslant 1$, and recall that $B$ is a proper subset of $S$, so $1\leqslant k < a$.
 Since $H$ has a normal Sylow $p$-subgroup $P$ which is semiregular on $S$, it follows that $H_B$ has a normal $p$-subgroup $P_1$ with $P_1\leq P$, and $P_1$ has $r$ orbits of length $p^k$ in $B$. In fact $P_1$ is characteristic in $H_B$ and hence is normal in $(G_{\alpha,S,B})^\theta$, so each $P_1$-orbit is a block for $G_{\alpha,S, B}$ in $B$ and hence for $G_{\alpha,S}$ in $S$. In particular the $P_1$-orbit $B'$ containing $\beta$ is a block of size $p^k>1$ and is properly contained in $\beta^{P}=B_2$ since $p^k<q$. Thus $(n,q,r)$ must be as in Claim $5$ and $B'=B_6$, but it was shown there that in this case there are no nontrivial blocks of size $p^kr$. 
\end{proof}

\begin{table}[h]
\begin{tabular}{ccclcc}
\hline
Line   & $q$ & $r$  & $B$  & $|B|$ & $G$  \\ 
\hline
1 & 4  & 3    &       $B_4= \{e_2, e_1 +  e_2\}$\  &2 &    $\GaL_2(4)$                \\ 
  &  &           &           $B_5= \{(1+\lambda)e_1 +\lambda e_2 \mid \lambda \in \mathbb{F}_4^*\}$&3 & \\ \hline
  2  & $2^4$  & 5    &     $B_{6}= \omr \{\lambda e_1 +e_2 \mid \lambda \in \mathbb{F}_4\}$\  &4  &  $\GaL_2(2^4)$ \\ \hline
3  & $3^4$  & 5    &     $B_{7,1}= \omr \{\lambda e_1 +e_2 \mid \lambda \in \mathbb{F}_9\}$\  &9      &   $Z \, \SL_2(3^4) \rtimes \langle \phi \rangle$           \\ 
  & &          &          $B_{7,2}= \omr \{\lambda e_1 +e_2 \mid \lambda \in \omega^5 \mathbb{F}_9\}$ &9 & \\ \hline
4  & $5^2$  & 3    &     $B_{8,1}= \omr \{\lambda e_1 +e_2 \mid \lambda \in \mathbb{F}_5\}$\  &5            &  $Z \, \SL_2(5^2) \rtimes \langle \phi \rangle$      \\ 
  & &          &          $B_{8,2}= \omr \{\lambda e_1 +e_2 \mid \lambda \in \omega^3 \mathbb{F}_5\}$ &5 & \\ \hline
  5 & $3^2$  & 2    &      $B_{9,i}= \omr \{\lambda e_1 +e_2 \mid \lambda \in \omega^i\mathbb{F}_3\}$ for $0 \leq i \leq 3$ & 3 &  $ \langle Y \, \SL_2(3^2),   \phi \, \diag(1,\omega)  \rangle$\\ \hline
\end{tabular}
\caption{Additional blocks for Theorem~\ref{th:blocksn2} }\label{t:blocksn2}
\end{table}

\begin{Th} \label{th:blocksn2}
        Let  $Y \, \SL_2(q) \leq G \leq\GammaL_2(q)$ with $G^{\Omega}$ of rank $3$ acting on $\Omega$ as in {\rm Construction \ref{con:psl}} and satisfying Hypothesis \ref{hyp1}. Let $B$ be a nontrivial block containing $\beta=\langle \omega^r \rangle e_2$ in a system of imprimitivity for $G_{\alpha}$ on $\Omega \backslash \sigma$. Then   one of the following holds.
        \begin{enumerate}[label=\normalfont (\arabic*)]
                \item $B$ is $B_1= \omr\{e_2, \omega e_2, \ldots, \omega^{r-1} e_2 \}$ of size $r$;
                \item $B$ is  $B_2=  \omr\{\lambda e_1+e_2\mid \lambda\in \mathbb{F}_q\}$ of size $q$;
             \item  $q, r$, $B$ and $G$ are as in one of the lines of Table~$\ref{t:blocksn2}$.
                          \end{enumerate}
\end{Th}

\begin{Rem}\label{rem:9iConj}
    In line 5 of Table \ref{t:blocksn2}, the group $G$ is given up to conjugation in $\GammaL_2(9)$. There are only two conjugates, the second being $G^{\phi} = \langle Y \, \SL_2(3^2),   \phi \, \diag(1,\omega^3)  \rangle.$ Notice, that the $B_{9,i}$ are still blocks for $G^{\phi}_{\alpha}$ on $\Omega$. In particular, the action of $\phi$ on $\Omega$ preserves $B_{9,0}$ and $B_{9,2}$ and swaps $B_{9,1}$ and $B_{9,3}$. It turns out (see Lemma \ref{lem:LtrSL}) that only 
    $B_{9,0}$ and $B_{9,2}$ give rise to a partial linear space via Theorem \ref{th:AD2} (2). Of course, conjugate permutation groups give rise to isomorphic partial linear spaces, so it suffices to consider only the group $G$ in Lemma \ref{lem:LtrSL} and in the proof of Theorem \ref{th:main}.
\end{Rem}

\begin{proof}
Since $n=2$, $\Omega \backslash \sigma$ is equal to $S$ of size $qr$ from Lemma \ref{lem:Ssystem}, so $G_{\alpha,S}=G_{\alpha}$ and the homomorphism $\theta$ from the first paragraph of the proof of Theorem \ref{th:blocks} is defined on $G_{\alpha},  (see \eqref{eq:ghat}).$

As in Theorem \ref{th:blocks} (see {\it Claim 2} of the proof), $B_1=\sigma(\langle e_2 \rangle)$ and is always a block for $G_{\alpha}.$
Let $P$ be the subgroup of $\SL_2(q)_{\alpha}$ consisting of matrices in \eqref{eq:Hmat} with $\mu=1.$ Again, similar to Theorem \ref{th:blocks} (see {\it Claim 1} of the proof), $P$  is the unique normal (hence characteristic) Sylow $p$-subgroup of $G_{\alpha} \cap \GL_2(q).$ Thus $P$ is normal in $G_{\alpha}$, and  it is straightforward to see that
$B_2$ is an orbit of  $P$ in $\Omega \backslash \sigma$, so $B_2$ is also always a block for $G_{\alpha}.$

 From now on we assume 
that $B\ne B_1, B_2.$ We aim to show that  one of the lines of Table \ref{t:blocksn2} must hold, and that all these cases are indeed blocks for $G_\alpha$, for the group $G$ as listed.

\medskip

Denote $K:=(G_{\alpha} \cap \GL_2(q))^{\theta}$. Our strategy for the remaining part of the proof is similar to the one we use in the proof of Theorem \ref{th:blocks}: we find the blocks $B$ for $G_{\alpha} \cap \GL_2(q)$ (more precisely, for $K$; they are the same since $\ker \theta =Y$) in $\Omega \backslash \sigma$, distinct from $B_1$ and $B_2$, and identify whether they are blocks for $G_{\alpha}$. Note that, in Theorem \ref{th:blocks} (where $n \geq 3$), $K=(G_{\alpha,S} \cap \GL_n(q))^\theta$  is always equal to the group $H$ defined in \eqref{eq:Hmat}. This is not always the case for $n=2$. 
However, we note that $K \leq (\GL_2(q)_{\alpha})^{\theta} = H$ holds.

Our analysis splits into three cases in the following way. 
Since $G \geq Y\, \SL_2(q)$ and $|Z/Y|=r$ is prime, either $G \cap Z=Z$ and $(2)$ of Lemma \ref{lem:gltrans} holds, or $G \cap Z=Y,$ so $G^{\Omega}$ is quasiprimitive by Lemma \ref{l:semi} and (2) of Remark \ref{alloccur} holds for $G_0=G^{\Omega}.$ To summarise, one of the following cases holds:
\begin{enumerate}
    \item[{\bf Case 1}.] $G \geq Z \, \SL_2(q)$ and $K=H$;
    \item[{\bf Case 2}.] $G \geq Z \, \SL_2(q)$ and  $K<H$;
    \item[{\bf Case 3}.] $G \not\geq Z \, \SL_2(q)$, $r=2$, $q \equiv 1 \pmod 4.$
\end{enumerate}
 Before we consider the cases separately, we prove the following general claim.

\medskip
\noindent
{\it Claim 1.\quad  If $|B|=p^kr$ with $k<a$, then there exists a block $B'$ for $K$ in $\Omega \backslash \sigma$ with $|B'|=p^k.$}

\medskip

 Let $P$ be as in the second paragraph of the proof. Hence $P \leq K$; in particular, $P$ is the unique normal Sylow $p$-subgroup of $K$.   It follows that $K_B$ has a normal $p$-subgroup $P_1$ with $P_1\leq P$, and $P_1$ has $r$ orbits of length $p^k$ in $B$. So each $P_1$-orbit is a block for $K_B$ in $B$ and hence for $K$ in $\Omega \backslash \sigma$. In particular the $P_1$-orbit $B'$ containing $\beta$ is a block of size $p^k>1$ and is properly contained in $B$ by construction. This completes the proof of  {\it Claim 1}.

\medskip
\noindent
{\bf Case 1}. Assume that $G \geq Z \, \SL_2(q)$ and $K=H$.

Recall that $|B|$ divides $|\Omega \backslash \sigma|=qr.$

Assume $|B|=r$ and $B \ne B_1$. Here {\it Claim 2} of the proof of Theorem \ref{th:blocks} applies directly, so either $(q,r)=(3,2)$ or $(q,r)=(4,3).$ Note that if $q=3$, then $G \leq \GaL_2(3)$ is soluble and $G^{\Omega}$ does not satisfy Hypothesis \ref{hyp1}. Hence we can assume $(q,r)=(4,3).$ Since $G \geq Z \times \SL_2(4)=\GL_2(4)$ and  $r-1=a=2$, Lemma \ref{l:semi} (4) yields $G=\GaL_2(4)$.  An easy   computation using {\sc Magma} (or repeating arguments from {\it Claim 4} of the proof of Theorem \ref{th:blocks} word for word) shows that all blocks $B$ (not equal to $B_1$ or $B_2$) for $G_\alpha$ and $(q,r)=(4,3)$ are recorded in Line 1 of Table \ref{t:blocksn2}.

Therefore, now we may assume that $(q,r) \ne (4,3)$ and $|B| \ne r.$ Next we identify all possibilities when $|B|=p^k$ for some positive integer $k$ and $B \ne B_2$. Again, since  $K=H,$  {\it Claim 5} from the proof of Theorem \ref{th:blocks} applies directly and the only possibility for such $B$ is when $(q,r)=(2^4,5)$ and $B=B_6$ as in Line 2 of Table \ref{t:blocksn2}. Note that $G \geq Z \times \SL_2(2^4)=\GL_2(2^4)$ and  Lemma \ref{l:semi} (4) yields $G=\GaL_2(2^4)$ since $r-1=a=4$. An easy   computation using {\sc Magma}  shows that the only block $B$ (not equal to $B_1$ or $B_2$) for $G_\alpha$ and $(q,r)=(2^4,5)$ is recorded in Line 2 of Table \ref{t:blocksn2}. 

There are no further possibilities for $B$ by {\it Claim 1}. This completes the consideration of {\bf Case 1}.

\medskip
\noindent
{\bf Case 2}. Assume that $G \geq Z \, \SL_2(q)$ and  $K<H$.
 This case holds precisely when $G \cap \GL_2(q) = Z \, \SL_2(q)$ and $q$ is odd.
 Indeed, following the proof of Lemma \ref{lem:gltrans}, if $G> Z \, \SL_2(q)$ then $G \cap \GL_2(q)= \GL_2(q)$, so $K=H.$ If $q$ were even then we would have $H \leq Z \, \SL_2(q)$, so $K=H$.
Here 
\begin{equation}\label{eq:Hhatmat}
    K=\widehat{H}:=\left\{ \begin{pmatrix}
    1 & 0 \\
    \lambda & \mu
\end{pmatrix} \mid \lambda \in \mathbb{F}_q, \mu \in \langle \omega^2 \rangle \right\}.
\end{equation}
    By Lemma \ref{lem:gltrans}, $\widehat{H}$ is transitive on $\Omega \backslash \sigma$ when $r$ is odd, and $\widehat{H}$ has two orbits $O_1,$ $O_2$ (see \eqref{eq:orbits}) when $r=2.$ 

    First, we consider the case $|B|=r$ and prove the following claim which is similar to {\it Claim 2} in the proof of Theorem \ref{th:blocks}.

\medskip
\noindent
{\it Claim 2.1.\quad  If $|B|=r$, then $B=B_1$.}

\medskip

Suppose first that $r$ is odd. Then $q-1$ is divisible by $2r$ as $q$ is odd. Also, as mentioned above, the group $\widehat{H}$ is transitive on $\Omega \backslash \sigma$, and hence $\widehat{H}_B$ is transitive on $B$. 
Thus there exists $h \in \widehat{H}_B$ as in \eqref{eq:h} with $\mu \in \langle \omega^2 \rangle$ such that $h$ induces an $r$-cycle on $B$. Arguing as below \eqref{eq:h}, $\lambda\ne 0$ (since $B\ne B_1$), $\mu\ne 1$ (since $r$ divides $q-1$), $\mu^r=1$ (since $h^r$ fixes $B$ pointwise), and $B$ is as in \eqref{eq:Br}. In this case $\diag(1, \omega^{2r})\in \widehat{H}$ and stabilises $\omr e_2$, and hence lies in $\widehat{H}_B.$  So the $(q-1)/2r$ points
\begin{align*}
\left(\omr(\lambda e_1 + \mu e_2)\right)\diag(1, \omega^{2ri}) & =\omr(\lambda e_1 + \omega^{2ri} \mu e_2) \\ & =\omr( \omega^{-2ri} \lambda e_1 + \mu e_2) \in B,\quad \text{ for } i \in \{1, \ldots, (q-1)/(2r)\},
\end{align*}
lie in $B$ and are pairwise distinct.  Hence $(q-1)/(2r) \le|B|-1= r-1$, so $q-1 \leq 2r^2-2r<2r^2.$ If $r^2$ divided $q-1$ then we would have $q-1=r^2$, which is odd, contradicting the fact that $q$ is odd.  Therefore  $r^2$ does not divide $q-1$, and hence the order $|\omr|=(q-1)/r$ is not divisible by $r$. Since $\mu \ne 1$ and $\mu^r=1$, this implies $\mu \not\in \omr$.
Note that the point $\omr v:=\left(\omr(\lambda e_1 + \mu e_2)\right)\diag(1, \omega^{2r})$ lies in $B$ (since $\diag(1, \omega^{2r})\in\widehat{H}$), and hence must be one of the points in  \eqref{eq:Br}. Now the coefficient of $e_2$ in $v$ must lie in $\omr \omega^{2r}\mu=\omr \mu$, and hence $\omr v$ must be the point $\omr (\lambda e_1 + \mu e_2)$ of $B$, and we conclude that $\omega^{-2r} \lambda = \lambda$, and hence $\omega^{2r}=1$ since $\lambda \ne 0$. 
 Since $q$ is odd and the odd prime $r$ divides $q-1$, it follows that 
 $r=(q-1)/2$  and so $q\equiv 3\pmod 4$.  However by Remark~\ref{r:sl},  
 $r-1$ is even and divides $a$, and this is not possible 
  since $q\equiv 1\pmod 4$ when $a$ is even and $q$ is odd. 

 Finally, we consider the remaining case where $r=2$.  In this case $K=\widehat{H}$ has two orbits $O_1, O_2$ on $\Omega \backslash \sigma$ as in \eqref{eq:orbits}. Note that $O_1=B_2$ is a block for $G_{\alpha}$ on $\Omega \backslash \sigma$ of size $q$ containing $\beta$. If $B \subset O_1$, then $\{(B)g \mid g \in G_{\alpha} \cap \GL_2(q)\}$ is a system of imprimitivity for $G_{\alpha} \cap \GL_2(q)$ on $O_1$,  and so $|B|=r=2$ divides $q$ which is a contradiction. Hence $B$ consists of one point from each of the $O_i.$ There exists $g\in G_{\alpha,B}$  such that $g$ interchanges the two points of $B$.  In view of \eqref{eq:delta1}, we may take $g$  of the form
 $$\phi^m 
 \begin{pmatrix}
1 & 0 \\
\lambda & \mu 
 \end{pmatrix}
$$ 
for some $m \in \{0,1, \ldots, a-1\},$ $\lambda \in \mathbb{F}_q$ and $\mu \in \mathbb{F}_q^*$, and $g^2$ stabilises $\beta=\langle \omega^2 \rangle e_2$, and we have 
 $$g^2 = \phi^{2m}
 \begin{pmatrix}
1 & 0 \\
\lambda^{p^m}+\mu^{p^m}\lambda & \mu^{p^m+1} 
 \end{pmatrix}.
 $$ 
Since $g^2$ fixes $\beta$ we must have $\lambda(\lambda^{p^m-1}+\mu^{p^m})=0$, so either $\lambda=0$ or $\lambda^{p^m-1}+\mu^{p^m}=0.$ If $\lambda=0,$ then $(\langle \omega^2 \rangle e_2)g= \langle \omega^2 \rangle \mu e_2$ must be the second point in $B$ and hence must lie in $O_2$, so   $\mu \in  \langle \omega^2 \rangle \omega$, and $B=B_1$. Since we are assuming that this is not the case it follows that $\lambda\ne0$  and hence $\lambda^{p^m-1}+\mu^{p^m}=0.$ 
 If $q \equiv 1 \pmod{ 4}$, then  $-1 =\omega^{(q-1)/2} \in  \langle \omega^2 \rangle$, and since $p^m-1$ is even it follows that 
 $$
 \mu^{p^m}=-\lambda^{p^m-1}\in  \langle \omega^2 \rangle
 $$ 
 which in turn implies that $\mu \in  \langle \omega^2 \rangle$ since $p^m$ is odd. Thus $( \langle \omega^2 \rangle e_2)g = \langle \omega^2 \rangle(\lambda e_1 + \mu e_2) \in O_1$ and $B \subset O_1$ which is a contradiction. Thus $q \equiv 3 \pmod 4$, so $-1 =\omega^{(q-1)/2} \in  \langle \omega^2 \rangle \omega$, and therefore also $\mu^{p^m}=-\lambda^{p^m-1}$ and $\mu$ lie in $ \omr[2]\omega.$ Since $\diag(1, \omega^2)$ stabilises $\omr[2] e_2$, it lies  in $\widehat{H}_B.$ So,
 $$
(\beta) g\, \diag(1, \omega^2)=(\omr[2](\lambda e_1 + \mu e_2))\diag(1, \omega^{2})=\omr[2](\lambda e_1 + \omega^{2} \mu e_2)=\omr[2]( \omega^{-2} \lambda e_1 + \mu e_2) \in B.
$$ 
Considering the coefficient of $e_2$, we see that this must be the point $\omr (\lambda e_1 + \mu e_2)$ of $B$, and we conclude that $\omega^{-2} \lambda = \lambda$ and $\omega^{2}=1$ since $\lambda \ne 0$. However this implies that $q-1=2=r$, and this case is excluded in Construction~\ref{con:psl} (see also the argument in Case 1). 
This completes the proof of {\it Claim 2.1}. 

\medskip 

Now we identify all possibilities when $|B|=p^k$ for some positive integer $k$ and $B \ne B_2.$ 

\medskip

\noindent\emph{Claim 2.2.\quad If $|B|=p^k$, for some positive integer $k$, then $B\subseteq B_2$, and if $B\ne B_2$ then one of the following holds:
\begin{enumerate}
    \item[$(i)$] $(q,r)=(3^4,5)$ and $B$ is $B_{7,1}= \omr \{\lambda e_1 +e_2 \mid \lambda \in \mathbb{F}_9\}$ or $B_{7,2}= \omr \{\lambda e_1 +e_2 \mid \lambda \in \omega^5 \mathbb{F}_9\}$, 
    \item[$(ii)$]   $(q,r)=(5^2,3)$ and $B$ is $B_{8,1}= \omr \{\lambda e_1 +e_2 \mid \lambda \in \mathbb{F}_5\}$ or $B_{8,2}= \omr \{\lambda e_1 +e_2 \mid \lambda \in \omega^3 \mathbb{F}_5\}$, 
\end{enumerate} 
as in Table~\ref{t:blocksn2}. Moreover, for $(q,r)$ as in $(i)$ or $(ii)$, there are no further nontrivial blocks  (including blocks of size $p^kr$).}
\medskip

Note that $\widehat{H}_{B_2}$ consists of all matrices as in \eqref{e:hb2} with $ri$ even. Thus if $r=2$ then  $\widehat{H}_{B_2}=H_{B_2}$ and it follows from {\bf  Case 1} that there are no nontrivial blocks $B$ properly contained in $B_2$. We may therefore assume that $r$ is odd, so $a$ is even as $r-1$ divides $a$.  Arguing as in {\it Claim 5} of the proof of Theorem \ref{th:blocks}, $\widehat{H}_\beta = \langle \diag(1, \omega^{2r}) \rangle$ and $\widehat{H}_B=P_1\rtimes \widehat{H}_\beta$ for some $1<P_1<P$, so $\widehat{H}_\beta$ is reducible on $\mathbb{F}_q$.  
 Thus there exists a proper divisor $c$ of $a$ such that $|\widehat{H}_\beta|=(p^a-1)/2r$ divides $p^c-1$. We therefore have similar inequalities to \eqref{e:2a} with $2r$ in place of $r$ so that $$4a>2(a+1)\geq 2r \geq \frac{p^a-1}{p^c-1}>p^{a-c}\geq 3^{a/2}.$$

 The condition $4a>3^{a/2}$ implies that $a< 6$, thus $a\in\{2,4\}$ since $a$ is even.  Assume $a=2$, so $r=3$ and $c=1$. Thus $(p^2-1)/6$ divides $p-1$, and so $p+1$ divides $6$ and $p=5$.
Now assume $a=4$ and $r=5$ and $c\in\{1,2\}$. Thus $2r=10>p^{4-c}$, and it follows that $c=2$ and $p=3$.
Thus $(a,p,r,c)=(4,3,5,2)$ or $(2,5,3,1)$. In both cases $(p^a-1)/2r=p^c-1$ and $|B|=p^c$, and the condition that $r-1$ divides $a/j$ forces $j=1$ so that some element of $G$ involves the field automorphism $\phi:x\to x^p$.  Since $\phi$ fixes $\alpha$, it follows that 
 $G_{\alpha}=\langle \widehat{H}, g\phi\rangle$ 
for some $g\in \GL_2(q)_\alpha$. 
In both cases there are two possibilities for $G$ (up to conjugation in $\GammaL_2(q)$):
\begin{enumerate}
    \item $G=(Z \, \SL_2(q)) \rtimes \langle \phi \rangle$;
    \item $G=\langle Z \, \SL_2(q)),   \phi \, \diag(1,\omega)  \rangle$.
\end{enumerate}
For the group in $(2)$, 
 it is routine to establish (using {\sc Magma}) that $(G_{\alpha})_{B_2}$ is primitive on $B_2$ so there is no smaller block $B$. So further we assume that $G$ is as in $(1)$ above. Here 
\begin{equation}\label{e:cl5-gp1}
(G_{\alpha})_{B_2}= \omr ((P \rtimes \diag(1,\omega^{2r})) \rtimes \langle \phi \rangle ).    
\end{equation}
Suppose first that $(a,p,r,c)=(4,3,5,2)$.  
Using {\sc Magma}, we establish that there are ten subgroups of $P$ of order $9$ normalised by $\diag(1,\omega^{2r})$ and only two of them are normalised by $\phi$;  
namely $\mathbb{F}_9$ and $\omega^5\mathbb{F}_9$ (recall that we identify $P$ with the additive group of $\mathbb{F}_q=\mathbb{F}_{3^4}$).
These two subgroups give rise to the blocks $B_{7,1}$ and $B_{7,2}$ in {\it Claim 2.2$(i)$}. Further investigation using {\sc Magma} shows that there are no other nontrivial blocks for $G_{\alpha,S}$ in $S$.

 Now suppose that $(a,p,r,c)=(2,5,3,1)$.  
Here the orbits of the group $(G_{\alpha})_{B_2}$ (as in \eqref{e:cl5-gp1}) in $\mathbb{F}_{25}$ are  $\{0\}, \mathbb{F}_5^*,  \omega^3\mathbb{F}_5^*, \omega\mathbb{F}_5^*\cup \omega^5\mathbb{F}_5^*$ and $ \omega^2\mathbb{F}_5^*\cup \omega^4\mathbb{F}_5^*$. It follows that there are exactly two $G_{\alpha,S}$-blocks of size $5$ contained in $B_2$ (containing $\beta$), the first corresponds to the subfield $\mathbb{F}_5$ of $\mathbb{F}_{25}$ and the stabiliser subgroup $\widehat{H}_B$ consists of the matrices in \eqref{e:hb2} with $\lambda\in\mathbb{F}_5$ and $i$ even, while the second corresponds to $\omega^3\mathbb{F}_5$ with stabiliser $\widehat{H}_B$ consisting of the matrices in \eqref{e:hb2} with $\lambda\in\omega^3 \mathbb{F}_5$ and $i$ even. These are the blocks $B_{8,1}$ and $B_{8,2}$ in {\it Claim $2.2(ii)$}.  Moreover, a computation using {\sc Magma} confirms that there are no further nontrivial blocks for $G_{\alpha,S}$ in $S$, and so {\it Claim 2.2} is proved

\medskip

There are no further possibilities for $B$ by {\it Claim 1} and {\it  Claim 2.2}, so consideration of {\bf Case 2} is complete.

\medskip
\noindent
{\bf Case 3.} Assume that $G \not\geq Z \, \SL_2(q)$, $r=2$, $q \equiv 1 \pmod 4.$ \quad 
Here $G \cap \GL_2(q)= Y \, \SL_2(q),$ so 
\begin{equation}
    K=(G_{\alpha} \cap \GL_2(q))^{\theta}=\widetilde{H}=
    \left\{ \begin{pmatrix}
    1 & 0 \\
    \lambda & \mu
\end{pmatrix} \mid \lambda \in \mathbb{F}_q, \mu \in \langle \omega^4 \rangle \right\}.
\end{equation}
As in previous cases, we assume that $B$ is a proper block for $K$ on $\Omega \backslash \sigma$ containing $\beta$ and $B \ne B_1, B_2$.

     \medskip

\noindent
     {\it Claim 3.1.\quad  If $|B|=2$, then $B=B_1$.}

\medskip

  In this case $K=\widetilde{H}$ has two orbits $O_1, O_2$ on $\Omega \backslash \sigma$ as in \eqref{eq:orbits}. Indeed, $G \cap \GL_2(q)$ is contained in $Z \, \GL_2(q),$ so it stabilises the $O_i$ by Lemma \ref{lem:gltrans}. On the other hand, $K$ clearly contains $P$ which is transitive on $O_1 =B_2.$  If $B \subset O_1$, then $\{(B)g \mid g \in G_{\alpha} \cap \GL_2(q)\}$ is a system of imprimitivity for $G_{\alpha} \cap \GL_2(q)$ on $O_1$,  and so $|B|=r=2$ divides $q$ which is a contradiction. Hence $B$ consists of one point from each of the $O_i.$ There exists $g\in G_{\alpha,B}$  such that $g$ interchanges the two points of $B$. Now $g$ is of the form
 $$\phi^m 
 \begin{pmatrix}
1 & 0 \\
\lambda & \mu 
 \end{pmatrix}
$$ 
for some $m \in \{0,1, \ldots, a-1\},$ $\lambda \in \mathbb{F}_q$ and $\mu \in \mathbb{F}_q^*$, and $g^2$ stabilises $\beta=\langle \omega^2 \rangle e_2$, and we have 
 $$g^2 = \phi^{2m}
 \begin{pmatrix}
1 & 0 \\
\lambda^{p^m}+\mu^{p^m}\lambda & \mu^{p^m+1} 
 \end{pmatrix}.
 $$ 
Since $g^2$ fixes $\beta$ we must have $\lambda(\lambda^{p^m-1}+\mu^{p^m})=0$, so either $\lambda=0$ or $\lambda^{p^m-1}+\mu^{p^m}=0.$ If $\lambda=0,$ then $(\langle \omega^2 \rangle e_2)g= \langle \omega^2 \rangle \mu e_2$ must be the second point in $B$ and hence must lie in $O_2$, so   $\mu \in  \langle \omega^2 \rangle \omega$, and $B=B_1$. Since we are assuming that this is not the case it follows that $\lambda\ne0$  and hence $\lambda^{p^m-1}+\mu^{p^m}=0.$ 
 Recall that $q \equiv 1 \pmod{ 4}$, so  $-1 =\omega^{(q-1)/2} \in  \langle \omega^2 \rangle$, and since $p^m-1$ is even it follows that 
 $$
 \mu^{p^m}=-\lambda^{p^m-1}\in  \langle \omega^2 \rangle
 $$ 
 which in turn implies that $\mu \in  \langle \omega^2 \rangle$ since $p^m$ is odd. Thus $( \langle \omega^2 \rangle e_2)g = \langle \omega^2 \rangle(\lambda e_1 + \mu e_2) \in O_1$ and $B \subset O_1$ which is a contradiction.
This completes the proof of {\it Claim 3.1}. 

\medskip 

     As we established in {\it Claim 3.1}, $\widetilde{H}$  is transitive on each of the $O_i$. Further, the stabiliser $\widetilde{H}_\beta$ of the point $\beta=\langle \omega^2\rangle e_2\in O_1$ is the group  $\langle \diag(1, \omega^4) \rangle\cong C_{(q-1)/4}$, of index $2$ in $\widehat{H}_\beta$.
     As $B$ and $O_1$ are both blocks for $G_\alpha$, the intersection $B\cap O_1$ is also a block, and as $G_\alpha$ is transitive on $\Omega \backslash \sigma$, either $B\subseteq O_1$ or $|B\cap O_1|=|B|/2 > 1$.   In either case $B':=B\cap O_1$ has size $|B'|=p^k$ for some $k\geq1$, and the block stabiliser $K_{B'} = P_1\rtimes \widetilde{H}_\beta$ for some subgroup $P_1=C_p^k$ of the normal Sylow $p$-subgroup $P=C_p^a$ of $K$. Since $B\ne B_i$ for $i=1,2$ it follows that $k<a$.  
     
     We argue as in the proof of {\it Claim 5} of Theorem~\ref{th:blocks}: conjugation by $ \diag(1, \omega^4)$ induces a linear map on $P$ equivalent to field multiplication $\lambda\to \omega^{-4}\lambda$ if we identify $P$ with the additive group of the field $\mathbb{F}_q$. As $\widetilde{H}_\beta$ normalises $P_1$ it follows that $\widetilde{H}_\beta$ acts reducibly on $P$. By Lemma~\ref{irrsub}  there is a proper divisor $c$ of $a$ such that $|\widetilde{H}_\beta|=(p^a-1)/4$ divides $p^c-1$. This implies that $p^a-1$ does not have a primitive prime divisor, and hence by Zsigmondy's Theorem \cite{ZSIG}, $a=2$ and $p+1=2^b$ for some $b\geq2$. This means that $c=1$, and so $(p^2-1)/4= 2^{b-2} (p-1)$ divides $p^c-1=p-1$. Hence $p=3$ and $q=p^a=9$.  In this case $\omega^4=-1$ and $\widetilde{H}_\beta$ induces $-I\in\GL_2(3)$ on $P$ and normalises each of the four subgroups $P_1$ of $P$ of order $3$. This gives four distinct possibilities for the block stabiliser $K_{B'}$ corresponding to the four blocks  $B'=B_{9,i}$ ($0\leq i\leq 3$) as in  Table~\ref{t:blocksn2}. If $B \subset O_1,$ then $B=B'$ and the $B_{9,i}$ $(0\leq i \leq 3)$ are indeed blocks for $G_{\alpha}$ on $\Omega \backslash \sigma$ which is easy to verify using {\sc Magma}. Using  Remark \ref{alloccur}, we deduce $G=  \langle Y \, \SL_2(3^2),  \phi \, \diag(1,\omega)  \rangle$  up to conjugation in $\GammaL_2(3^2)$. 

     Finally, assume that $B\not\subset O_1$. Then as $B'=B\cap O_1$ is a nontrivial block, we have $q=9$ and $B'=B_{9,i}$ for some $i$. Thus $|B|=6$, and a computation with  {\sc Magma} shows that there is no such block. This completes the proof of {\bf Case 3} and Theorem~\ref{th:blocksn2}.
\end{proof}

\section{Systems of imprimitivity for point stabilisers of unitary groups}\label{sec:unitary}

 In this section we investigate systems of imprimitivity of $G_{\alpha}$ on $\Omega \backslash \sigma$ for $Z \, \SU_3(q) \leq G \leq \GammaU_3(q)$ of rank $3$ as in Construction \ref{con:psu}.  We assume that $G/Y$ is properly innately transitive of rank 3, so all conditions on $q,r,n$ in  the third line of Table \ref{t:qpinsprk3} hold,  that is,  $r$ is an odd prime dividing $q-1=p^a-1$, $o_r(p)=r-1$, and 
 $|G/(G \cap \GL_n(q^2))|=2a/j$ where $j$ divides $2a$ and $(j,r-1)=1$. 
 Observe  $o_r(p)=r-1$ divides $a$ since $r$ divides $p^a-1=q-1$. This implies that $a>1$ since $r>2$. 
 Recall from the beginning of Section \ref{sec:spacesU} that the natural module $V$ for $\GU_n(q)$ has basis $\{e,x,f\}$ as in \eqref{unbasis}.

\begin{table}[h]
\begin{tabular}{ccclcc}
\hline
Line    & $q$ & $r$  & $B$  & $|B|$ & $G$  \\ 
\hline
1  & 4  & 3    &       $B_5=\omr\{\lambda f + (1-\lambda )e \mid \lambda \in \mathbb{F}_4^* \}$   & 3    &  $\GammaU_3(4)$             \\ 
  &     &         &           $B_6= \omr\{f, f+e\}$        &2  &\\ \hline
2  & 16  & 5    &       $B_7=\omr\{ f+\lambda e \mid \lambda \in \mathbb{F}_{4}\}$\  &4    & $\GammaU_3(16)$                \\  \hline
\end{tabular}
\caption{Additional blocks for Theorem~\ref{th:blocksSU} }\label{t:blocksSU}
\end{table}

\begin{Lem}\label{lem:ZSUtrans}
    Let $\alpha=\omr e$ and let $r$ be an odd prime dividing $(q-1)$. Then $(Z \, \SU_3(q))_{\alpha}$ is transitive on $\Omega \backslash \sigma.$
\end{Lem}
\begin{proof}
    Let $\omr v \in \Omega \backslash \sigma$, so $v=\lambda_1 e+ \lambda_2 x+ \lambda_3f$ such that $\lambda_1\lambda_3^q + \lambda_2^{q+1}+ \lambda_3 \lambda_1^q=0$ and $v \notin \langle e \rangle$. In particular, $\lambda_3 \ne 0.$ Let 
    $$g = \lambda I\cdot
    \begin{pmatrix}
        (\lambda_3 \lambda^{-1})^{-q} & 0 & 0 \\
         - \lambda_2^q\lambda_3^{-1} \lambda^{1-q}  & (\lambda_3\lambda^{-1})^{q-1} & 0 \\
        \lambda_1\lambda^{-1} & \lambda_2\lambda^{-1} & \lambda_3\lambda^{-1}
    \end{pmatrix}
    $$
    where $\lambda \in \mathbb{F}_{q^2}^*$ is such that $\lambda_3^{-q} \lambda^{q+1} \in \omr.$ Such $\lambda$ always exists since $(r,q+1)=1.$ It is routine to check that $(\lambda I)^{-1}g \in \SU_3(q)$ and $(\omr e)g=\omr e$. Hence $g \in (Z \, \SU_3(q))_{\alpha}$ and $(\omr f)g=\omr v$, so all elements of $\Omega \backslash \sigma$ lie in the same orbit as $\omr f$ and the lemma follows.
\end{proof}

 In the following theorem, for $\gamma= \omr v \in \Omega$, we use $\sigma_{\gamma}$ to denote the block  $ \sigma(\langle v \rangle) \in \Sigma$ (see Construction \ref{con:psu}). So $\sigma_{\gamma}$ is the block in $\Sigma$ containing $\gamma$.

\begin{Th}\label{th:blocksSU} Let  $Z \, \SU_3(q) \leq G \leq \GammaU_3(q)$ with $G^{\Omega}$ of rank 3 and let $\alpha=\omr e$. Let $B$ be a block containing $\langle \omega^r \rangle f$ in a system of imprimitivity of $G_{\alpha}$ on $\Omega \backslash \sigma$. Then  one of the following holds.
    \begin{enumerate}[label=\normalfont (\arabic*)]
    \item $B$ is $B_1= \omr\{f +v \mid v \in \la x,e \ra \text{ with } (f+v,f+v)=0  \}$ of size $q^3$;
    \item $B$ is $B_2= \omr \{f +b  e \mid b\in \mathbb{F}_{q^2}, \mathrm{Tr}(b)=0\}$ of size $q$;
    \item $B$ is $B_3= \omr\{f, \omega f, \ldots, \omega^{r-1} f \}$ of size $r$;
    \item $B$ is $B_4=\cup_{\gamma \in B_2} \sigma_{\gamma}$ of size $qr$;
    \item $q,r$ and $B$ are as in one of the lines of Table \ref{t:blocksSU}.
    \end{enumerate}
\end{Th}
\begin{proof}
   Let $\beta:=\langle \omega^r \rangle f$. We first describe the elements of $G_\alpha$.

\medskip\noindent
\emph{Claim 1.}\quad
Elements of $G_{\alpha}$ have the form 
\begin{equation}\label{eq:unist}
\lambda I \cdot 
\begin{pmatrix}
    z & 0 & 0 \\
    d & y & 0 \\
    b & c & z^{-q}
\end{pmatrix} \phi^i 
\end{equation}
where $\lambda, b,c, d,z, y \in \mathbb{F}_{q^2}$ with $\lambda z \in \omr,$ $y^{q+1}=1$, 
$\mathrm{Tr}(bz^{-1}) + c^{q+1}=0,$ and $dz^{-1}+yc^q=0.$ 

\medskip
Indeed, all elements in $\GammaU_3(q)$ have shape $(\lambda I) g \phi^i$ where $\lambda \in \mathbb{F}_{q^2}^*$, $g \in \GU_3(q)$ and $i \in \{0, 1, \ldots, 2a-1\}$  by definition. Assume that such an element stabilises $\alpha=\omr e$. Since $\phi$ stabilises $\alpha$, so does $(\lambda I) g \in \GL_3(q^2)$.  In particular, $g \in \GU_3(q)$ stabilises the subspace $\langle e \rangle<V$ and, hence also $\langle e \rangle^\bot=\langle e, x \rangle$, so $g$ is lower-triangular. The conditions on $\lambda z$ and the entries of $g$ now follow directly from the facts that $(\lambda I) g$ stabilises $\alpha$ and that the inner product $(ug,vg)=(u,v)$ for all $u,v \in V.$ Thus Claim 1 is proved.

\medskip\noindent
\emph{Claim 2.}\quad  $B_1, \dots, B_4$ are blocks of imprimitivity for the action of $G_{\alpha}$ on $\Omega \backslash \sigma$. 

\medskip
In particular, it follows from Claim 1 that $G_{\alpha}$ has a normal subgroup $P$  of order $q^3$, which is the unique Sylow $p$-subgroup of $G_{\alpha} \cap \GL_3(q^2)$, and consists of all matrices of the form 
\begin{equation}\label{eq:Pmat}
\begin{pmatrix}
    1 & 0 & 0 \\
    -c^q & 1 & 0 \\
    b & c & 1
\end{pmatrix}    
\end{equation}

with  $\mathrm{Tr}(b)+c^{q+1}=0.$ The centre $R$ of $P$ is the normal subgroup of $G_{\alpha}$ consisting of the matrices 
$$
\begin{pmatrix}
    1 & 0 & 0 \\
    0 & 1 & 0 \\
    b & 0 & 1
\end{pmatrix}
$$ 
with $\mathrm{Tr}(b)=0$; note that $|R|=q.$
Hence the orbits of $P$ and $R$ on $\Omega \backslash \sigma$ form systems of imprimitivity. It is routine to verify that $B_1$ and $B_2$ are the blocks of these imprimitivity systems containing $\beta=\omr f$. 
Next, it is easy to see that $B_3$ is the block $\sigma_\beta$ containing $\beta$ of the system of imprimitivity $\Sigma$ of $G$ on $\Omega$ (see Construction \ref{con:psu}), so it is a block of the system of imprimitivity $\Sigma \backslash \sigma$ of $G_{\alpha}$ on $\Omega \backslash \sigma$.

Notice that $B_4$ is a disjoint union of blocks in $\Sigma$ by definition. On the other hand, $\cup_{\gamma \in B_2} \sigma_{\gamma}=\cup_{i=0}^{r-1} \omega^i B_2.$ So $B_4$ is also a disjoint union of blocks of the system $\{(B_2)g \mid g \in G_{\alpha}\}.$ 
Since $G_{\alpha}$ preserves these systems, for each $g\in G_\alpha$,  $(B_4)g$ is also a disjoint union of blocks of $\Sigma$ and a  disjoint union of blocks of $\{(B_2)h \mid h \in G_{\alpha}\}.$ In particular,
$$(B_4)g= \cup_{\gamma \in (B_2)g}\sigma_{\gamma}=\cup_{i=1}^{r-1} \omega^i (B_2)g.$$

Assume that  $ \gamma \in (B_4)g_1 \cap (B_4)g_2 \ne \emptyset$ for some $g_1, g_2 \in G_{\alpha},$ so $\sigma_{\gamma}$ is a subset of both $(B_4)g_1$ and $(B_4)g_2$. Without loss of generality, we may assume that $\gamma \in (B_2)g_1$. Hence $(B_2)g_1 \subset (B_4)g_2$ and, therefore,  $\omega^i(B_2)g_1 \subset (B_4)g_2$
 for $0 \leq i \leq r-1$ and $(B_4)g_1=(B_4)g_2$ (since these sets have the same cardinaltiy). So $\{(B_4)g \mid g \in G_{\alpha}\}$ is a system of imprimitivity. Thus Claim 2 is proved.

 \medskip
To proceed further we examine more carefully the elements of $G_{\alpha}$; recall that they have shape \eqref{eq:unist}. In particular, $G_{\alpha} \cap Z \, \SU_3(q)$ consists of all elements of the shape 
\begin{equation}\label{eq:unist1}
\lambda I \cdot 
\begin{pmatrix}
    z & 0 & 0 \\
    \tilde{d} & z^{q-1} & 0 \\
    \tilde{b} & \tilde{c} & z^{-q}
\end{pmatrix} 
\end{equation}
where $\lambda, \tilde{b},\tilde{c},\tilde{d},z \in \mathbb{F}_{q^2}$ with $\lambda z \in \omr,$  
$\mathrm{Tr}(\tilde{b}z^{-1}) + \tilde{c}^{q+1}=0,$ and $\tilde{d}z^{-1}+z^{q-1}\tilde{c}^q=0.$
Since $\lambda z \in \omr,$ there exists $\lambda_1 \in \omr$ such that $\lambda_1 \lambda z=1$. If $g$ has the shape \eqref{eq:unist1}, then $g$ and $\lambda_1I\cdot g$ induce the same action on $\Omega$ since $\lambda_1 I$ lies in the kernel $Y=\langle \omega^r I \rangle$ of the $G$-action on $\Omega$. Hence the subgroup $H$ of $G_{\alpha} \cap Z \, \SU_3(q)$ consisting of all elements of shape \eqref{eq:unist1} with $\lambda z=1$ acts on $\Omega \backslash \sigma$ transitively by Lemma \ref{lem:ZSUtrans}, and this action is faithful and $H$ preserves all the systems of imprimitivity that $G_{\alpha} \cap Z \, \SU_3(q)$ does. Since $\lambda=z^{-1}$ for elements in $H$, we may write each element of $H$ as a matrix of the following form (replacing each entry $u$ of the matrix in \eqref{eq:unist1} by $\lambda u=z^{-1}u$): 
\begin{equation}\label{eq:elH}
\begin{pmatrix}
    1 & 0 & 0 \\
    d & z^{q-2} & 0 \\
    b & c & z^{-(q+1)}
\end{pmatrix}    
\end{equation}
and the conditions on these entries become:
$b,c,d,z \in \mathbb{F}_{q^2}$ with $z \ne 0$,
$\mathrm{Tr}(b)z^{-(q+1)}+c^{q+1}=0$ and $dz^{-(q+1)}+ z^{q-2}c^q=0.$ 
(Warning: the entries $b,c$ and $d$ are not the same as the corresponding ones in \eqref{eq:unist}.)  Observe that, by Lemma \ref{l:trace}, $d$ is uniquely determined in terms of $c$ and $z$, and for each pair $(c,z)$ there are exactly $q$ choices of $b \in \mathbb{F}_{q^2}$ satisfying the conditions above. Hence $|H|= q^3 (q^2-1).$

\medskip
 \noindent\emph{Claim 3.}\quad  $G_{\alpha}$ has no other systems of imprimitivity on $\Omega \backslash \sigma$ than those in Claim 2, unless $(q,r)$ is either $(4,3)$ or $(16,5).$ 

 \medskip
Assume that $B$ is a block of a system of imprimitivity of $H$ on $\Omega \backslash \sigma$ containing $\beta$ and $B$ is not one of $B_1, \dots,B_4$. Then $|B|$ divides $|\Omega \backslash \sigma|=q^3 r.$ It is useful to know $H_{\beta}$ since $H_B>H_{\beta}.$ If $g$ of shape \eqref{eq:elH} stabilises $\beta=\omr f$, then $b=c=d=0$ and $z^{-(q+1)} \in \omr.$ Since $(r,q+1)=1,$ it follows that $H_{\beta}$ is generated by the element
 \begin{equation}\label{eq:hmatuni}
 h=\begin{pmatrix}
    1 & 0 & 0 \\
    0 & \omega^{-r(q-2)} & 0 \\
    0 & 0 & \omega^{r(q+1)}
\end{pmatrix}
\end{equation}
of order $\frac{q-1}{r} \cdot \frac{q+1}{(3,q+1)}.$ Notice that $$h^{(q-1)/r}=
\begin{pmatrix}
    1 & 0 & 0 \\
    0 & \omega^{3(q-1)} & 0 \\
    0 & 0 & 1
\end{pmatrix}
$$ is of order $(q+1)/(3,q+1)$ and it stabilises $f$. 

We subdivide the proof of Claim 3 into three cases: $|B|=r$, $|B|=p^k$ for some $k \leq 3a$, and $|B|=p^kr$ for some $k < 3a$

{\bf Case 1.} Assume that $|B|=r,$ so $H_B$ contains an element $g$  of shape \eqref{eq:elH}  inducing on $B$ a permutation of order $r$. (We will show that either $B=B_3$ or $(q,r)=(4,3)$.) 

\medskip
The element $g^r$ must stabilise $\omr f$, so $(g^r)_{3,2}=(g^r)_{2,1}=0$, where the equality for $(g^r)_{2,1}$ follows from the facts that $g^r \in H$, so the conditions after \eqref{eq:elH} hold for it. Direct calculations shows that
 \[ 
 (g^r)_{2,1}  = (1+z^{q-2}+z^{2(q-2)}+\cdots+ z^{(r-1)(q-2)})d.
 \]

\medskip\noindent\emph{Subclaim 1.} $d=0$ for such $g$.

 \medskip
Assume first that $z^{q-2}=1$. 
 Thus $(g^r)_{2,1}  = (1+ \ldots +1)d=(r \pmod p) \cdot d$. Hence $d=0$ since $r$ divides $(q-1)$ and  $(r \pmod p)\ne 0$. This proves the Subclaim 1 when  $z^{q-2}=1.$

Now assume that  $z^{q-2}\neq 1$. Then 
\[
(g^r)_{2,1}= \left(\frac{z^{r(q-2)}-1}{z^{q-2}-1}\right)d.
\]
 If $d=0$, then Subclaim 1 follows, so let us assume that $d\neq 0$,  so $z^{r(q-2)}=1$. It follows that $c\neq 0$ and $b\neq 0$ since  $dz^{-(q+1)}+ z^{q-2}c^q=0$ and $\mathrm{Tr}(b)z^{-(q+1)}+c^{q+1}=0.$ 
Note that $$(r(q-2),q^2-1)=r(q-2,\frac{q-1}{r}\cdot(q+1))=r(q-2,q+1)=r(3,q+1),$$ while $$(q-2,q^2-1)=(q-2,(q-1)(q+1))=(q-2,q+1)=(3,q+1).$$
 Thus  $z \in \langle \omega^{\frac{q^2-1}{(3,q+1)r}} \rangle\setminus \langle \omega^{\frac{q^2-1}{(3,q+1)}} \rangle$. In particular the order $|z|$ of $z$  divides $(3,q+1)r$ but does not divide  $(3,q+1)$, so $r$ divides $|z|$.

Since $h \in H_B$,
\begin{equation}\label{znotinomr}
\begin{aligned}
    (\omr f)gh^i & = \omr(z^{-(q+1)}f+cx +be)h^i \\ & =\omr(\omega^{ir(q+1)}z^{-(q+1)}f+ \omega^{-ir(q-2)}cx +be) 
 \in B \text{ for } i \in \{1, \ldots, (q-1)/r\}.  
\end{aligned}
\end{equation}
Recall that $b,c\neq 0$. Since the coefficients of $e$ of the above points are the same and the coefficient of $x$ are all distinct and non-zero, we conclude that $(q-1)/r \leq |B|-1= r-1$, so $q-1 \leq r^2-r<r^2.$ This implies that $z \not\in \omr$, as otherwise $|z|$  divides $\frac{q^2-1}{r}$, so $r^2$  divides $q-1$ which is a contradiction.
Since $q+1$ is coprime with $r$, we also obtain that $z^{q+1} \not\in \omr$.
Hence the $z^{i(q+1)}$ lie in pairwise distinct cosets of $\omr$, for $i \in \{0, 1, \ldots, r-1\}.$ Since $g$ induces a permutation of order $r$ on $B$, we obtain  
$$B=\omr\{f, z^{-(q+1)}f+cx +be, z^{-2(q+1)}f + (\ldots), \ldots, z^{-(r-1)(q+1)}f +(\ldots) \}$$
where $(\ldots)$ is some linear combination of $x$ and $e$.
It follows that for distinct points of $B$, say $\omr v$ and $\omr v'$, the coefficients of $f$ in $v, v'$ must lie in different cosets of $\omr$ in $\mathbb{F}_{q^2}^*$. 

Note that $\beta^{gh}=\omr(z^{-(q+1)}f+cx +be)h\in B$ and we compute that this point is $$\omr(z^{-(q+1)}\omega^{r(q+1)}f+c\omega^{-r(q-2)}x +be)=\omr(z^{-(q+1)}f+c\omega^{-r(q-2)}\omega^{-r(q+1)}x +b\omega^{-r(q+1)}e).$$ Since the coefficient of $f$ of this point is in the same $\omr$-coset as $\omr(z^{-(q+1)}f+cx +be)$, we see that they must be the same point of $B$.
In particular, $b\omega^{-r(q+1)}=b$ and so $\omega^{-r(q+1)}=1$ (since $b\neq 0$). It follows that  $r=q-1$.
Therefore also $c\omega^{-r(q-2)}=c$ and so $\omega^{-r(q-2)}=\omega^{-(q-1)(q-2)}=1$ (since $c\neq 0$). It follows that $q+1$ divides $q-2$, and so $q=2$ and $r=q-1=1$, which is a contradiction.  Hence $d=0$ and this proves Subclaim 1. 

\medskip
We also have that $c=0$ since $dz^{-(q+1)}+ z^{q-2}c^q=0$, and $\mathrm{Tr}(b)=0$ since $\mathrm{Tr}(b)z^{-(q+1)}+c^{q+1}=0.$ 
Therefore,
\begin{equation}\label{eq:grun}
    g^r=\begin{pmatrix}
    1 & 0 & 0 \\
    0 & z^{r(q-2)} & 0 \\
    (z^{-(r-1)(q+1)}+z^{-(r-2)(q+1)}+ \ldots+ z^{-(q+1)}+1)b & 0 & z^{-r(q+1)}
\end{pmatrix} 
\end{equation}
fixes $B$ pointwise, and so $(z^{-(r-1)(q+1)}+z^{-(r-2)(q+1)}+ \ldots+ z^{-(q+1)}+1)b=0$.

Assume first that $b=0$. Since $g$  induces on $B$ a permutation of order $r$, we obtain $B=\omr\{f,z^{-(q+1)}f, \ldots,z^{-(r-1)(q+1)}f \}$. Thus $z\notin\omr$ and the coefficients of $f$ for the points of $B$ are in pairwise distinct cosets of $\omr$. In other words, $B=B_3$.

Now assume $b\neq 0$, so  $z^{-(r-1)(q+1)}+z^{-(r-2)(q+1)}+ \ldots+ z^{-(q+1)}+1=0$. Note that $z^{-(q+1)} \ne 1$ as otherwise $$z^{-(r-1)(q+1)}+z^{-(r-2)(q+1)}+ \ldots+ z^{-(q+1)}+1 = 1+ \ldots +1=(r \pmod p)\neq 0.$$ 
Therefore $$z^{-(r-1)(q+1)}+z^{-(r-2)(q+1)}+ \ldots+ z^{-(q+1)}+1=\frac{z^{-r(q+1)}-1}{z^{-(q+1)}-1},$$ and so 
 $z^{r(q+1)}=1.$  In particular, $r$ divides $|z|$ since $z^{q+1} \ne 1$ as we note in the beginning of the paragraph.
 As in the above proof of Subclaim 1, we consider $(\beta){gh^i}=(\omr f)gh^i$ for $i \in \{1, \ldots, (q-1)/r\}$ 
 and note that the coefficients of $f$ in \eqref{znotinomr} are distinct for distinct values of $i$, while the coefficients of $e$ are equal to $b$ for each $i$. Hence, all points in \eqref{znotinomr} are distinct and $(q-1)/r \leq |B|-1$, so $q-1<r^2$, and in particular $r^2$ does not divide $q-1$. This implies that    $z \not\in \omr$ as otherwise $r^2$ divides $q-1$ which is a contradiction.  In particular, $z^{q+1}\notin \omr$ and  $z^{i(q+1)}$ lie in distinct cosets of $\omr$ for $i \in \{0, 1, \ldots, r-1\}.$ 
 Since $g$ induces a permutation of order $r$ on $B$, we obtain  
\begin{equation}\label{eq:BCase1}
\begin{aligned}
B= \omr\{f, z^{-(q+1)}f+be,  z^{-2(q+1)}f+(z^{-(q+1)}+1)be,  \ldots \\ \ldots,  z^{-(r-1)(q+1)}f+(z^{-(r-2)(q+1)}+\ldots + z^{-(q+1)}+1)be \}.
\end{aligned}
\end{equation}
In particular, $ \omr (z^{-(q+1)}f+be)h=\omr (\omega^{r(q+1)}z^{-(q+1)}f+be)= \omr (z^{-(q+1)}f+\omega^{-r(q+1)}be)$ lies in $B$ and so $\omega^{-r(q+1)}b=b.$ So, since $b \ne 0,$ we obtain $\omega^{r(q+1)}=1$   and $r=q-1$. 
 Recall from the beginning of the section that $r$ is an odd prime and $r-1$ divides $a>1$ since $G^{\Omega}$ is rank $3$.
 Thus $r-1=p^a-2$ divides $a$. It follows that  $q=4$ and $r=3$.  Therefore, there are no options for $B$ of size $r$ other than $B_3$ unless $q=4$ and $r=3.$

\medskip

{\bf Case 2.} Assume that $|B|=p^k$ for some $k\leq 3a.$ (We will show that $B$ is $B_1$ of size $q^3$, or $B_2$ of size $q$, or  $(q,r) \in \{(4,3), (16,5)\}.$)

\medskip
Since $|B|=p^k$, the Sylow $p$-subgroup $P_1$  of $H_B$ is a subgroup of $P$ of order $p^k.$ If $P_1=P$, then $(\omr f)S=B_1,$ and $B=B_1$, so we may assume that $k<3a.$ First, assume that $P_1$ does not lie in the subgroup $R$ defined in Claim 2, so $P_1R/R\ne 1.$ 
Recall that $P$ consists of matrices of the form \eqref{eq:Pmat}
with  $\mathrm{Tr}(b)+c^{q+1}=0.$
Observe that two elements $$ R \begin{pmatrix}
     1 & 0 & 0 \\
    -c^q & 1 & 0 \\
    b& c & 1
\end{pmatrix}\text{ and } R \begin{pmatrix}
     1 & 0 & 0 \\
    -c'^q & 1 & 0 \\
    b'& c' & 1
\end{pmatrix}$$ of $P/R$ are equal if and only if $c=c'$, since $c=c'$ yields $\mathrm{Tr}(b-b')=0.$  
Moreover 
$$ R \begin{pmatrix}
     1 & 0 & 0 \\
    -c^q & 1 & 0 \\
    b& c & 1
\end{pmatrix}\cdot R \begin{pmatrix}
     1 & 0 & 0 \\
    -c'^q & 1 & 0 \\
    b'& c' & 1
\end{pmatrix}=R \begin{pmatrix}
     1 & 0 & 0 \\
    -(c+c')^q & 1 & 0 \\
    b+b'-cc'^q& c+c' & 1
\end{pmatrix}. $$
So $P/R$ is isomorphic to the additive group of $\mathbb{F}_{q^2}$ which is a $2a$-dimensional vector space over $\mathbb{F}_{p}.$ 
In particular,  $P_1R/R$ is a subspace of this vector space.  Since $H_{\omr f} \leq H_B$ and $P$ is normal in $H$, $H_{\omr f}$ stabilises $P_1R/R$ in the   action on $P/R$ induced by conjugation.  Let  $h$ be as in \eqref{eq:hmatuni}, so  
 $$R\begin{pmatrix}
     1 & 0 & 0 \\
    -c^q & 1 & 0 \\
    b & c & 1
\end{pmatrix}^h =R
\begin{pmatrix}
     1 & 0 & 0 \\
    -\omega^{r(q-2)}c^q & 1 & 0 \\
    \omega^{-r(q+1)}b & \omega^{-r(2q-1)}c & 1
\end{pmatrix}
$$
and the action of $h$ on $P/R$ is equivalent to the action on $\mathbb{F}_{p^{2a}}$ induced by multiplication by $\omega^{-r(2q-1)}.$ Hence $\langle h \rangle$ induces a subgroup $\langle \hat{h} \rangle$ of a Singer cycle of $\GL_{2a}(p)$ of order  $|\omega^{-r(2q-1)}|=\frac{q^2-1}{(r(2q-1),q^2-1)}$. Note that $(r(2q-1),q^2-1)=r(2q-1,\frac{q-1}{r}\cdot(q+1))=r(2q-1,q+1)=r(3,q+1).$  Since $h$ stabilises  $P_1R/R$ (which is a subspace of $P/R$), $ \langle \hat{h} \rangle $ is reducible and its order has to divide $p^\ell-1$ for some proper divisor $\ell$ of $2a$ by Lemma \ref{irrsub}. So, $(q^2-1)/(r(3,q+1)) \leq p^\ell-1$ and 
\begin{equation*}
    r(3,q+1) \geq (p^{2a}-1)/(p^\ell-1)>p^{2a-\ell}\geq p^a.
\end{equation*}
 Recall from the beginning of the section that $r$ is an odd prime  dividing $q-1$, and $r-1$ divides $a>1$ since $G^{\Omega}$ is rank $3$. In particular, $r\leq a+1,$ so 
$$(a+1)(3,q+1) \geq p^a.$$ 
It follows that $q=8$ and $a=3$.  The only odd prime $r\leq a+1=4$ is $3$, which does not divide $q-1=7$. 
Thus there is no block $B$ of size $|B|=p^k$ other than $B_1$ in the case where $P_1 \not\leqslant R$.

\medskip
Now we assume that $P_1\leqslant R$; if $P_1=R,$ then $B=B_2.$ So assume that $P_1<R$ and $|B|<|B_2|=p^a$. Recall that $R$ consists of all matrices 
$$\begin{pmatrix}
     1 & 0 & 0 \\
    0 & 1 & 0 \\
    b & 0 & 1
\end{pmatrix}$$ 
with $b \in \mathbb{F}_{q^2}$ such that $\mathrm{Tr}(b)=0$.  Thus $R$ is elementary Abelian of order $p^a$ and can be considered as an $a$-dimensional space over $\mathbb{F}_{p}.$ As before, $P_1$ is a subspace of $R$ stabilised by $h$. Observe that 
$$\begin{pmatrix}
     1 & 0 & 0 \\
    0 & 1 & 0 \\
    b & 0 & 1
\end{pmatrix}^h =
\begin{pmatrix}
     1 & 0 & 0 \\
    0 & 1 & 0 \\
    \omega^{-r(q+1)}b & 0 & 1
\end{pmatrix},
$$ 
so $h$ induces a linear invertible map $\hat{h}$ of order $(p^a-1)/r$, regarding $R$ as an $\mathbb{F}_{p}-$vector space. Hence $\langle \hat{h}\rangle$  is a subgroup of a Singer cycle of $\GL_a(p)$ that stabilises a proper subspace. Therefore, $|\hat{h}|=(p^a-1)/r$ divides $p^\ell-1$ for some proper divisor $\ell$ of $a$ by Lemma \ref{irrsub}.
So, $(p^a-1)/r \leq p^\ell-1$ and 
\begin{equation}\label{eq:r0un}
    r\geq (p^{a}-1)/(p^\ell-1)>p^{a-\ell}\geq p^{a/2}.
\end{equation}
 Recall that $r$ is an odd prime dividing $q-1$ and $r-1$ divides $a>1$. In particular, $r\leq a+1,$ so \eqref{eq:r0un} implies that
$$a+1 \geq p^{a/2}.$$ 
It follows that $q\in\{4,8,9,16,32\}$.
If $q\in \{8,32\}$, then $q-1$ is prime so $r=q-1$, and $r-1$ does not divide $a$. If $q=9$, then no odd prime divides $q-1=8$.
Thus $q\in \{4,16\}.$
If $q=4$, we immediately get that $(q,r,\ell)=(4,3,1)$, which satisfies all the conditions.
If $q=16$, then $r \mid 15$ so $r\in\{3,5\}$. On the other hand, $15/r$ divides $2^\ell-1$ for some proper divisor $\ell$ of $4$, and so $r=5,\ell=2$. 
Thus  $(q,r) \in \{(4,3), (16,5)\}.$

\medskip

{\bf Case 3.} Assume that $|B|=p^kr$ where $0<k<3a.$ (We will show that $B$ is $B_4$ of size $qr$,  or  $(q,r) \in \{(4,3), (16,5)\}.$)

\medskip

 Here $H_B$ has a normal Sylow $p$-subgroup $P_1$ of size $p^k$. The orbits of $P_1$ form a system of imprimitivity  for $H_B$ on  $B$.  Let $B'$ be a block of this system containing $\beta$. Acting on these orbits by $H$ we get a system of imprimitivity of $H$ on   $\Omega\setminus \sigma$ with blocks of size $p^k$, and  the analysis of Case 2 forces either $p^k=q$ and $B'=B_2$ or 
$(q,r)$ = $(4,3)$ or $(16,5)$. 

Assume that $p^k=q$ and $B'=B_2$. In particular, $R \leq H_{B_2} \leq H_B$ and $|H_B| = |H|/q^2 = q(q^2-1).$ Since $R$ is normal in $H_B$, we obtain $H_B= R \rtimes C$ for some $C \leq H_B$ of order $q^2-1$ by the Schur–Zassenhaus Theorem.  One can easily check that $H= P \rtimes \langle t \rangle$ where $t = \diag(1, \omega^{q-2}, \omega^{-(q+1)})$ of order $q^2-1.$ Let $\pi$ be the set of prime divisors of $q^2-1,$ so $\langle t \rangle$ and $C$ are Hall $\pi$-subgroups of $H$ that are conjugate since $H$ is soluble.  Therefore, all subgroups of $H$ of the form $R \rtimes C$ with $|C|=q^2-1$ are conjugate to $R \rtimes \langle t \rangle$ and are stabilisers of blocks of the system $\{(B)g \mid g \in H\}.$ It is routine to check that $R \rtimes \langle t \rangle=H_{B_4}$ and, hence, $\{(B)g \mid g \in H\}= \{(B_4)g \mid g \in H\}.$

\medskip
The analyses of Cases 1--3 complete the proof of Claim 3 that $B$ is either one of $B_1, \ldots, B_4$, or $(q,r) \in \{(4,3), (16,5)\}.$

\medskip

 Now we describe blocks $B$ arising when $(q,r)$ is either $(4,3)$ or $(16,5)$ and  $B$ is not one of $B_1, \ldots, B_4$.  Since $G$ is rank $3$, $(j,r-1)=1$ where $|G^{\Sigma}/(G^{\Sigma} \cap \PGU_3(q))|=2a/j,$ so $j=1$ in both cases. Since $(3,q+1)=1$ in both cases, $Z \, \SU_3(q)= Z \, \GU_3(q)$ and, therefore, $G=\GammaU_3(q).$ 

 Assume $|B|=r.$ By the arguments of Case 1, $(q,r)=(4,3)$ and the element $g$ inducing a permutation of order $r$ on $B$ satisfies $g_{2,1}=g_{3,2}=0$, $g_{3,1}\neq 0$. Moreover $B$ is as in \eqref{eq:BCase1}. Setting $\mu:=z^{-(q+1)}=z^{-5}$ we obtain
$$B= \omr\{f, \mu f+be, \mu^2f+(1+\mu)be\},$$
where  $\mathrm{Tr}(b)=0$ and $\mu$ has order $r=3.$
In particular $b,\mu\in \mathbb{F}_4^*=\omr[5]$ with $\mu\neq 1$.

 Observe that $\phi$ stabilises $f$ so, since $G = \GammaU_3(4),$ $\phi$ stabilises $B$. In particular
$$\omr(\mu f+be)\phi =\omr( \mu^2 f+b^2 e) \in B$$
and $b^2 =(1+\mu)b$. Hence $b=1+\mu$ and 
\begin{align*}
     B & = \omr\{f, \mu f+(1+\mu)e, \mu^2f+(1+\mu^2)e\} \} \\
       & = \omr\{\lambda f + (1-\lambda )e \mid \lambda \in \mathbb{F}_4^* \}=B_5.
\end{align*}

Assume $|B|=p^k$ with $k<a.$ 
 By the arguments of Case 2, $(q,r,\ell)=(4,3,1)$ or $(16,5,2).$ Moreover 
 $$B= \omr\{f + b e \mid b \in U\}$$
where $U \subset \{b\in \mathbb{F}_{q^2} \mid  \mathrm{Tr}(b)=0\}$ such that $U$ is a proper additive subgroup of $\{b\in \mathbb{F}_{q^2} \mid \mathrm{Tr}(b)=0\}$ invariant under  multiplication by $\omega^{-r(q+1)}.$ Since $q$ is even, $\{b\in \mathbb{F}_{q^2} \mid \mathrm{Tr}(b)=0\}=\mathbb{F}_q.$
Recall  from Case 2 that $|U|=|B|=p^\ell$.
Since $G=\GammaU_3(q)$ in both cases, $\phi \in (G_{\alpha})_{\omr f}$, so $U$ must be invariant under the field automorphism $\phi$.
 If $(q,r)=(4,3)$, it immediately follows that $U=\mathbb{F}_{2}$ and $B=B_6$.
Assume $(q,r)=(16,5)$ so that $k=2$.  Recall that $U$ is invariant under  multiplication by $\omega^{-85}$ and note that $\omr[85] =\mathbb{F}_4^*$. So $U$ is invariant under multiplication by an element of $\mathbb{F}_4$. Let $v\in U^*$ and $\lambda\in \mathbb{F}_4^*$ with $\lambda\neq 1$. Then $U=\{0, v,\lambda v,\lambda^2 v \}$. Since $U$ is invariant under the field automorphism $\phi$, we get that $v^2=\lambda^i v $ for some $i\in\{0,1,2\}$, and thus $v=\lambda^i\in \mathbb{F}_4$. It follows that $U=\mathbb{F}_4$ and $B=B_7$.

 It is easy to check by hand or using {\sc Magma} that $B_5,$ $B_6$ and $B_7$ are indeed blocks of imprimitivity for $G_{\alpha}$ acting on $\Omega\backslash \sigma$ and there are no further such blocks  (in particular, no block  as in Case 3 apart from $B_4$).  This finishes the proof of the theorem. 
\end{proof}

\section{Proof of the main theorem}\label{mainproof}

Let $G \leq \Sym(\Omega)$ have rank $3$. Recall from Theorems \ref{th:AD1} and \ref{th:AD2}, and the comments following these statements, that  a block $B$ of a nontrivial system of imprimitivity for the action of $G_{\alpha}$ on one of its orbits in $\Omega$ gives rise to a partial linear space if and only if the setwise stabiliser $G_{L}$ acts transitively on $L=\{\alpha\} \cup B$. The following simple consequence of this fact will be useful in our analysis for connected partial linear spaces.

\begin{Lem}\label{lem:int}
    Suppose that $G \leq \Sym(\Omega)$ is imprimitive of rank $3$, and that $\Sigma$ is the unique nontrivial $G$-invariant block system on $\Omega$. Let $\alpha\in\Omega$ and let $\sigma\in\Sigma$ be the block containing $\alpha$, so that the $G_\alpha$-orbits are $\{\alpha\},\ \sigma\setminus\{\alpha\},$ and $\Omega\setminus\sigma$. If $B$ is a block of imprimitivity for the $G_\alpha$-action on  $\Omega\setminus\sigma$ such that the setwise stabiliser $G_L$ of $L:=B\cup \{\alpha\}$ is transitive on $L$, then $|L\cap \sigma'|\leq 1$ for all $\sigma'\in\Sigma$.
\end{Lem}

\begin{proof}
    Note that the subgroup $G_L$ leaves $\Sigma$ invariant. Thus, since  $G_L$ is transitive on $L$, it follows that the non-empty intersections $L\cap \sigma'$, for $\sigma'\in\Sigma$, must have a constant size. Then, since $L\cap \sigma=\{\alpha\}$ by the definition of $L$ (since $B\subseteq \Omega\setminus\sigma$), this constant size is $1$.
\end{proof}

First we consider the case where $G$ is a linear group  as in Construction \ref{con:psl}.

\begin{table}[h]
\begin{tabular}{cccclc}
\hline
Case  & $n$  & $q$ & $r$  & $B$  & $G$ \\ 
\hline
(i) & $\geq 3$ & 3  & 2    &       $B_4= \{e_2, 2e_1 + 2 e_2\}$\   & $\SL_n(3), \GL_n(3)$  \\ 
  & $\geq 2$ & $4$ & 3  &  $B_4=  \{e_2, e_1 +  e_2\}$\    & $\SL_n(4) \leq G \leq \GammaL_n(4)$  \\   
    & $\geq 2$ & $4$ & 3  &  $B_5= \{(1+\lambda)e_1 +\lambda e_2 \mid \lambda \in \mathbb{F}_4^*\}$\    & $\SL_n(4) \leq G \leq \GammaL_n(4)$  \\ 
    (ii) & $\geq 2$  & $2^4$  & 5    &     $B_{6}= \omr \{\lambda e_1 +e_2 \mid \lambda \in \mathbb{F}_4\}$\  &$Y\,\SL_n(2^4)/Y \leq G \leq \GammaL_n(2^4)/Y$     \\ 
    (iii) & $2$  & $3^4$  & 5    &     $B_{7,1}= \omr \{\lambda e_1 +e_2 \mid \lambda \in \mathbb{F}_9\}$\quad or  &$(Z \, \SL_2(3^4)\rtimes \langle \phi \rangle)/Y$    \\ 
        &   &   &   &$B_{7,2}= \omr \{\lambda e_1 +e_2 \mid \lambda \in \omega^5\,\mathbb{F}_9\}$ \\
    (iv) & $2$  & $5^2$  & 3    &     $B_{8,1}= \omr \{\lambda e_1 +e_2 \mid \lambda \in \mathbb{F}_5\}$\quad or   &$(Z \, \SL_2(5^2)\rtimes \langle \phi \rangle)/Y$     \\ 
        &   &   &   &$B_{8,2}= \omr \{\lambda e_1 +e_2 \mid \lambda \in \omega^3\,\mathbb{F}_5\}$  \\
(v) & $2$ & $3^2$  & 2    &      $B_{9,0}= \omr \{\lambda e_1 +e_2 \mid \lambda \in \mathbb{F}_3\}$\quad or  &  $\langle  Y\,\SL_2(3^2),   \phi \, \diag(1,\omega)  \rangle/Y$ \\ 
        &   &   &   & $B_{9,2}= \omr \{\lambda e_1 +e_2 \mid \lambda \in \omega^2\mathbb{F}_3\}$ \\ \hline
\end{tabular}
\caption{Blocks for Lemma~\ref{lem:LtrSL}, where $\omega$ is a primitive element in $\mathbb{F}_q$ }\label{t:LtrSL}
\end{table}

\begin{Rem}\label{r:LtrSL}
In all lines of Table~\ref{t:LtrSL}, we require that the group $G$ has rank $3$ and so satisfies the conditions in Table \ref{t:qpinsprk3}. Here we make a few additional comments.

In case (i), lines 2 and 3 of Table~\ref{t:LtrSL},  if $n \geq 3$, then the rank $3$ groups $G$   are the following: 
$\GammaL_n(4)$, $\langle \SL_n(4), \phi \rangle$, and $\langle \SL_n(4), \diag(1, \ldots, 1, \omega) \phi \rangle$. If $n=2$, then $G=\GammaL_2(4)$ by Theorem \ref{th:blocksn2}. Further, in all lines for case (i) the group $Y=\{I\}$.

In case (ii), all groups in the range given in Table~\ref{t:LtrSL} have rank $3$, except that,  when $n=2$ then by Theorem \ref{th:blocksn2} we have $G=\GammaL_2(16)/Y$, so the group $G$ can be quasiprimitive only if $n\geq3$.

In cases (iii) and (iv) of Table~\ref{t:LtrSL}, the group $G$ is equal to  the group given in Table~\ref{t:LtrSL} (and is normal in $\GammaL_2(q)/Y$), and $G$ is properly innately transitive.

 In case (v) of Table~\ref{t:LtrSL},  the group $G$ is conjugate in $\GammaL_2(3^2)/Y$ to 
 $\langle  Y\,\SL_2(3^2),   \phi \, \diag(1,\omega)  \rangle/Y$ (see Remark \ref{rem:9iConj}). In particular, $G$ is quasiprimitive.

\end{Rem}

\begin{Lem}\label{lem:LtrSL}
   Let $G \leq \Sym(\Omega)$ satisfying $Y\,\SL_n(q)/Y \leq G \leq \GammaL_n(q)/Y$, where $\Omega$ and $Y$ are as in Construction \ref{con:psl}, and let $\alpha=\omr e_1$ and $\beta=\omr e_2$. Assume that $G$  has rank $3$, and that $B$ is a nontrivial block of imprimitivity for the action of $G_{\alpha}$ on $\Omega \backslash \sigma$ such that $\beta\in B$. Then, for $L :=\{\alpha\} \cup B$,  $G_L$ is transitive on $L$ if and only if one of the lines of Table~\ref{t:LtrSL} holds. 
\end{Lem}

\begin{proof}
Since the block $B$ is nontrivial, it must be one of those determined in Theorems \ref{th:blocks} or \ref{th:blocksn2}. In all cases $B$ consists of points $\omr v$ for various $v\in\langle e_1,e_2\rangle$, and it follows that $G_L$ leaves $\langle e_1,e_2\rangle$ invariant. Also, as $G_{\alpha,B}$ is transitive on $B$, the stabiliser $G_L$ is transitive on $L=B\cup\{\alpha\}$ if and only if there exists an element $gY \in G_L$ such that $(\omr e_1)gY=\omr e_2.$ 
For any element $gY$ with this property, $g\in\GammaL_n(q)$ has restriction to $\langle e_1,e_2\rangle$ of the form
 \begin{equation} \label{e:gY}
g \mid_{\langle e_1, e_2 \rangle} = \begin{pmatrix} 
        0 & \delta \\
        \lambda_1 & \lambda_2
    \end{pmatrix} \cdot \phi^k
 \end{equation}
    for some $\delta \in \omr,$ $\lambda_i \in \mathbb{F}_q$ with $\lambda_1\ne0$, and $k \in \{0, 1, \ldots, a-1\}.$ 
    We first use this observation together with Lemma~\ref{lem:int} to identify precisely the blocks occurring in Table~\ref{t:LtrSL} for further analysis.
    
 \medskip\noindent   
    {\it Claim: If $G_L$ is transitive on $L$, then $B$ is one of the blocks in the list below.
    \begin{enumerate}[label=\normalfont $(\roman*)$]
    \item $n \geq 3$, $(q,r) \in \{(3,2),(4,3)\}$, or $(n,q,r)=(2,4,3)$, and $B \in \{B_4, B_5\}$;\label{trblcase1}
    \item $n \geq 2$, $(q,r)=(2^4,5)$ and $B=B_6$;\label{trblcase11}
    \item $n=2$, $(q,r) = (3^4,5)$ and $B=B_{7,i}$ for $i \in \{1,2\}$; \label{trblcase2}
    \item $n=2$, $(q,r) = (5^2,3)$ and $B=B_{8,i}$ for $i \in \{1,2\}$; \label{trblcase21}
    \item $n=2$, $(q,r) = (3^2,2)$ and $B=B_{9,i}$ for $0\leq i \leq 3$. \label{trblcase3}
\end{enumerate}  
} 

\medskip\noindent 
 It follows from Theorems \ref{th:blocks} and \ref{th:blocksn2} that either $B$ is one of the blocks in the list above, or $B$ is one of $B_1, B_2$, or $B_3$, as defined in Theorems \ref{th:blocks} and \ref{th:blocksn2}. Further, if $G_L$ is transitive on $L$, then by Lemma~\ref{lem:int}, the non-empty intersections $L\cap \sigma'$, for $\sigma'\in\Sigma$, all have size $1$. Since this condition fails for $B_1$ and $B_3$, it remains to consider the case where $B=B_2$ and $G_L$ is transitive on $L$. Thus $G_L$ contains an element $gY$ with $g$ as in \eqref{e:gY}, and for this element $g$ we let $\mu :=(\omega - \lambda_2)\delta^{-1}\in\mathbb{F}_q$. By the definition of $B_2$ in Theorems~\ref{th:blocks} and \ref{th:blocksn2}, we have $ \omr (\mu e_1 + e_2)\in B_2$, and the image of this point under $gY$ must lie in $L$, that is,
$$
( \omr (\mu e_1 + e_2))gY = \omr (\lambda_1 ^{p^k}  e_1 + (\mu \delta +\lambda_2 )^{p^k} e_2 ) \in L. 
$$
This implies, again by the definition of $B_2$, that $(\mu \delta +\lambda_2 )^{p^k} \in \omr \cup \{0\}$, and hence also $\mu \delta +\lambda_2  \in \omr \cup \{0\}$. However $\mu \delta +\lambda_2=\omega\not\in\omr \cup\{0\}$, and we have a contradiction. Thus the Claim is proved.

\medskip
We now consider each of the cases (i)--(v) of the Claim separately.

\medskip \noindent
{Case \ref{trblcase1}:} \quad 
 { Assume first that $(q,r)=(3,2)$ with $n\geq3$ and $B=B_4$ or $B_5$},  as in Theorems \ref{th:blocks} and \ref{th:blocksn2}. In this case $\omr =\{1\}$ and $q=p^a$ with $a=1$, so $Y=\{I\}$ and $G \leq \GL_n(3),$ and we identify $B$ with the corresponding subset of $\langle e_1, e_2 \rangle$.  If $B=B_4=\{e_2, 2e_1+2e_2\}$, then there is  an element $g\in \SL_n(3) \leq G$, fixing $\langle e_1, e_2 \rangle$ setwise such that the restriction in \eqref{e:gY} is   
 $$g \mid_{\langle e_1, e_2 \rangle} =\begin{pmatrix}
        0 &1 \\
        1& 0
    \end{pmatrix}.$$
    This element satisfies $(e_i)g= e_{3-i}$ for $i=1,2$, and fixes $2e_1+2e_2$. Hence $g\in G_L$, so $G_L$ is transitive on $L$ and we have the example in line $1$ of Table~\ref{t:LtrSL} with $G$ either $\SL_n(3)$ or $\GL_n(3)$. Suppose now that $B=B_5=\{e_2, e_1+2e_2\}$. Here the element $g\in G_L$ in \eqref{e:gY} must satisfy $( e_1)g=  e_2$, and hence $\delta=1$. Also we require 
    $(e_2)g= \lambda_1 e_1 + \lambda_2e_2 \in L\setminus\{e_2\} = \{e_1, e_1+2e_2\}$, and hence $(\lambda_1,\lambda_2)=(1,0)$ or $(1,2)$. In either case $\lambda_1=1$ and hence $(e_1+2e_2)g= 2 e_1 + (1+2\lambda_2) e_2$, which does not lie in $L$, so $g\not\in G_L$ and we have a contradiction. Thus for $B=B_5$, the stabiliser $G_L$ is not transitive on $L$.

    Now assume that $(q,r)=(4,3)$  with $n\geq2$ and $B=B_4$ or $B_5$. Here $\omr =1$ and $a=2$, so $Y=\{I\}$ and $G \leq \GammaL_n(4),$ and again we identify $B$ with a subset of  $\langle e_1, e_2 \rangle$. Let us show that $G_L$ contains suitable elements $g_4, g_5$ as in \eqref{e:gY} for $B=B_4, B_5$, respectively. Note that the group induced by $G_{\langle e_1, e_2 \rangle}$  on $\langle e_1, e_2 \rangle$  contains  $Z \, \SL_2(4)= \GL_2(4)$ and consider the following elements $g_4, g_5\in Z\, \SL_n(4)$ such that 
    $$g_4 \mid_{\langle e_1, e_2 \rangle}  =\begin{pmatrix}
        0 &1 \\
        1& 0
    \end{pmatrix} \text{ and } g_5 \mid_{\langle e_1, e_2 \rangle} =\begin{pmatrix}
        0 &1 \\
        1+\omega & \omega
    \end{pmatrix}.$$
    It is straightforward to check that $g_4$ and $g_5$ lie in $G_{L}$ and satisfy the required condition for $B=B_4$ and $B_5$ respectively.  This justifies lines 2 and 3 of Case (i) in Table~\ref{t:LtrSL}.

\medskip \noindent
{Case \ref{trblcase11}:} \quad Here $n \geq 2$, $(q,r)=(2^4,5)$ and $B=B_6= \omr \{\lambda e_1 +e_2 \mid \lambda \in \mathbb{F}_4\}$. Note that $\omr=\omr[5]=\mathbb{F}_4^*=\mathbb{F}_4\setminus\{0\}.$  Since $G$ has rank $3$, we must have $G = \GammaL_2(2^4)/Y$ when $n=2$ (see Remark~\ref{r:LtrSL}) and hence $G$ contains an element $gY$ mapping $\alpha=\omr[5]e_1$ to $\beta=\omr[5] e_2$, as in \eqref{e:gY}, such that 
\[
g \mid_{\langle e_1, e_2 \rangle}  =\begin{pmatrix}
        0 &1 \\
        1& 0
    \end{pmatrix}.
    \]
Also  $gY$ maps $\beta$ to $\alpha$ and,  for $\lambda \in \mathbb{F}_4^*=\omr[5]$,
    $$
    (\omr[5](\lambda e_1 +e_2))gY = \omr[5]( e_1 + \lambda e_2)=\omr[5](\lambda^{-1} e_1 +e_2)\in B_6.
    $$
It follows that $gY$ leaves $L=\{\alpha\}\cup B_6$ invariant, and we conclude that $G_L$ is transitive on $L$, as in Table~\ref{t:LtrSL}.

\medskip \noindent
{Cases  \ref{trblcase2}, \ref{trblcase21}:} \quad Here $n=2$, and $(q,r) = (3^4,5)$ or $(5^2,3)$. Let $q_0=q^{1/2}$ and note that $\mathbb{F}_{q_0}^*=\omr[2r] < \omr$ in each case.  Let $s$ be $7$ or $8$, in case  \ref{trblcase2} or \ref{trblcase21}, respectively.
Then $B=B_{s,i}= \omr \{\lambda e_1 +e_2 \mid \lambda \in a_{s,i}\,\mathbb{F}_{q_0}\}$, for some $i\in\{1,2\}$, where $a_{s,1}=1$ and $a_{s,2}=\omega^r$ in $\mathbb{F}_{q}$.  Note that $a_{s,i}\,\mathbb{F}_{q_0}\subseteq \{0\}\cup  \omr$ in both cases.   By Theorem~\ref{th:blocksn2}, we have $G= (Z \, \SL_2(q) \rtimes \langle \phi \rangle)/Y$ in each case. Thus $G$ contains the element $gY$ with 
\[
g =\begin{pmatrix}
        0 &1 \\
        1& 0
    \end{pmatrix}=(\mu^{-1}I)\cdot\begin{pmatrix}
        0 &\mu \\
        \mu& 0
    \end{pmatrix} \textrm{ where } \mu=\omega^{\frac{q-1}4} \textrm{ so that } \mu^2=-1
\]
(over the appropriate field), and this element interchanges $\alpha$ and $\beta$. 

Let  $\lambda \in a_{s,i}\,\mathbb{F}_{q_0}^*$. Then  $\lambda\in\omr$. When $i=1$, clearly $\lambda^{-1}\in \mathbb{F}_{q_0}^*=a_{s,i}\,\mathbb{F}_{q_0}^*$. When $i=2$, $\lambda=\omega^r \nu$ for some $\nu\in \mathbb{F}_{q_0}^*,$ and so  $\lambda^{-1}=\omega^{-r} \nu^{-1}=\omega^r\omega^{-2r} \nu^{-1}$ and $\omega^{-2r} \nu^{-1}\in\mathbb{F}_{q_0}^*=\omr[2r].$ Thus in both cases $\lambda^{-1}\in a_{s,i}\,\mathbb{F}_{q_0}^*\subseteq \omr$, and therefore, 
$$
 (\omr[r](\lambda e_1 +e_2))gY= \omr[r]( e_1 + \lambda e_2)
 =\omr[r](\lambda^{-1} e_1 +e_2) \in B_{s,i}.
    $$
It follows that $gY$ leaves $L=\{\alpha\}\cup B$ invariant, and hence $G_L$ is transitive on $L$, as in Table~\ref{t:LtrSL}.

 \medskip \noindent
{Case   \ref{trblcase3}:} \quad   Here $n=2$ and $(q,r) = (3^2,2)$, and the block  is 
 $B=B_{9,i}= \omr \{\lambda e_1 +e_2 \mid \lambda \in \omega^i\,\mathbb{F}_{3}\}$, for some $i\in\{0,1,2,3\}$. Note that $\mathbb{F}_3^*=\omr[4]=\omr[2r] < \omr$.   Note also that these blocks occur only in  Case 3 of the proof of Theorem~\ref{th:blocksn2}, where the group $G$ is quasiprimitive,  and $G=(Y \, \SL_2(3^2) \rtimes \langle \diag(1,\omega) \phi \rangle)/Y$ by Theorem ~\ref{th:blocksn2} (up to conjugation). Thus 
$$g=\begin{pmatrix}
    0 & 1 \\
    1 & 0
\end{pmatrix}=(\omega^{-2}I)\cdot \begin{pmatrix}
     0&\omega^2 \\
     \omega^2&0
\end{pmatrix}\in Y\,\SL_2(9)$$ since $\omega^4=-1.$
For each value of $i$, this element interchanges the points $\alpha=\omr e_1$ and $\beta=\omr e_2$ of $L$.  Assume first that $i \in \{0,2\}$. For $\lambda \in \omega^i \mathbb{F}_3^*=\omega^i\omr[4]$, we see that $\lambda\in\omr=\omr[2]$ and that $\lambda^{-1}\in \omega^i\omr[4]$. 
Thus, for  $\lambda \in \omega^i \mathbb{F}_3^*$,
$$(\omr[2](\lambda e_1 + e_2))gY = \omr[2](e_1 + \lambda e_2)= \omr[2](\lambda^{-1} e_1 + e_2) \in B_{9, i}.
$$
It follows that $gY \in G_{L},$ and hence $G_L$ is transitive on $L$, as in Table~\ref{t:LtrSL}.
Finally assume that $i \in \{1,3\}$. We claim that in these cases the stabiliser $G_L$ is not transitive on $L$, equivalently, there is no element $gY\in G_L$ as in \eqref{e:gY}. Suppose to the contrary that $gY$ is such an element. Since $G$ contains $Y\,\SL_2(9)/Y$, we may assume that the entry $\delta$ as in \eqref{e:gY} is $1$, and as noted above $gY$ maps $\alpha=\omr e_1$ to $\beta=\omr e_2$. Consider $\gamma=(\omr[2](\lambda e_1 +e_2))\in B_{9,i} = L\setminus\{\alpha\}$. Then 
\[
(\omr[2](\lambda e_1 +e_2))gY =\omr[2](\lambda_1^{\phi^k}e_1+ (\lambda+\lambda_2)^{\phi^k} e_2) \in  L\setminus\{\beta\}.
\]
In particular, this image must equal $\alpha$ for one of the three points $\gamma$, and this holds precisely when $\lambda=-\lambda_2$ and $\lambda_1^{\phi^k}\in\omr[2]$. This implies that 
$\lambda_1\in\omr[2]$. For the other two points $\gamma$ the quantity $ (\lambda+\lambda_2)^{\phi^k}\ne 0$ and 
we have
\[
(\gamma ) gY =  \omr[2]((\lambda+\lambda_2)^{\phi^k})((\lambda_1(\lambda+\lambda_2)^{-1})^{\phi^k}e_1+  e_2) \in B_{9,i}.
\]
Considering the coefficient of $e_2$ 
we conclude that $(\lambda+\lambda_2)^{\phi^k}\in\omr[2]$, but then the coefficient $(\lambda_1(\lambda+\lambda_2)^{-1})^{\phi^k}$ of $e_1$ also lies in $\omr[2]$, contradicting the definition of $B_{9,i}$. Thus no such element $gY$ exists and 
$G_L$ is not transitive on $L$. This completes the proof.
\end{proof}

Let now $Z \, \SU_3(q)/Y \leq G \leq \GammaU_3(q)/Y$, $G \leq \Sym(\Omega)$ where $\Omega$ and $Y$ are as in Construction \ref{con:psu}.

\begin{Lem}\label{lem:LtrSU}
     Let  $Z \, \SU_3(q)/Y \leq G \leq \GammaU_3(q)/Y$ with $G^{\Omega}$ of rank 3 and let $\alpha=\omr e$. Let $B$ be a block containing $\langle \omega^r \rangle f$ in a system of imprimitivity of $G_{\alpha}$ on $\Omega \backslash \sigma$. Then $G_L$ is transitive on $L=\{\alpha\} \cup B$ if and only if $G=\GammaU_3(q)/Y$ and one of the following holds:
      \begin{enumerate}[label=\normalfont (\arabic*)]
       \item $q=4,$ $r=3$,  $B=B_5$ as in Table \ref{t:blocksSU};
       \item  $q=4,$ $r=3$,  $B=B_6$ as in Table \ref{t:blocksSU};
       \item $q=16,$ $r=5$, $B=B_7$ as in Table \ref{t:blocksSU}.
       \end{enumerate}
\end{Lem}
\begin{proof}
Since the block $B$ is nontrivial, it must be one of those determined in Theorem \ref{th:blocksSU}.

\medskip 

First we show that   if $B$ is  $B_1$ or $B_2$ in Table \ref{t:blocksSU}, then $G_L$ is not transitive on $L$. 
Note that here  $B_2\subseteq B_1$ and all elements in $B$ are of the form $\omr (f+v)$ with $v\in \langle x,e\rangle$.

Assume that $G_L$ is transitive on $L$, so there exists $\tilde{g}Y \in G_L$ such that $(\omr e)\tilde{g}=\omr f.$ Note that $\tilde{g}=g \phi^k$ where $g\in (Z \, \GU_3(q))_{L}$ and $k \in \{0, \ldots, 2a-1\}$ since $\phi$ (and hence $\phi^k$) stabilises $L$.
Observe that  $$(\omr e)g=(\omr e)\tilde{g}\phi^{-k}=(\omr f)\phi^{-k}=\omr f.$$ So $g$ in $(Z \, \GU_3(q))_L$ is such that $(\omr e)g=\omr f$.  
Hence $g$ must have the form 
\begin{equation}\label{eq:etof}
\lambda I \cdot 
\begin{pmatrix}
    0 & 0 & \tilde{z} \\
    0 & \tilde{y} & \tilde{d} \\
    \tilde{z}^{-q}&\tilde{c} & \tilde{b}  
\end{pmatrix} \end{equation}
where $\lambda, \tilde{b},\tilde{c}, \tilde{d},\tilde{z}, \tilde{y} \in \mathbb{F}_{q^2}$ with $\lambda \tilde{z} \in \omr,$ $\tilde{y}^{q+1}=1$, 
$\mathrm{Tr}(\tilde{b}\tilde{z}^{-1}) + \tilde{c}^{q+1}=0,$ and $\tilde{d}\tilde{z}^{-1}+\tilde{y}\tilde{c}^q=0.$ 
Since $\lambda \tilde{z} \in \omr,$ there exists $\lambda_1 \in \omr$ such that $\lambda_1 \lambda \tilde{z}=1$. If $g$ has the shape \eqref{eq:etof}, then $\lambda_1I\cdot g$ acts on $\Omega$ in the same way that $g$ does since $\lambda_1 I$ lies in the kernel of the action $Y=\langle \omega^r I \rangle$.
Thus we may assume that $g$ has the form 
\begin{equation*}
\begin{pmatrix}
    0 & 0 & 1 \\
    0 & y & d \\
    y^{q+1}&c & b  
\end{pmatrix} \end{equation*}
where $ b,c, d, y \in \mathbb{F}_{q^2}$ with $y\neq 0$,  
$\mathrm{Tr}(b) + (cy^{-1})^{q+1}=0,$ and $d+(cy^{-1})^q=0.$

Now there must be a point of $B$ that $g$ maps to $\omr e$, say $\omr (f+\lambda x+\mu e)$ where $\mathrm{Tr}(\mu)+\lambda^{q+1}=0$ (if $B=B_2$, then $\lambda=0$).
Thus 
\[ \omr e= \omr (f+\lambda x+\mu e)g= \omr \left((b+\lambda d+\mu)f+(c+\lambda y)x+y^{q+1}e\right).\]
It follows that $b+\lambda d+\mu=c+\lambda y=0$, and $y^{q+1}\in \omr$.
Let $U:= \mathrm{ker}(\mathrm{Tr})= \{b \in \mathbb{F}_{q^2} \mid \mathrm{Tr}(b)=0\},$ so $U$ is described in Lemma \ref{l:trace}(2).
    Note that $U \backslash \omr$ contains a non-zero element.  Indeed, since $r$ does not divide $q+1$, $\omega^{(q+1)} \in U \backslash \omr$ if $q$ is even and  $\omega^{(q+1)/2} \in U \backslash \omr$ if $q$ is odd. 
 Let $\nu\in U \backslash \omr$ with $\nu\neq 0$. 
Observe that $$\mathrm{Tr}(\mu+\nu)+\lambda^{q+1}=\mathrm{Tr}(\mu)+\mathrm{Tr}(\nu)+\lambda^{q+1}=0,$$ so $\omr(f + \lambda x + (\mu+\nu) e)\in B$. Indeed, this is obvious if $B=B_1$ and it follows from the fact that $\lambda=0$ if $B=B_2$. Thus
\begin{align*}
    \omr(f + \lambda x + (\mu+\nu)  e)g & = \omr \left((b+\lambda d+\mu+\nu)f+(c+\lambda y)x+y^{q+1}e\right) \\ & = \omr \left(\nu f+y^{q+1}e\right) \in B
\end{align*}
which  is a contradiction since  $\nu\notin \omr$ and all elements in $B$ are of the form $\omr (f+v)$ with $v\in \langle x,e\rangle$. Hence $G_L$ is not transitive on $L=\{\alpha \} \cup B$ for $B \in \{B_1, B_2\}.$

\medskip

 Further, if $G_L$ is transitive on $L$, then by Lemma~\ref{lem:int}, the non-empty intersections $L\cap \sigma'$, for $\sigma'\in\Sigma$, all have size $1$. If $B$ is $B_3$ or $B_4$, then this condition fails and $G_L$ is not transitive on $L$.

\medskip

 Finally, let us show that  if  $B$ is one of the $B_5$ -- $B_7$, then $G_L$ is transitive on $L$.  It is shown in the proof of Theorem \ref{th:blocksSU} after {\bf Case 3} that if $(q,r)$ is $(4,3)$ or $(16,5),$ then $G=\GammaU_3(q)/Y.$
Since $G_B$ is transitive on $B$, it suffices to find $gY \in G_L$ such that $(\omr e)gY=\omr f$.
    Let $gY$ be 
$$
\begin{pmatrix}
    0 & 0 & 1 \\
    0 & 1 & 0 \\
    1 & 0 & 0 \\
\end{pmatrix}Y \in Z \,  \SU_3(q)/Y, 
$$
so $(\omr e)gY=\omr f$.
It is easy to see  that, if $B$ is either $B_5$ or $B_6$, then $gY \in G_L$ by the symmetry of $L.$
Let $B=B_7$, so $(q,r)=(16,5)$ and $\langle \omega \rangle = \mathbb{F}_{{16}^2}^*$.  We claim that $gY \in G_L$. Indeed, if $\lambda \in \mathbb{F}_4^*$, then   
$$\omr(f+ \lambda e)gY= \omr (e + \lambda f )=\omr(f + \lambda^{-1}e) \in B,$$
where the last equality holds since $\mathbb{F}_{4}^* \leq \omr.$
If $\lambda =0,$ then 
$\omr fgY= \omr e  \in B.$
\end{proof}

Now we are ready to prove the main result of the paper.

\begin{proof}[Proof of Theorem $\ref{th:main}$] 

Let $G\leq \Sym(\Omega)$ be a rank 3 group satisfying Hypothesis \ref{hyp1}, so $G$ affords a unique nontrivial system of imprimitivity $\Sigma$ on $\Omega$. Suppose that $\mathcal{D}=(\Omega, \mathcal{L})$ is a connected proper partial linear space such that $G \leq \Aut(\mathcal{D})$, and let    $L\in\mathcal{L}$, so $\mathcal{L}= \{(L)g \mid g \in G\}$. Let $\sigma\in\Sigma$ be such that $L\cap \sigma$ is nonempty and let $\alpha\in L\cap \sigma$ and $B:=L\setminus\{\alpha\}$. Since $\mathcal{D}$ is connected it follows that case (2) of Theorem~\ref{th:AD2} holds, and hence $B$ is a block of imprimitivity for the action of $G_\alpha$ on one of its orbits in $\Omega\setminus\{\alpha\}$, namely $\sigma\setminus \{\alpha\}$ or $\Omega\setminus\sigma$.  If $B$ were contained in  $\sigma\setminus \{\alpha\}$, then the line $L$ would be contained in $\sigma$ and hence each line would be contained in a block of $\Sigma$ and $\mathcal{D}$ would be disconnected. As this is not the case we must have $B\subset \Omega\setminus\sigma$. Also, by  Theorem~\ref{th:AD2}(2), the stabiliser $G_L$ is transitive on $L$. We complete the proof by considering the various possibilities for the group $G$, as given in Table \ref{t:qpinsprk3}.

\medskip\noindent
{\it Case $1$:  $G$  as in one of  the first two lines of Table $\ref{t:qpinsprk3}$.}\quad Here $G \leq \ff$ with $\ff$ as in Construction \ref{con:psl}, and $\Omega$ is as in Construction \ref{con:psl}.  We choose $\alpha=\omr e_1$ with $\alpha\in \sigma$ so, by  Lemma \ref{lem:LtrSL}, $B, G$ satisfy one of the Cases of Table~\ref{t:LtrSL}. We now consider each of these Cases.

\medskip\noindent
\emph{Case $1.1$.\ Lines $1$ and $2$ of Case $\mathrm{(i)}$ in Table \ref{t:LtrSL}.}\quad Here $B=B_4$  and $\Omega=\mathbb{F}_q^n\setminus\{0\}$ since $r=q-1$. We claim that  $\mathcal{L}=\{ \{u,v, -(u+v)\} \mid u,v \in V, \dim \langle u,v \rangle=2\}.$ If $n\geq 3$, then $G \geq \SL_n(q)$ and acts transitively on linearly independent pairs from $V= \mathbb{F}_q^n$;  this transitivity also holds if $n=2$ since the conditions in Table \ref{t:qpinsprk3} force the group $G$ to contain $\GL_2(q)$. This proves the claim, and hence $\mathcal{D}= \Delta(n,q)$ with $q\in \{3,4\}$ (see Definition \ref{Deltadef}), and is a partial linear space by Lemma~\ref{l:AGDel}, as in lines 2 or 3 of Table~\ref{tabLU}. 

\medskip\noindent
\emph{Case $1.2$.\ Line $3$ of Case $\mathrm{(i)}$ in Table \ref{t:LtrSL}.}\quad Here $B=B_5$,  $\Omega=\mathbb{F}_4^n\setminus\{0\}$ since $r=q-1=3$, and $\mathcal{L}=\{\{\lambda_1 u+ \lambda_2 v \mid \lambda_i \in \mathbb{F}_4, \lambda_1 + \lambda_2=1\} \mid u,v \in V, \dim\langle u,v \rangle=2 \}$.  Thus $\mathcal{D}=\mathrm{AG}^*(n,4)$ (see Definition \ref{AGdef}), and is a partial linear space by Lemma~\ref{l:AGDel}, as in line 1 of Table~\ref{tabLU}.

\medskip\noindent
\emph{Case $1.3$.\ Case $\mathrm{(ii)}$ in Table \ref{t:LtrSL}.}\quad Here $L=\{\alpha\}\cup B_6$, and we claim that $\mathcal{L}=\{(L)g \mid g \in G\}$ is equal to $\{(L)g \mid g \in \GL_n(16)/Y\}$. Note first that the field automorphism $\phi$ leaves $L$ invariant, so $\mathcal{L}=\{(L)g \mid g \in G \cap(\GL_n(16)/Y)\}$. If $n\geq 3$ then $G$ contains $Y\,\SL_n(q)/Y$, and $\SL_n(q)$ is transitive on linearly independent pairs from $V$, while if $n=2$ then $G \geq Z \, \SL_2(16)/Y=\GL_2(16)/Y$ (see Remark \ref{r:LtrSL}) which also has this transitivity property. The claim now follows from these comments. Thus, by Definition \ref{con:subfield} and the form of $B_6$ in Table \ref{t:LtrSL}, we have $\mathcal{D}=\mathrm{LSub}(n,16,4,5)$ (with $k=t=1$ in Definition~\ref{con:subfield}). Finally,  $\mathrm{LSub}(n,16,4,5)$ is a partial linear space by Theorem~\ref{th:LSub}, as in line 4 of Table~\ref{tabLU}. 

\medskip\noindent
\emph{Case $1.4$.\ Cases $\mathrm{(iii)}$ or $\mathrm{(iv)}$ in Table \ref{t:LtrSL}.}\quad Here  $B = B_{s,i}$ for some $i\in\{1,2\}$, with $s=7$ or $s=8$, respectively. Also $n=2$ and we note that  the field automorphism $\phi$ leaves $L=\{\alpha\}\cup B$ invariant in all cases, so $\mathcal{L}=\{(L)g \mid g \in G \cap(\GL_2(q)/Y)\}$. Now $G \cap(\GL_2(q)/Y) = Z\,\SL_2(q)/Y$ (see Table~\ref{t:LtrSL}), and it follows that $\mathcal{L}= \{(L)hY \mid hY \in \GL_2(q)/Y, \det(h) \in \omr[2]\}$. It is easy to see that if $i=1$, then $\mathcal{D}=\mathrm{LSub}(2,q,q_0,r)$ with $q=q_0^2$ (as in Definition \ref{con:subfield} with $k=2$ and $t=2$).
On the other hand, if  $i=2$, then the action of the matrix $\diag(\omega^r,1)$ on $\Omega$ induces an isomorphism from $\mathcal{D}$ to $\mathrm{LSub}(2,q,q_0,r)$. (This is easy to verify by hand or computationally using {\sc Magma}.) Thus we have just one possibility for $\mathcal{D}$ up to isomorphism in each of Cases $\mathrm{(iii)}$ and $\mathrm{(iv)}$, and by Theorem~\ref{th:LSub}, $\mathcal{D}$ is a partial linear space as in lines 5 and 6 of Table~\ref{tabLU}.

\medskip\noindent
\emph{Case $1.5$.\ Case $\mathrm{(v)}$ in Table \ref{t:LtrSL}.}\quad Here $n=r=2$ and $B = B_{9,i}$ for some $i\in\{ 0,2\}$.  Recall from Table~\ref{t:blocksn2} that $G= \langle Y \, \SL_2(9), \phi \, \diag(1,\omega) \rangle /Y$. Note that $\phi\, \diag(1,\omega)$ normalises $Y \, \SL_2(9)$ and $(\phi\, \diag(1,\omega))^2=\diag(1, \omega^4) \in Y \, \SL_2(9)$, so each element of $\langle Y \, \SL_2(9), \phi \, \diag(1,\omega) \rangle$ can be written as $(\phi\, \diag(1,\omega))^s h$ with $s \in \{0,1\}$ and $h \in \SL_2(9).$  Since the field automorphism $\phi$ leaves $L=\{\alpha\}\cup B$ invariant, we obtain  $\mathcal{L}= \{(L)g \mid g \in  Y\,\SL_2(9){/Y}\} \cup \{(L)\diag(1, \omega)g 
 \mid g \in  Y\,\SL_2(9){/Y} \}.$ Note that $\diag(\omega, \omega^{-1}) \in \SL_2(9)$, so $\{(L)\diag(1, \omega)g 
 \mid g \in  Y\,\SL_2(9){/Y} \} = \{(L)\diag(\omega, 1)g 
 \mid g \in  Y\,\SL_2(9){/Y} \} .$ Now, by Definition \ref{DLSdef}, if $i=0$ 
then 
 $\mathcal{D}=\mathrm{DLSub}(9,3,2,1)$, while if $i=2$ then the action of the matrix $\diag(\omega^2,1)$ on $\Omega$ induces an isomorphism from $\mathcal{D}$ to  $\mathrm{DLSub}(9,3,2,1)$. 
(It is easy to verify this by hand or computationally using {\sc Magma}.) Thus 
 $\mathcal{D}\cong \mathrm{DLSub}(9,3,2,1)$,
and by Lemma~\ref{lem:DLS}, $\mathcal{D}$ is a partial linear space as in line 7 of Table~\ref{tabLU}.

\medskip\noindent
In all the sub-cases of Case 1, the possibilities for the group $G$ given in Table~\ref{tabLU} follow from Lemma~\ref{lem:LtrSL} and Remark~\ref{r:LtrSL}, and the 
 assertions about the group $G$ given in the last column of Table~\ref{tabLU} follow from Lemma~\ref{l:semi}, see also Remark~\ref{qpitsp}.

\medskip\noindent

    \medskip

    {\it Case 2}. Let $G$ be as in the third line of Table \ref{t:qpinsprk3}, so $G \leq \ff$ where $\ff$ is as in Construction \ref{con:psu}. In particular, $\Omega$ is as in Construction \ref{con:psu} and we fix $\alpha$ to be $\omr e$. By  Lemma \ref{lem:LtrSU}, we may assume that one of the following holds:
    \begin{enumerate}
    \item[$(a)$]  $q=4$, $r=3$, $L= B_5 \cup \{\alpha\}= \omr\{\lambda_1 e+ \lambda_2 f \mid \lambda_i \in \mathbb{F}_q, \lambda_1+ \lambda_2=1\} $;
    \item[$(b)$]  $q=4$, $r=3$, $L= B_6 \cup \{\alpha\}= \omr\{e,f,e+f\} $ ;
    \item[$(c)$]  $q=16,$ $r=5$,   $L=B_7\cup \{\alpha\}= \{\omr e\} \cup \{\omr (\lambda e+ f) \mid \lambda \in \mathbb{F}_4\}$;
\end{enumerate}
and $G=\ff=\GammaU_3(q)/Y$ in all three cases above. Notice that an element $g \in \GammaU_3(q)$ can always be written as
$$
g=
\diag(1, \delta, 1) \cdot \phi^i \cdot \Tilde{g}
$$
with $\diag(1, \delta, 1) \in \GU_3(q)$, $0 \leq i \leq 2a-1$ and $\Tilde{g} \in Z \, \SU_3(q).$
    It is easy to see that in $(a)$--$(c)$ we have $L$ stabilised by $\phi$ and $\diag(1, \delta, 1)$ for all $\delta \in \mathbb{F}_q^*.$ Therefore, in each of the cases, 
    $$
    \mathcal{L}=\{(L)g \mid g \in G\}=\{(L)g \mid g \in Z \, \SU_3(q)\}.
    $$
Hence $\mathcal{D}$ is $\mathrm{AGU^*}(4)$, $\mathrm{USub}(4,2,3)$ and $\mathrm{USub}(16,4,5)$ (see  Definitions \ref{con:AGSU} and \ref{con:subfieldSU}) in cases $(a)$, $(b)$ and $(c)$ respectively, and these are partial linear spaces by Theorems~\ref{t:AGUstar} and~\ref{t:USub}, as in lines 8--10 of Table \ref{tabLU}.

 \medskip

    {\it Case 3}. In the remaining cases, $G$ is as in one of lines 4--12 of Table~\ref{t:qpinsprk3}. In these cases we obtain the result computationally using {\sc Magma}, showing that one of the lines 1--10 of Table~\ref{tabPLS} holds. See Appendix for details.
\end{proof}

\section{Disconnected rank 3 partial linear spaces}\label{sec:7}

Systems of imprimitivity of $G_{\alpha}$ on $\sigma \backslash \{\alpha\}$ are very similar for linear groups (as in Construction  \ref{con:psl}) and for unitary groups (as in Construction \ref{con:psu}). Hence we formulate and prove corresponding results simultaneously for both families of groups of rank 3. For this purpose, let us denote the first basis vector in $V$ by $e$ instead of $e_1$ in the case of linear groups and let $\alpha=\omr e$ in both cases. Also let $b=a$ for linear groups and $b=2a$ for unitary groups.

\begin{Lem}\label{lem:bsigmaPSL}
      Let either $Y \,\SL_n(q) \leq G \leq \GammaL_n(q)$ or $Z \, \SU_3(q) \leq G \leq \GammaU_3(q)$ with $G^{\Omega}$ of rank $3$, where $\Omega$ is defined as in Constructions \ref{con:psl} or \ref{con:psu} respectively. Let $B$ be a nontrivial block of imprimitivity containing $\omr \omega e$ for the action of $G_{\alpha}$ on $\sigma \backslash \{\alpha\}$.   Then $r\geq5$ and there exists $k$ dividing $r-1$  with $1<k<r-1$ such that 
      $$B =\omr\{\omega e, \omega^{p^{jk}} e, \omega^{p^{2jk}}e, \ldots, \omega^{p^{j(r-1 -k)}}e\}.$$
\end{Lem}
\begin{proof}
   Since $B$ is a nontrivial block of imprimitivity for the transitive $G_\alpha$-action on $\sigma \backslash \{\alpha\}$, the cardinality $|B|$ divides $|\sigma \backslash \{\alpha\}|$ and satisfies $1<|B|<|\sigma \backslash \{\alpha\}|=r-1$. In particular $r\geq 5$ and $r-1$ is composite.
   Let ${\bf u}$ be $1$ in the linear case and $2$ in the unitary case, so $V$ is $\mathbb{F}_{q^{\bf u}}^n$ with $n=3$ for the unitary case. In the unitary case, there is a unitary form on $V$ and $e\in V$ is an isotropic vector.  Recall from Table \ref{t:qpinsprk3} that 
   $$
   |G^{\Sigma}/(G^{\Sigma} \cap \PGL_n(q^{\bf u}))|=b/j
   $$ 
   with $(j,r-1)=1.$ In particular, all elements of $G$ have the form $g \phi^{ji}$ with $g \in \GL_n(q^{\bf u})$ and $i \in \{1, \ldots, b/j-1\}$. Notice that if  $g \phi^{ji} \in G_{\alpha}$, then
   $$g= \begin{pmatrix}
       \delta & 0 \\
       c & A
   \end{pmatrix}$$ where $\delta \in \omr,$ $c^{\top} \in (\mathbb{F}_{q^{\bf u}})^{n-1}$ and $A \in \GL_{n-1}(q^{\bf u})$.  Therefore, $g$ stabilises $\sigma$ pointwise and $G_{\alpha}$ acts on $\sigma \backslash \{\alpha\}$ as $\langle \phi^{j} \rangle.$ Further, $\langle \phi^{j} \rangle$ acts on $\sigma \backslash \{\alpha\}$ as a  cycle of order $r-1$ since $(j,r-1)=1$ and $o_r(p)=r-1.$ Since each system of imprimitivity of an abelian regular permutation group
   is the set of orbits of a subgroup, there exists $k$ dividing $r-1$,  with  $1<k<r-1$, such that $B$ is the orbit of $\langle \phi^{jk} \rangle$ containing $\omr \omega e$, so 
   $$
   B= \omr\{\omega e, \omega^{p^{jk}} e, \omega^{p^{2jk}}e, \ldots, \omega^{p^{j(r-1 -k)}}e\}. \eqno{\qedhere}
   $$ 
\end{proof}

\begin{Lem}\label{lem:sigmatr}
    If $B$ is as in Lemma \ref{lem:bsigmaPSL}, then $G_L$ is not transitive on $L$, where $L=B\cup\{\alpha\}$.
\end{Lem}
\begin{proof}
    By Lemma \ref{lem:bsigmaPSL}, $r\geq5$ and
   $L=\omr\{e,\omega e, \omega^{p^{jk}} e, \omega^{p^{2jk}}e, \ldots, \omega^{p^{j(r-1 -k)}}e\}$. We denote the set $\{0,1,p^{jk}, p^{j2k}, \ldots,  p^{j(r-1 -k)} \}$ by $I,$ so $L=\{\omr \omega^s  e \mid s \in I\}.$ 
    Assume that $G_L$ is transitive on $L$. Since $G_{\alpha}$ is transitive on $B=L\setminus\{\alpha\}$, it follows that $G_L$ is 2-transitive on $L$. Thus there exists $g \in G_L$ interchanging $\omr e$ and $\omr \omega e$.
Since $g$ maps $\omr e$ to $\omr \omega e$, we have 
    $$g = \phi^{jm} \cdot \begin{pmatrix}
       \omega \delta & 0 \\
       c & A
   \end{pmatrix}Y$$
   with $m \in \{0,1, \ldots. b/j-1\},$ 
   $\delta \in \omr$, $c^{\top} \in \mathbb{F}_{q^{\bf u}}^{n-1}$ and $A \in \GL_{n-1}(q^{\bf u}).$

   Hence 
   $$ (\omr \omega^se)g=\omr (\omega^s)^{p^{jm}} \omega\delta e=\omr \omega^{sp^{jm}+1}e.$$
So, if $s=1$, then  $$(\omr \omega e)g=\omr \omega^{p^{jm}+1}e=\omr e,$$
and $\omega^{p^{jm}+1}\in \omr$.
Therefore  $$ (\omr \omega^se)g=\omr \omega^{sp^{jm}+1}e=\omr \omega^{s(p^{jm}+1)+1-s}e=\omr \omega^{1-s} e.$$
 It follows that
    $  L=(L)g  = \{\omr \omega^{1-s} e \mid s \in I\}.$
 In particular, 
 \begin{equation}\label{IDC}
 \{s \mod r \mid s \in I\}=\{(1-s) \mod r \mid s \in I\}.
 \end{equation}

We claim  that $\sum_{s \in I}s \equiv 0 \pmod r.$ Indeed, 
$$(\sum_{s \in I}s) \cdot p^{jk}=p^{jk}+p^{2jk}+ \ldots+p^{j(r-1 -k)}+ p^{j(r-1)}\equiv \sum_{s \in I}s \pmod r$$
since $p^{r-1}\equiv 1\pmod r$.

Therefore, combining the above with \eqref{IDC}, we obtain
$$0 \equiv \sum_{s \in I}s \equiv \sum_{s \in I} (1-s) \equiv |I| -\sum_{s \in I} s \equiv (r-1)/k+1  \pmod r,$$
which is a contradiction since $r>3$ and $k>1$, so $(r-1)/k+1 \not\equiv 0 \pmod r.$
\end{proof}

\begin{proof}[Proof of Theorem \ref{th:discon}]
    If $G$ is as in lines $1$--$3$ of Table \ref{t:qpinsprk3}, then $\mathcal{D}= (\Omega, (B \cup \{\alpha\})^G)$ is not a partial linear space by Theorem \ref{th:AD2} and Lemma \ref{lem:sigmatr}. The results for $G$ as in lines 4--12   are obtained computationally using {\sc Magma}. See Appendix for details.
\end{proof}

\section*{Appendix}

Here we describe our computations with {\sc Magma}. Let $G\leq \Sym(\Omega)$ be a group satisfying Hypothesis \ref{hyp1}. First, we address how to construct such a group in the correct permutation representation in  {\sc Magma}. For $G$ as in the last two lines of Table \ref{t:qpinsprk3}, the representations (including generators in {\sc Magma}) are listed in the ATLAS of Finite Group Representations \cite{AtlasV3}. For $G$  as in the first three lines of Table \ref{t:qpinsprk3}, let $S$ be $\SL_n(q)$, or $\SU_3(q)$, for the first two lines, or for the third line, respectively. Then $\langle \omega^r I  \rangle S/ \langle \omega^r I  \rangle$ is a normal subgroup of $G$. One can construct $\langle \omega^r I  \rangle S/ \langle \omega^r I  \rangle$ in {\sc Magma} using the following procedure:
\begin{enumerate}
    \item find the stabiliser $H$ in $S$ of a (totally isotropic) $1$-subspace of the natural module of $S$ via the command \texttt{ClassicalMaximals};
    \item find the unique normal subgroup $R$ of index $r$ of $H$ via the command  \texttt{NormalSubgroups};
    \item obtain $\langle \omega^r I  \rangle S/ \langle \omega^r I  \rangle$ as a permutation group  via the command \texttt{CosetAction(S,R)}.
\end{enumerate}
Now $G$ can be obtained as a subgroup of the normaliser $N$ of $\langle \omega^r I  \rangle S/ \langle \omega^r I  \rangle$ in $\Sym(\Omega)$ using \texttt{MaximalSubgroups} or  \texttt{NormalSubgroups} since in all cases, we are interested in, $G$ is either equal to $N$ or is an (often normal) subgroup of low (often prime) index.

For the rest of the groups $G$ in Table \ref{t:qpinsprk3}, $G$ can be easily constructed using \texttt{CosetAction(M,R)} (and then taking the normaliser in $\Sym(\Omega)$, if necessary) where $M$ is the plinth of $G$ (which can be obtained using generic commands in {\sc Magma}, and is isomorphic to \texttt{PSL(3,4)}) and $R$ is a subgroup of $M$ of index $|\Sigma|r$ that is found using the command \texttt{SubgroupClasses}.

Once the group $G$ is constructed as a permutation group on $\Omega$, the following procedure constructs the set $\texttt{Or}$ of orbits of a point stabiliser, the set \texttt{Par3} of all its blocks of imprimitivity on the third orbit ($\Omega \backslash \sigma$) and, if a block gives rise to a partial linear space (via Theorem \ref{th:AD1}), it is recorded in the list \texttt{D3}. Of course, the same procedure can be applied to the second orbit, instead of the third, to construct possible disconnected partial linear spaces as well.

\lstset{
basicstyle=\ttfamily}  
{\small
\begin{lstlisting}
d:=Degree(G); 
vst:=Stabiliser(G,1);
Or:=Orbits(vst);

Or3G:=ActionImage(vst,Or[3]); 
Par3:=AllPartitions(Or3G); 
D3:=[];
for i in [1..#Par3] do
ExtractRep(~Par3, ~B);
Include(~B,1);
GBst:=Stabiliser(G,B);
BGset:=GSet(GBst,B); 
ActGB:=ActionImage(GBst,BGset);
if IsTransitive(ActGB) then 
L:=Orbit(G,B);
Append(~D3, NearLinearSpace<d|L>);
end if;
end for;
\end{lstlisting}}

\bibliographystyle{abbrv}
\bibliography{PLSRev.bib}

\end{document}